\pdfoutput=1

\documentclass[12pt]{amsart}
\usepackage[margin=1in]{geometry} 

\usepackage{amsmath,amssymb,amsthm,bm,colonequals,graphicx,mathrsfs,mathtools,microtype,multicol,stmaryrd}
\usepackage[shortlabels]{enumitem}
\usepackage[colorinlistoftodos]{todonotes}

\numberwithin{equation}{section}

\theoremstyle{plain}

\newtheorem{theorem}{Theorem}[section] 

\newtheorem{lemma}[theorem]{Lemma} 

\newtheorem{proposition}[theorem]{Proposition} 

\newtheorem{proposition-definition}[theorem]{Proposition-Definition} 

\newtheorem{corollary}[theorem]{Corollary} 

\theoremstyle{definition}

\newtheorem{definition}[theorem]{Definition} 

\newtheorem{example}[theorem]{Example} 

\theoremstyle{remark}

\newtheorem{remark}[theorem]{Remark} 

\newcommand{\Aff}{\mathbb{A}}
\newcommand{\wt}{\widetilde}

\newcommand{\CC}{\mathbb{C}}
\newcommand{\EE}{\mathbb{E}}
\newcommand{\FF}{\mathbb{F}}
\newcommand{\GG}{\mathbb{G}}
\newcommand{\PP}{\mathbb{P}}
\newcommand{\QQ}{\mathbb{Q}}
\newcommand{\RR}{\mathbb{R}}
\newcommand{\ZZ}{\mathbb{Z}}

\newcommand{\sgn}{\operatorname{sgn}}

\newcommand{\rank}{\operatorname{rank}}

\newcommand{\ord}{\operatorname{ord}}

\newcommand{\vol}{\operatorname{vol}}

\newcommand{\abs}[1]{\lvert #1 \rvert}
\newcommand{\card}[1]{\lvert #1 \rvert}
\newcommand{\norm}[1]{\lVert #1 \rVert}

\newcommand{\eps}{\epsilon}

\newcommand{\Frac}{\operatorname{Frac}}

\newcommand{\Hom}{\operatorname{Hom}}

\newcommand{\Spec}{\operatorname{Spec}}
\newcommand{\Supp}{\operatorname{Supp}}

\newcommand{\Gal}{\operatorname{Gal}}
\newcommand{\Pic}{\operatorname{Pic}}

\newcommand{\map}{\operatorname}
\newcommand{\ol}{\overline}

\newcommand{\inmat}[1]{\left[\begin{smallmatrix} #1 \end{smallmatrix}\right]}

\newcommand{\GL}{\operatorname{GL}}

\newcommand{\SL}{\operatorname{SL}}

\newcommand{\defeq}{\colonequals}

\newcommand{\maps}{\colon}

\newcommand{\belongs}{\subseteq}
\newcommand{\contains}{\supseteq}

\newcommand{\set}[1]{\{#1\}}

\DeclareMathOperator*{\argmax}{arg\,max}

\newcommand{\Rat}{\operatorname{Rat}}
\renewcommand{\div}{\map{div}}

\newcommand{\ucord}{u_1}
\newcommand{\vcord}{u_2}

\newcommand{\iboundary}{\mathfrak{i}}
\newcommand{\jboundary}{\mathfrak{j}}
\newcommand{\kboundary}{\mathfrak{k}}
\newcommand{\lboundary}{\mathfrak{l}}

\newcommand{\ispecial}{k}
\newcommand{\jspecial}{\ell}
\newcommand{\kspecial}{\mathfrak{K}}
\newcommand{\lspecial}{\mathfrak{L}}

\newcommand{\phase}{\mathsf{e}}
\newcommand{\efield}{E}
\newcommand{\adele}{\mathbf{A}}
\newcommand{\lbundle}{\mathcal{L}}
\newcommand{\oset}{\mathfrak{o}}
\newcommand{\openint}{\operatorname{Int}}
\newcommand{\Top}{\mathcal{T}}

\newcommand{\Kbad}{\mathbf{K}}
\newcommand{\Sbad}{\mathsf{S}}
\newcommand{\rbad}{\mathsf{r}}
\newcommand{\Nbad}{\mathsf{N}}
\newcommand{\Mbad}{\mathbf{M}}

\newcommand{\rord}{e_1}
\newcommand{\sord}{e_2}
\newcommand{\unitone}{\mathsf{w}_1}
\newcommand{\unittwo}{\mathsf{w}_2}
\newcommand{\unitthree}{\mathsf{w}_3}
\newcommand{\tloc}{\mathsf{t}}
\newcommand{\uloc}{\upsilon}
\newcommand{\redpur}{\rho_{\ZZ_p^{\operatorname{ur}}}}
\newcommand{\redp}{\rho_{\ZZ_p}}
\newcommand{\redR}{\rho_R}

\newcommand{\cfinal}{\mathsf{c}}
\newcommand{\Cfinal}{\mathsf{C}}

\usepackage[pagebackref=true]{hyperref}
\usepackage[alphabetic,initials]{amsrefs}

\title{Asymptotic growth of translation-dilation orbits}
\author{Victor Y. Wang}
\date{}
\address{Courant Institute, 251 Mercer Street, New York 10012, USA}
\address{IST Austria, Am Campus 1, 3400 Klosterneuburg, Austria}
\email{vywang@alum.mit.edu}

\subjclass{Primary 14G05; Secondary 11F72, 11G25, 11G50, 14M27}
\keywords{Manin conjectures, automorphic functions, special divisors, biases, cancellation}

\begin{document}

\begin{abstract}
By studying some Clausen-like multiple Dirichlet series,
we complete the proof of Manin's conjecture for sufficiently split smooth equivariant compactifications of the translation-dilation group over the rationals.
Secondary terms remain elusive in general.
\end{abstract}

\maketitle

\vspace*{-1.5em}

\setcounter{tocdepth}{1}
{\scriptsize
\begin{multicols}{2}
\tableofcontents
\end{multicols}
}
\setcounter{tocdepth}{3}

\section{Introduction}
\label{SEC:intro}

Even for rational projective varieties over $\QQ$
(e.g.~the Fermat cubic surface \cite{elkiescomplete}),
Manin's conjecture (see \cite{franke1989rational}) remains extremely difficult in general,
due to the arithmetic complexity of height functions governing point counts.
Hope increases
in the presence of symmetry or other favorable structure.
The present paper concerns one-sided \emph{equivariant compactifications} of a non-abelian algebraic group $G$;
these are defined to be projective $G$-varieties $\mathcal{Y}$ equipped with a $G$-equivariant open immersion $G\to \mathcal{Y}$ of dense image.

Cases of \emph{two-sided} compactifications
(e.g.~\cites{shalika2015height,shalika2007rational,gorodnik2008manin}),
and especially cases where the group in question is \emph{abelian}
(e.g.~\cites{batyrev1995manin,Breteche,chambert2002distribution}),
have already received fairly comprehensive treatments,
thanks to the simpler (yet deep) representation-theoretic quantities involved.
This is explained in \cites{tanimoto2012height,takloo2016distribution}.
For us, a certain infinite-dimensional representation (see \cite{tanimoto2012height}*{Proposition~3.1}; cf.~\cite{folland2016course}*{\S6.7, \S7.6; \cite{he2020representation}} over $\RR$) enjoys a subtle significance first explored, but not fully isolated, in \cite{tanimoto2012height}.

We build on \cite{tanimoto2012height}.
From now on, let $G = \set{\inmat{a & b \\ 0 & 1}} \belongs \GL_2$ be the $ax+b$ group,
viewed as an algebraic group over $\QQ$.
Explicitly, the group law on $(a,b), (\ucord,\vcord) \in G$ is
\begin{equation}
\label{EQN:group-law-on-G}
(a,b) \cdot (\ucord,\vcord) = (a\ucord, a\vcord+b).
\end{equation}
Let $X$ be a smooth, projective, right-sided equivariant compactification of $G$ over $\QQ$; so $G$ acts on $X$ from the right, and the inclusion $G\to X$ is $G$-equivariant on the right.
See \S\ref{SEC:examples} for some examples and constructions of $X$.
In this paper, we resolve key issues from \cite{tanimoto2012height},
and establish Manin's conjecture
(via the dense open set $G$)
for all sufficiently split $X$.

\begin{theorem}
\label{THM:main-Manin-application}
Assume $X$ is strictly split (Definition~\ref{DEFN:strictly-split-X}).
If $\mathsf{H}$ is a standard Weil height (Definition~\ref{DEFN:class-of-standard-Weil-heights}) associated to the anticanonical line bundle $K_X^{-1}$,
then
\begin{equation}
\label{sharp-cutoff}
\lim_{B\to \infty}
\frac{\#\set{x\in G(\QQ): \mathsf{H}(x) \le B}}{B (\log{B})^{\rank(\Pic(X)) - 1}}
= \mathcal{A}_{X,\mathsf{H}},
\end{equation}
where $\mathcal{A}_{X,\mathsf{H}}>0$ is Peyre's constant \eqref{EQN:Peyre-constant}.
Moreover, if $w\in C^\infty_c(\RR)$ and $B\ge 2$, then
\begin{equation}
\label{EQN:log-saving-Manin-Peyre-conjecture-for-standard-anticanonical-Weil-heights}
\sum_{x\in G(\QQ)} w{\left(\frac{\mathsf{H}(x)}{B}\right)}
= B (\log{B})^{\rank(\Pic(X)) - 1}
\left(\mathcal{A}_{X,\mathsf{H}}\int_0^\infty w(t)\, dt
+ O_{X,\mathsf{H},w}{\left(\frac{1}{\log{B}}\right)}\right).
\end{equation}
\end{theorem}

\begin{remark}
\label{remove-hypotheses}
We work in ``geometric generality, realized arithmetically over $\QQ$''
to keep ideas clean and clear.
But our arguments may well extend
to arbitrary number fields.
The splitness condition on $X$ would then always be satisfiable after base change;
alternatively, it
might be directly removable with enough extra notation and combinatorial effort.
\end{remark}

\begin{remark}
Generalizing our work to equivariant compactifications of homogeneous spaces of $G$ (see e.g.~\cite{derenthal2015equivariant} for the distinction) may require serious new ideas, because
(1) a quotient $H\backslash G$ of $G$ need not be a group (unlike in the case of abelian groups);
and (2) in general, $(H\backslash G)(\QQ)$ includes Galois-invariant elements of $H(\ol{\QQ})\backslash G(\ol{\QQ})$, not just $H(\QQ)\backslash G(\QQ)$.
\end{remark}

To compare our Theorem~\ref{THM:main-Manin-application} with \cite{tanimoto2012height}*{Theorem~5.1}, we need to elaborate on the geometry of $X$.
Certainly $X$ is a rational surface, since it is birational to $G$.
View the coordinates $a$, $b$ on $G$ as rational functions on $X$.
For $f\in \Rat(X)$, let $\div_0(f)$, $\div_\infty(f)$ be the zero and polar Weil divisors of $f$, respectively; so $\div(f) = \div_0(f) - \div_\infty(f)$.

The boundary $D\defeq X \setminus G$
coincides with the indeterminacy locus of the rational map $(a,b)\maps X \dashrightarrow G$.
So by the algebraic Hartogs lemma \cite{stacks-project}*{\href{https://stacks.math.columbia.edu/tag/0BCS}{Tag~0BCS}}, we have
\begin{equation}
\label{EQN:basic-boundary-decomposition}
D = \Supp(\div(a)) \cup \Supp(\div_\infty(b)).
\end{equation}

Write $D = \bigcup_{\jboundary\in J} D_\jboundary$,
where the $D_\jboundary$ are irreducible over $\QQ$.
Then \cite{tanimoto2012height}*{Theorem~5.1} proves our Theorem~\ref{THM:main-Manin-application} for all split $X$ satisfying the following conditions, which we list in roughly increasing order of significance:
(i) the divisor $D\cup \Supp(\div_0(b))$ has strict normal crossings,
(ii) the divisor $\div(a)$ is reduced,
and (iii) the mysterious condition
(for each $\jboundary\in J$)
\begin{equation}
\label{INEQ:mysterious-numerical-condition}
\ord_{D_\jboundary}(a) < 0 \Rightarrow \ord_{D_\jboundary}(b) < \ord_{D_\jboundary}(a).
\end{equation}
We remark that for split $X$,
the condition (i) implies in particular that $X$ is strictly split.

The condition \eqref{INEQ:mysterious-numerical-condition}
is related to the positivity of $K_X^{-1}$;
see Proposition~\ref{PROP:upper-bound-on-anticanonical-Weil-divisor} below.
Similar conditions, involving variables and degrees, are familiar in the circle method.
Later (in \S\ref{SEC:new-geometry}) we will prove the following result, which we state now to give numerical context for \eqref{INEQ:mysterious-numerical-condition}:

\begin{proposition}
\label{PROP:apply-lower-bound-on-anticanonical-Weil-divisor}
Let $\jboundary\in J$ and $c\in \QQ$.
Then $\ord_{D_\jboundary}(b-c) \le \ord_{D_\jboundary}(a)$.
\end{proposition}

When \eqref{INEQ:mysterious-numerical-condition} fails, we seem to need a new idea.
The main culprit, revealed by geometric calculations going beyond \cite{tanimoto2012height}, turns out to be pairs of divisors at which \eqref{INEQ:mysterious-numerical-condition} and a counterpart for $\ord_{D_\jboundary}(a)>0$ fail.
Relevant here are the \emph{special} divisors we now define.

\begin{definition}
\label{DEFN:special-divisors}
Given $\jboundary\in J$, call $D_\jboundary$ \emph{special} if
$\max_{c\in \QQ} \ord_{D_\jboundary}(b-c) = \ord_{D_\jboundary}(a)$.
\end{definition}

\begin{remark}
Given $c\in \QQ$ and a compactification $i\maps G\to X$,
the \emph{left} translate $i_c\maps G\to X$ given by $({}'a,{}'b) \mapsto i((1,c)({}'a,{}'b))$
satisfies ${}'a=a$ and ${}'b+c=b$.
This explains why $\QQ$-translates of $b$, not just $b$ itself, matter in our work.
\end{remark}

Suppose there are $\kspecial\ge 0$ special divisors with $\ord_{D_\jboundary}(a)<0$, and $\lspecial\ge 0$ special divisors with $\ord_{D_\jboundary}(a)>0$.
Then the main issue,
after new ``complexity-lowering'' non-archimedean calculations in \S\ref{SEC:non-archimedean-local-calculations} (in the spirit of \cites{heath1996new,
getz2018secondary,tran_thesis,wang2024_isolating_special_solutions})
relying on a new $G$-related source (Proposition~\ref{PROP:G-induced-nonconstancy-of-a-rational-function-on-Dj}) of local coordinates and cancellation in complete exponential integrals,
is to appropriately bound a class of multiple Dirichlet series including,
roughly,
\begin{equation}
\label{EXPR:desired-key-general-multiple-Dirichlet-series}
\sum_{\substack{\alpha = (m_1\cdots m_\kspecial)/(n_1\cdots n_\lspecial): \\
\textnormal{pairwise coprime }m_1,\dots,m_\kspecial,n_1,\dots,n_\lspecial\ge 1}}
\frac{\mathcal{F}(\alpha) \phase(\cfinal_0\alpha)}{m_1^{\beta'_1}\cdots m_\kspecial^{\beta'_\kspecial}}
\prod_{1\le \jspecial\le \lspecial} \frac{\phase(-\cfinal_\jspecial\alpha\bmod{\ZZ_{n_\jspecial}})}{n_\jspecial^{\gamma'_\jspecial}},
\end{equation}
for some constants $\cfinal_0,\cfinal_1,\dots,\cfinal_\lspecial\in \QQ$ and a suitable Fourier transform $\mathcal{F}\maps \RR_{>0} \to \CC$ of a reciprocal archimedean height function.
More precisely, a variant \eqref{EXPR:final-expanded-adelic-data-sum}
of the series \eqref{EXPR:desired-key-general-multiple-Dirichlet-series}
will appear in \S\ref{SEC:archimedean-endgame},
but for now we remain somewhat vague in order to sketch ideas.

To analyze \eqref{EXPR:desired-key-general-multiple-Dirichlet-series},
we first use a change of variables and a height derivative bound (Lemma~\ref{LEM:localized-regular-and-height-derivative-estimates})
in order to prove a decay bound
$(\alpha\,\frac{\partial}{\partial\alpha})^k \mathcal{F}(\alpha)
\ll_k (\alpha^\xi+\alpha^{-\xi})^{-1}$
for all integers $k\ge 0$,
for some small real $\xi>0$.
We then decompose \eqref{EXPR:desired-key-general-multiple-Dirichlet-series} into regions according to the quality of expected oscillation as $m_1,\dots,m_\kspecial,n_1,\dots,n_\lspecial$ vary in dyadic intervals.

When $\#\set{\cfinal_1,\dots,\cfinal_\lspecial} \ge 2$, we find cancellation in certain correlations of Kloosterman fractions; \cites{duke1997bilinear,bettin2018trilinear} need not apply, but reciprocity plus a multivariate Weyl-type inequality (Proposition~\ref{PROP:Weyl-inequality-for-monomials-in-several-variables}) for monomials suffices, though just barely when in the most lopsided ranges.
Ultimately, for certain $\beta'_1,\dots,\beta'_\kspecial,\gamma'_1,\dots,\gamma'_\lspecial = 1+O(s-1)$, the series \eqref{EXPR:desired-key-general-multiple-Dirichlet-series}
\emph{morally} has a pole of order $\le \rank(\Pic(X)) - 1$ at $s=1$.
Theorem~\ref{THM:main-Manin-application} then follows.

In our setting, one could express \eqref{EXPR:desired-key-general-multiple-Dirichlet-series} as a weighted average of Clausen zeta functions $\sum_{m\ge 1} \phase(m\theta) m^{-s}$ over a family of angles $\theta$.
Clausen functions for fixed $\theta$ can be studied in depth (see e.g.~\cite{knill2012analytic}), but obtaining uniform results may be difficult,
since the complexity of $\theta$ may vary wildly.
Our work smooths away such difficulties by
averaging.

It remains open to obtain a power-saving asymptotic expansion of \eqref{EQN:log-saving-Manin-Peyre-conjecture-for-standard-anticanonical-Weil-heights} in general;
secondary terms (cf.~\cites{heath1996new,getz2018secondary,tran_thesis}) elude us when $D$ is sufficiently complicated.
In our approach, the key missing ingredient seems to be to
meromorphically continue (not just \emph{bound})
series such as $\sum_{m,n\ge 1} \phase(m/n) \tau_i(m)\tau_j(n) m^{-s_1}n^{-s_2}$,
where $\tau_i$ is the $i$-fold divisor function.
Existing work on multiple Dirichlet series, such as \cite{diaconu2003multiple},
might help.
Takloo-Bighash
pointed out to us that
Pi's work in \cite{ramin},
which is based on the functional equation of the Hurwitz zeta function,
implies that $\sum_{m,n\ge 1} \phase(m/n) m^{-s_1} n^{-s_2}$,
which converges absolutely for $\Re(s_1),\Re(s_2)>1$,
has a meromorphic continuation to all $(s_1,s_2)\in \CC^2$.
This might generalize.

We remark that a homological stability result
for equivariant compactifications of solvable linear algebraic groups over $\CC$
is known \cite{BoyerHurtubiseMilgram}.
It would be interesting to try to identify any connections between
the methods of \cite{BoyerHurtubiseMilgram} and the present paper.

\subsection{Conventions}
\label{SUBSEC:conventions}

A \emph{variety} is an integral, separated scheme of finite type over a field.
We let $\Rat(\Upsilon)$ denote the field of rational functions on an integral scheme $\Upsilon$.
(If $\eta$ is the generic point of $\Upsilon$, then $\Rat(\Upsilon)=\mathcal{O}_{\Upsilon,\eta}$.)
A reduced, effective Weil divisor on $X$ has \emph{strict normal crossings} if
its irreducible components $\Upsilon_i$ are smooth,
the pairwise intersections $\Upsilon_i\cap \Upsilon_j$ are smooth of dimension $\le 0$,
and the triple intersections $\Upsilon_i\cap \Upsilon_j\cap \Upsilon_k$ are empty.

Let $\ZZ_n \defeq \prod_{p\mid n} \ZZ_p$ and $\QQ_n \defeq \prod_{p\mid n} \QQ_p$.
For $t\in \RR$, let $\phase(t) \defeq \exp(2\pi it)$.
For $y\in \QQ_n$, let $$\phase(y\bmod{\ZZ_n}) \defeq \phase(y'),$$
for any $y'\in \ZZ[1/n]$ with $y'\equiv y\bmod{\ZZ_n}$.
Define the additive automorphic character
$\psi = \prod_v \psi_v \maps \adele_\QQ \to \CC^\times$
by $\psi(x) \defeq \prod_v \psi_v(x_v)$ for $x=(x_v)_v\in \adele_\QQ$,
where $$\psi_\infty(x_\infty) \defeq \phase(-x_\infty),
\quad \psi_p(x_p) \defeq \phase(x_p\bmod{\ZZ_p}).$$
Here $p\in \Spec(\ZZ)$ denotes a prime, and $\infty$ denotes the real place of $\QQ$.

For $\mathcal{Q}_1\in \CC$ and $\mathcal{Q}_2\in \RR$,
we write $\mathcal{Q}_1\ll_\mathscr{P} \mathcal{Q}_2$,
or $\mathcal{Q}_2\gg_\mathscr{P} \mathcal{Q}_1$,
or $\mathcal{Q}_1 = O_\mathscr{P}(\mathcal{Q}_2)$,
to mean that $\abs{\mathcal{Q}_1} \le \mathcal{B} \mathcal{Q}_2$
for a real constant $\mathcal{B}=\mathcal{B}(\mathscr{P})>0$ depending on $\mathscr{P}$.
We write $\mathcal{Q}_1\asymp \mathcal{Q}_2$
to mean $\mathcal{Q}_1\ll \mathcal{Q}_2\ll \mathcal{Q}_1$.
For a statement $\mathcal{S}$, let $\bm{1}_\mathcal{S} \defeq 1$ if $\mathcal{S}$ holds, and $\bm{1}_\mathcal{S} \defeq 0$ otherwise.

There are some pieces of notation, such as the letters $f$ and $A$,
that we will sometimes number sequentially,
in order to avoid confusion between different contexts.

\section{Background}
\label{SEC:background}


Recall the setting of \S\ref{SEC:intro},
with $G$, $X$ specified as in the paragraph before Theorem~\ref{THM:main-Manin-application}.
We need some background on geometry, heights, and analysis,
drawn mostly from \cite{tanimoto2012height}*{\S\S1--3}.

\subsection{Geometry}

Let $g\in G$ act on elements $f\in \Rat(X)$ from the right:
$(fg)(x) \defeq f(xg^{-1})$.
There is an analogous left action on $\Rat(G) = \Rat(X)$,
but by default on $\Rat(X)$ we act from the right.
These actions on functions induce corresponding actions on differentials.

Let $K_X$ be the canonical line bundle on $X$.
It has two local sections of particular interest:
the right-invariant top form $\omega \defeq db\, da/a$ on $G$,
and the left-invariant top form $\omega/a$ on $G$.
For each $\jboundary\in J$, let $\mathsf{d}_\jboundary \defeq -\ord_{D_\jboundary}(\omega)$;
then $-\div(\omega) = \sum_{\jboundary\in J} \mathsf{d}_\jboundary D_\jboundary$.

%
Throughout the paper,
a useful bookkeeping tool is
the \emph{equivariant Picard group} $$\Pic^G(X)
= \bigoplus_{\jboundary\in J} \ZZ D_\jboundary;$$
see \cite{tanimoto2012height}*{\S1} or \cite{heller2015equivariant}*{Proposition~2.12(2)} for details, noting that $D_\jboundary G = D_\jboundary$.


\begin{proposition}
\label{PROP:compute-Picard-group-in-terms-of-boundary}
We have $\Pic(X_\efield) = \Pic^G(X_\efield)/\ZZ\div(a\vert_{X_\efield})$ for any field $\efield\contains \QQ$.
\end{proposition}

\begin{proof}
Since $\Pic(G_\efield) = 0$,
the irreducible components of $D_\efield$ generate $\Pic(X_\efield)$.
(Here we identify a Weil divisor $W$ on $X_\efield$
with the invertible sheaf $\mathcal{O}_{X_\efield}(W)\belongs \Rat(X_\efield)$.)
Relations correspond to $\efield$-morphisms $X_\efield\setminus D_\efield \to (\GG_m)_\efield$,
namely $c a^k$ for $c\in \efield^\times$, $k\in \ZZ$.
\end{proof}

An important combinatorial role will be played by
the following closed cone:
$$\Lambda(X) \defeq \sum_{\jboundary\in J} \RR_{\ge 0} D_\jboundary
\belongs \Pic(X)\otimes \RR.$$
(In fact,
$\Lambda(X)$ is the \emph{cone of pseudo-effective divisors} of $X$,
by \cite{tanimoto2012height}*{Proposition~1.1(3)}.)
By the following result, $K_X^{-1}$ lies in $\Lambda^\circ(X)$, the interior of $\Lambda(X)$.

\begin{proposition}
[\cite{tanimoto2012height}*{Proposition~1.2}]
\label{PROP:poles-of-left-invariant-top-form}
We have $-\ord_{D_\jboundary}(\omega/a) \ge 1$ for all $\jboundary\in J$.

\end{proposition}





The $\ZZ$-module
$\Pic(X_{\ol{\QQ}})$ is finite free, since $X_{\ol{\QQ}}$ is a rational surface over an algebraically closed field.
Also, $\Pic(X) = \Pic(X_{\ol{\QQ}})^{\Gal(\ol{\QQ}/\QQ)}$ by \cite{kresch2008effectivity}*{Remark~3.2(ii)}, because $X(\adele_\QQ) \ne \emptyset$;
so in particular, $\Pic(X)$ is free.
Say $X$ is \emph{split} if
$\Pic(X) = \Pic(X_{\ol{\QQ}})$
\cite{manin1993notes}*{Definition~2.2(a)}.
By Proposition~\ref{PROP:compute-Picard-group-in-terms-of-boundary},
$X$ is split if and only if each $D_\jboundary$ is geometrically irreducible.
%
%

\begin{definition}
\label{DEFN:strictly-split-X}
Say a split $X$ is \emph{strictly split} if
there exists a composition $\mathcal{Y}\to X$ of blowups with smooth $G$-invariant centers,
such that $\mathcal{Y}$ is split and $\mathcal{Y}\setminus G$ has strict normal crossings.
\end{definition}

We can take $\mathcal{Y}=X$ if $D$ has strict normal crossings.
Split $\mathcal{Y}$ help
in the final portion (\S\ref{SEC:final-reductions}) of the proof of Theorem~\ref{THM:main-Manin-application}.
We do not know if all split $X$ are strictly split (over $\QQ$).

\subsection{Heights}

Let $\mathscr{H}_{\PP^n}\maps \PP^n(\QQ) \to \RR_{>0}$ be defined by the formula
$$\mathscr{H}_{\PP^n}([y_0:\dots:y_n]) \defeq (y_0^2+\dots+y_n^2)^{1/2},$$
for coprime integers $y_0,\dots,y_n$.
Let $\mathcal{O}_{\PP^n_\QQ}(1)\in \Pic(\PP^n_\QQ)$
be the \emph{hyperplane bundle} on $\PP^n_\QQ$.

Weil heights
(see \cite{hindry2000diophantine}*{{\S}B.3})
are only unique up to bounded factors,
so Theorem~\ref{THM:main-Manin-application} requires reasonable restrictions on $\mathsf{H}$.
For convenience, we impose smoothness, as in \cite{chambert2010igusa}.\footnote{However,
a referee pointed out to us that Manin's conjecture for smooth heights is known to imply
Manin's conjecture for continuous heights, by \cite{peyre1995hauteurs}*{Proposition~3.3} (or \cite{chambert2010igusa}*{Proposition~2.10}).
Therefore, while smoothness of $\mathsf{H}$ is convenient for proofs,
it is not necessary in the statement of Theorem~\ref{THM:main-Manin-application}.}
Call a function $\Delta\maps X(\adele_\QQ) \to \RR_{>0}$ \emph{simple} if there exist smooth functions $\Delta_v\maps X(\QQ_v) \to \RR_{>0}$, with $\Delta_v=1$ at all but finitely many places $v$,
such that $$\Delta(x) = \prod_v \Delta_v(x_v)
\quad \textnormal{for all $x=(x_v)_v\in X(\adele_\QQ)$}.$$
For finite $v$,
\emph{smooth} means \emph{locally constant}, as in \cite{chambert2010igusa}*{\S2.1.2}.

%

\begin{definition}
\label{DEFN:class-of-standard-Weil-heights}
Let $\lbundle \in \Pic(X)$.
Call $H_\lbundle\maps X(\QQ) \to \RR_{>0}$
a \emph{standard Weil height associated to $\lbundle$} if there exist
\begin{enumerate}
\item morphisms $f_1\maps X\to \PP^i_\QQ$
and $f_2\maps X\to \PP^j_\QQ$ with $i,j\ge 0$;

\item an integer $k\ge 1$
with $f_1^\ast(\mathcal{O}_{\PP^i_\QQ}(1))
\cong f_2^\ast(\mathcal{O}_{\PP^j_\QQ}(1))
\otimes \lbundle^{\otimes k}$;
and

\item a simple function $\Delta\maps X(\adele_\QQ) \to \RR_{>0}$;
\end{enumerate}
such that for all $x\in X(\QQ)$ we have
$H_\lbundle(x) = \Delta(x)
\cdot \mathscr{H}_{\PP^i}(f_1(x))^{1/k}
/ \mathscr{H}_{\PP^j}(f_2(x))^{1/k}$.
\end{definition}


\begin{remark}
\label{uniqueness-of-Weil-heights}
The ratio of any two standard Weil heights associated to $\lbundle$ is the restriction to $X(\QQ)$ of a simple function $X(\adele_\QQ) \to \RR_{>0}$.
\end{remark}




To tackle Theorem~\ref{THM:main-Manin-application} analytically,
we want local heights $H_{D_\jboundary,v}$ that are roughly inversely proportional to ``$v$-adic distance'' to $D_\jboundary$.
The key constructions and formulas are given below
(following \cite{tanimoto2012distribution}*{\S3.2.1}
or \cite{tanimoto2012height}*{\S2},
after tracing through the definitions).

Embed $X$ in some projective space $\PP^{\mathsf{n}}_\QQ$ over $\QQ$,
with $\mathsf{n}\ge 2$.
Let $\mathscr{X}$ be the closure of $X$ in $\PP^{\mathsf{n}}_\ZZ$.
Scheme $\mathscr{X}$ is integral, and the closure of any prime divisor on $X$ is a prime divisor on $\mathscr{X}$.
For each divisor $W\in \bigoplus_{\jboundary\in J} \ZZ D_\jboundary$,
let $\mathscr{W}$ be the closure of $W$ in $\mathscr{X}$.
That is,
let $\mathscr{D}_\jboundary$ denote the closure of $D_\jboundary$ in $\mathscr{X}$,
and define $\mathscr{W}\in \bigoplus_{\jboundary\in J} \ZZ \mathscr{D}_\jboundary$ in general
so that the map $W\mapsto \mathscr{W}$ is linear.

The integral model $\mathscr{X}$ is smooth over a dense open subset $\oset\belongs \Spec(\ZZ)$.
Moreover, if $p\in \oset$ and $x_p\in X(\QQ_p)=\mathscr{X}(\ZZ_p)$,
then $x_p\in \mathcal{V}(\ZZ_p)$
for every regular open neighborhood $\mathcal{V}$
of the point $(x_p\bmod{p})\in \mathscr{X}(\FF_p)$ in $\mathscr{X}$.
Upon considering local equations of divisors on $\mathscr{X}$,
we are led to the following definition,
featuring a certain ``distance'' quantity $\abs{f(x_v)}_v$.


\begin{definition}
\label{DEFN:standard-local-Weil-height-for-boundary-divisors}
Given a divisor $W\in \Pic^G(X) = \bigoplus_{\jboundary\in J} \ZZ D_\jboundary$,
call a collection of functions
$$(H_{W,v}\maps G(\QQ_v) \to \RR_{>0})_v$$
\emph{good} if there exists a simple function
$\Delta\maps X(\adele_\QQ) \to \RR_{>0}$ such that
for each element $f\in \Rat(\mathscr{X})$
and regular open subscheme $\mathcal{V}\belongs \mathscr{X}$
with $\div(f\vert_{\mathcal{V}}) = \mathscr{W}\vert_{\mathcal{V}}$,
the following hold:
\begin{enumerate}
\item For each place $v$ and compact set $K_v\belongs \mathcal{V}(\QQ_v)$,
there exists a smooth function $\varphi\maps K_v\to \RR_{>0}$
such that for all $x_v\in G(\QQ_v) \cap K_v$,
we have $$H_{W,v}(x_v) = \frac{\varphi(x_v)}{\abs{f(x_v)}_v}.$$

\item For all $v\in \oset$ and $x_v\in G(\QQ_v)\cap \mathcal{V}(\ZZ_v)$,
we have $$H_{W,v}(x_v) = \frac{\Delta_v(x_v)}{\abs{f(x_v)}_v}.$$
\end{enumerate}
\end{definition}

\begin{remark}
If we fix $v\in \oset$ and $x_v\in G(\QQ_v)$, and we let $\mathcal{V}$ and $f$ vary,
then in condition~(2),
the value of $\abs{f(x_v)}_v$ is independent of the choice of $\mathcal{V}$ and $f$,
provided that $x_v\in \mathcal{V}(\ZZ_v)$.
But condition~(1) is more flexible, since the set $\mathcal{V}(\QQ_v)$ is in general larger than $\mathcal{V}(\ZZ_v)$.
\end{remark}






\begin{proposition}
\label{PROP:homomorphic-adelic-local-height-decomposition}
Let $\mathsf{H}$ be a standard Weil height associated to $K_X^{-1}$.
\begin{enumerate}
\item There exists a homomorphism $\Pic(X)\to \map{Fun}(X(\QQ), \RR_{>0})$,
assigning each $\lbundle\in \Pic(X)$ to a standard Weil height $H_\lbundle$ associated to $\lbundle$,
such that $H_{K_X^{-1}} = \mathsf{H}$.

\item For each $W\in \Pic^G(X)$,
let $H_W \defeq H_{\mathcal{O}_X(W)}$.
There exists a homomorphism $\Pic^G(X)\to \prod_v \map{Fun}(G(\QQ_v), \RR_{>0})$,
assigning each $W$ to a good collection $(H_{W,v})_v$, such that
\begin{align}
\abs{a(g_v)}_v &= \prod_{\jboundary\in J} H_{D_\jboundary,v}(g_v)^{-\ord_{D_\jboundary}(a)}
&&\textnormal{for all $g_v=(a(g_v),b(g_v))\in G(\QQ_v)$, for all $v$, and}
\label{EQN:local-character-height-consistency-relation} \\
H_W(g) &= \prod_v H_{W,v}(g)
&&\textnormal{for all $g\in G(\QQ)$}.
\label{EQN:local-global-height-consistency}
\end{align}
\end{enumerate}
\end{proposition}

\begin{proof}
This is all proven in \cite{tanimoto2012distribution}*{first paragraph of \S3.2.1},
up to differences in exposition,
so we will be relatively brief.

(1):
Use the fact that $\Pic(X)$ is a finite free $\ZZ$-module.

(2):
Use Proposition~\ref{PROP:compute-Picard-group-in-terms-of-boundary} to choose very ample Weil divisors $W_1,\dots,W_{\card{J}-1} \belongs \sum_{\jboundary\in J} \ZZ D_\jboundary$ that generate $\Pic(X) \otimes \QQ$.
For each divisor $W_l$ (where $1\le l\le \card{J}-1$),
choose sections $$\phi^{(l)}_{1},\dots,\phi^{(l)}_{\dim \Gamma(X,\mathcal{O}_X(W_l))}
\in \Gamma(X,\mathcal{O}_X(W_l))$$
that globally generate $\mathcal{O}_X(W_l)$.
If $\phi^{(l)}\maps X\to \PP^{-1 + \dim \Gamma(X,\mathcal{O}_X(W_l))}_\QQ$
denotes the morphism defined by the sections $\phi^{(l)}_{\ast}$,
then by Remark~\ref{uniqueness-of-Weil-heights}
for the line bundle $\lbundle=\mathcal{O}_X(W_l)$, we have
$$\frac{H_{W_l}(x)}
{\mathscr{H}_{\PP^{-1 + \dim \Gamma(X,\mathcal{O}_X(W_l))}}(\phi^{(l)}(x))}
= \Delta^{(l)}(x)$$
for all $x\in X(\QQ)$,
for some simple function $\Delta^{(l)}\maps X(\adele_\QQ) \to \RR_{>0}$.
But if $x\in G(\QQ)$,
then $\phi^{(l)}(x) = [\phi^{(l)}_{1}(x):\dots:\phi^{(l)}_{\dim \Gamma(X,\mathcal{O}_X(W_l))}(x)]$,
since $\Gamma(X,\mathcal{O}_X(W_l))
\belongs \Gamma(G,\mathcal{O}_X)$.
Thus
$$H_{W_l}(x) = \Delta^{(l)}(x)
\cdot \mathscr{H}_{\PP^{-1 + \dim \Gamma(X,\mathcal{O}_X(W_l))}}([\phi^{(l)}_{1}(x):\dots:\phi^{(l)}_{\dim \Gamma(X,\mathcal{O}_X(W_l))}(x)])$$
for all $x\in G(\QQ)$.
Now, let $$H_{W_l,\infty}(x_\infty)
\defeq \Delta^{(l)}_\infty(x_\infty)
\cdot (\abs{\phi^{(l)}_{1}(x_\infty)}_\infty^2 + \dots + \abs{\phi^{(l)}_{\dim \Gamma(X,\mathcal{O}_X(W_l))}(x_\infty)}_\infty^2)^{1/2}$$
for all $x_\infty\in G(\RR)$,
and for all $x_p\in G(\QQ_p)$ let $$H_{W_l,p}(x_p)
\defeq \Delta^{(l)}_p(x_p)
\cdot \max(\abs{\phi^{(l)}_{1}(x_p)}_p, \dots, \abs{\phi^{(l)}_{\dim \Gamma(X,\mathcal{O}_X(W_l))}(x_p)}_p).$$
Then in particular, \eqref{EQN:local-global-height-consistency} holds for $W=W_l$,
in view of the well-known factorization
\begin{equation*}
\mathscr{H}_{\PP^n}([y_0:\dots:y_n])
= (y_0^2+\dots+y_n^2)^{1/2} \prod_p \max(\abs{y_0}_p,\dots,\abs{y_n}_p)
\end{equation*}
for $(y_0,\dots,y_n)\in \QQ^{n+1} \setminus \{(0,\dots,0)\}$.


Given $(H_{W_l,v})_v$, the relation \eqref{EQN:local-character-height-consistency-relation} then uniquely specifies a homomorphism $\Pic^G(X)\to \prod_v \map{Fun}(G(\QQ_v), \RR_{>0})$,
since $W_1,\dots,W_{\card{J}-1},\div(a)$ generate $\Pic^G(X)\otimes \QQ$.
The resulting $(H_{W,v})_v$ are good.
Furthermore, by (1), \eqref{EQN:local-character-height-consistency-relation}, and the product formula $\prod_v \abs{a(g)}_v = 1$ for $g\in G(\QQ)$,
the relation \eqref{EQN:local-global-height-consistency} extends from $W_l$ to all $W$,
in view of the formula $\Pic(X) = \Pic^G(X)/\ZZ\div(a)$ from Proposition~\ref{PROP:compute-Picard-group-in-terms-of-boundary}.
This completes the proof of (2).
\end{proof}

From now on, fix heights $H_\lbundle$, $(H_{W,v})_v$ satisfying Proposition~\ref{PROP:homomorphic-adelic-local-height-decomposition}.

\subsection{Polar combinatorics}
\label{SUBSEC:polar-combinatorics}

For combinatorial purposes,
we make the identification
$$\Pic^G(X) = \bigoplus_{\jboundary\in J} \ZZ D_\jboundary = \ZZ^J,$$
where $D_\jboundary\in \Pic^G(X)$ is identified with
the $\jboundary$th basis vector of $\ZZ^J$,
for each $\jboundary\in J$.
Moreover, for any subset $I\belongs J$, we identify $\ZZ^I$ with
the subspace $\bigoplus_{\jboundary\in I} \ZZ D_\jboundary
\belongs \ZZ^J$.
We make similar identifications over $\RR$ and $\CC$,
so that for instance, $\Pic^G(X)\otimes \CC = \CC^J$.
Finally, for each vector $\bm{z}=(z_\jboundary)_{\jboundary\in J}\in \CC^J$,
let $\Re(\bm{z})\defeq (\Re(z_\jboundary))_{\jboundary\in J}\in \RR^J$
and let $\bm{z}\vert_I\defeq (z_\jboundary)_{\jboundary\in I}\in \CC^I$.

Let $A_1,A_2\in \RR\cup \{\infty\}$ with $A_1<A_2$,
and let $\openint(A_1,A_2)\belongs \RR$ denote the open
interval from $A_1$ to $A_2$.
For a holomorphic function $\varphi(\bm{z})$
on the region $\Re(\bm{z}) \in \openint(A_1,A_2)^J$,
we define the \emph{shifted integral}
\begin{equation}
\label{EQN:define-shifted-integral}
\mathscr{S}(\varphi)(\bm{z})\defeq (2\pi)^{-1} \int_\RR \varphi(\bm{z}+it\div(a))\, dt.
\end{equation}
The main term in the asymptotic \eqref{EQN:log-saving-Manin-Peyre-conjecture-for-standard-anticanonical-Weil-heights} will come from
such integrals,
whose general structure we now describe
in terms of the combinatorics of the vector $\div(a)\in \ZZ^J$.

Recall \eqref{EQN:basic-boundary-decomposition}.
Let $J_1\defeq \set{\jboundary\in J: \ord_{D_\jboundary}(a) > 0}$
and $J_2\defeq \set{\jboundary\in J: \ord_{D_\jboundary}(a) < 0}$,
and let
\begin{equation*}
J_3 \defeq \set{\jboundary\in J: \ord_{D_\jboundary}(a) = 0}
\belongs \set{\jboundary\in J: \ord_{D_\jboundary}(b) < 0}.
\end{equation*}
Clearly $J_1$, $J_2$, $J_3$ partition $J$.
For convenience, let $\mathsf{u}_\jboundary\defeq \ord_{D_\jboundary}(a)$ for each $\jboundary\in J$.

The $\mathcal{X}$-function framework of \cite{chambert2001fonctions}*{\S3} is crucial
for studying \eqref{EQN:define-shifted-integral}.
For each $I\belongs J$, let
\begin{equation*}
\mathcal{X}_{\RR_{\ge 0}^I}(\bm{z})
\defeq \int_{\RR_{\ge 0}^I} e^{-\bm{y}\cdot \bm{z}\vert_I}\, d\bm{y}
= \prod_{\jboundary\in I} z_\jboundary^{-1}
\quad\textnormal{for $\Re(\bm{z})\in \RR_{>0}^J$},
\end{equation*}
with meromorphic continuation to all $\bm{z}\in \CC^J$.
We have $\div(a)\vert_I\in \ZZ^I$.
Consider the cone
\begin{equation*}
\Lambda_I(X)
\defeq \sum_{\jboundary\in I} \RR_{\ge 0} D_\jboundary
\belongs \RR^I / (\RR \div(a)\vert_I).
\end{equation*}
Let $\mu_I$ be the quotient measure on $\RR^I / (\RR \div(a)\vert_I)$
induced by assigning unit length to the vector $\div(a)\vert_I\in \RR \div(a)\vert_I$.
The dual cone $\Lambda^\ast_I(X)$ of $\Lambda_I(X)$ is
the set of elements $\varkappa\in (\RR^I)^\ast = \Hom(\RR^I,\RR)$
such that $(\varkappa(D_\jboundary))_{\jboundary\in I}\in \RR_{\ge 0}^I$
and $\varkappa(\div(a)\vert_I) = 0$.
Let
\begin{equation}
\label{EQN:define-X-Lambda_I-quotient-cone-function}
\mathcal{X}_{\Lambda_I(X)}(\bm{z})
\defeq \int_{\Lambda^\ast_I(X)} e^{-\varkappa(\bm{z}\vert_I)}\, d\varkappa
\quad\textnormal{for $\Re(\bm{z})\in \RR_{>0}^J + \RR\div(a)$},
\end{equation}
where $d\varkappa$ is the Haar measure on $(\RR^I / (\RR \div(a)\vert_I))^\ast$
such that $d\varkappa$ is dual to the measure $\mu_I$.
The function $\mathcal{X}_{\Lambda_I(X)}$ is \emph{$\CC\div(a)$-invariant},
i.e.~it is invariant under translation by $\CC\div(a)$.

\begin{remark}
\label{RMK:saturated-case-dual-measure-interpretation}
If $\gcd_{\jboundary\in I}(\mathsf{u}_\jboundary) = 1$,
then $d\varkappa$ is the unique Haar measure on $(\RR^I / (\RR \div(a)\vert_I))^\ast$
such that the dual lattice $(\ZZ^I / (\ZZ \div(a)\vert_I))^\ast$ has covolume $1$.
In particular, this applies if $I=J$,
because $\gcd_{\jboundary\in J}(\mathsf{u}_\jboundary) = 1$
(since $\Pic(X)$ is free
and $\Pic(X) = \ZZ^J / \ZZ \div(a)$).
\end{remark}

\begin{remark}
\label{RMK:X-function-vanishing-criteria}
In some degenerate cases, $\mathcal{X}_{\Lambda_I(X)}$ vanishes identically.
For example,
if $\card{I\setminus J_3}\ge 2$
and $\RR_{\ge 0}^I \cap (\RR\div(a)\vert_I) \ne 0$,
then all the nonzero coordinates of $\div(a)\vert_I$ are of the same sign
(and the number of nonzero coordinates is $\card{I\setminus J_3}\ge 2$),
so $$\dim(\Lambda^\ast_I(X)) \le \card{I}-2
< \card{I}-1 = \dim(\RR^I / (\RR \div(a)\vert_I)),$$
whence $\mathcal{X}_{\Lambda_I(X)} = 0$.


\end{remark}


Let $\norm{\bm{z}}\defeq \max_{\jboundary\in J}{\abs{z_\jboundary}}$.
Let $\mathcal{H}_J(A_1,A_2)$ be the ring of holomorphic functions $\varphi(\bm{z})$ on $\Re(\bm{z})\in \openint(A_1,A_2)^J$ that are \emph{polynomially bounded in vertical strips},
meaning that for each compact set $K\belongs \openint(A_1,A_2)$,
there exists a real $A=A_3(K)>0$ such that
$$\varphi(\bm{z})\ll_K (1+\norm{\bm{z}})^A
\quad\textnormal{for $\Re(\bm{z})\in K^J$}.$$
Let $\mathcal{H}^\CC_J(A_1,A_2)$ be the ring of $\CC\div(a)$-invariant functions in $\mathcal{H}_J(A_1,A_2)$.
Let $\mathcal{H}_{\dagger,J}(A_1,A_2)$ be the set of functions $\varphi\in \mathcal{H}_J(A_1,A_2)$
such that for each compact set $K\belongs \openint(A_1,A_2)$,
there exists a real $A=A_4(K)>0$ such that
\begin{equation}
\label{INEQ:growth-decay-condition-for-f-in-H_dagger,J}
\varphi(\bm{z}+it\div(a)) \ll_K (1+\norm{\bm{z}})^A (1+t^2)^{-1}
\quad\textnormal{for $(\Re(\bm{z}),t)\in K^J\times \RR$}.
\end{equation}
Let $\mathcal{H}_J \defeq \bigcup_{\delta>0} \mathcal{H}_J(-\delta,\delta)$,
and similarly define $\mathcal{H}^\CC_J$ and $\mathcal{H}_{\dagger,J}$.
Let $\mathcal{M}_{\dagger,J}$ be the set of functions $\digamma(\bm{z})$
(meromorphic on a neighborhood of the line $\Re(\bm{z})=\bm{0}$)
such that
$$\digamma(\bm{z}) \prod_{\jboundary\in J} \frac{z_\jboundary}{1+z_\jboundary}
\in \mathcal{H}_{\dagger,J}.$$
Both $\mathcal{H}_{\dagger,J}$ and $\mathcal{M}_{\dagger,J}$
are $\mathcal{H}^\CC_J$-modules.

The following lemma expresses $\mathscr{S}(\digamma)$ in terms of $\mathcal{X}$-functions, for any $\digamma\in \mathcal{M}_{\dagger,J}$ (in the spirit of \cite{chambert2010integral}*{(3.5.2)}, which builds on \cite{chambert2001fonctions}*{Th\'{e}or\`{e}me~3.1.14}).
For a set $I\belongs J$,
let $$\mathbf{g}_I(\bm{z})\defeq \mathcal{X}_{\RR_{\ge 0}^I}(\bm{z})
= \prod_{\jboundary\in I} z_\jboundary^{-1}$$
if $\card{I\setminus J_3}\ne 1$,
and if $\card{I\setminus J_3} = 1$ then let $$\mathbf{g}_I(\bm{z})\defeq \frac{\mathcal{X}_{\RR_{\ge 0}^I}(\bm{z})}{\prod_{\jboundary\in I} (1+z_\jboundary)}
= \prod_{\jboundary\in I} \frac{1}{z_\jboundary(1+z_\jboundary)}.$$


\begin{lemma}
\label{LEM:shifted-integral-polar-and-residue-structure}
Let $P\belongs J$ be a set.
Suppose $\digamma\in \mathcal{M}_{\dagger,J}$
and $$\digamma(\bm{z}) \prod_{\jboundary\in P} \frac{z_\jboundary}{1+z_\jboundary}
\in \mathcal{H}_{\dagger,J}.$$
\begin{enumerate}
\item There exists a vector $(h_I)_{I\belongs P}\in \prod_{I\belongs P} \mathcal{H}_J$,
with $h_I\in \mathcal{H}_{\dagger,J}$ for $I\belongs J_3$
and $h_I\in \mathcal{H}^\CC_J$
for all other $I$,
such that we have an identity of meromorphic functions
\begin{equation}
\label{decompose-f11}
\digamma = \sum_{I\belongs P} h_I \mathbf{g}_I
\end{equation}
(on a neighborhood of the line $\Re(\bm{z})=\bm{0}$).


\item
We have $h_P(\bm{0}) = (\digamma \prod_{\jboundary\in P} z_\jboundary)(\bm{0})$
and $$\mathscr{S}(\digamma)
= \sum_{I\belongs P:\, \card{I\setminus J_3}\ge 1} h_I \mathscr{S}(\mathbf{g}_I)
+ \sum_{I\belongs P\cap J_3} \mathbf{g}_I \mathscr{S}(h_I).$$

\item We have $\mathscr{S}(\mathbf{g}_I) = 0$
if $\card{I\setminus J_3} = 1$,
and $\mathscr{S}(\mathbf{g}_I) = \mathcal{X}_{\Lambda_I(X)}$
if $\card{I\setminus J_3}\ge 2$.

\item We have $$\mathscr{S}(\mathcal{H}_{\dagger,J}(A_1,A_2))
\belongs \mathcal{H}^\CC_J(A_1,A_2).$$
That is, for any function $\varphi\in \mathcal{H}_{\dagger,J}(A_1,A_2)$,
we have $\mathscr{S}(\varphi)\in \mathcal{H}^\CC_J(A_1,A_2)$.
\end{enumerate}
\end{lemma}

\begin{proof}
(1):
Let $2^P$ denote the set of all subsets $I\belongs P$.
We construct $(h_I)_{I\belongs P}$ by the following recursive algorithm:
\begin{enumerate}[(a)]
\item Let $h_P\defeq \digamma/\mathbf{g}_P\in \mathcal{H}_J$, and $h_I\defeq 0\in \mathcal{H}_J$ for all $I\in 2^P\setminus \set{P}$;
and let $\mathscr{I}\defeq 2^P$.
Note that $h_I\mathbf{g}_I\belongs \mathcal{M}_{\dagger,J}$ for all $I\in 2^P$; we will maintain this property at each step.

\item If $\mathscr{I} = \emptyset$, terminate the algorithm.
Otherwise, choose an element $I_\star\in \mathscr{I}$ of maximal cardinality.

\item Suppose first that $\card{I_\star\setminus J_3}\ge 1$.
Choose an element $i_\star\in I_\star\setminus J_3$,
and replace the functions $h_{I_\star}, h_{I_\star\setminus \set{i_\star}}\in \mathcal{H}_J$ with, respectively,
\begin{equation*}
h'_{I_\star}((z_\jboundary)_{\jboundary\in J}) \defeq
h_{I_\star}((z_\jboundary - \mathsf{u}_\jboundary z_{i_\star}/\mathsf{u}_{i_\star})_{\jboundary\in J})\in \mathcal{H}^\CC_J
\end{equation*}
and $h'_{I_\star\setminus \set{i_\star}}\defeq h_{I_\star\setminus \set{i_\star}} + (h_{I_\star} - h'_{I_\star}) \mathbf{g}_{I_\star}/\mathbf{g}_{I_\star\setminus \set{i_\star}} \in \mathcal{H}_J$.\footnote{One can use Cauchy's integral formula to prove $(h_{I_\star} - h'_{I_\star})/z_{i_\star}\in \mathcal{H}_J$,
which implies $h'_{I_\star\setminus \set{i_\star}}\in \mathcal{H}_J$.}
Note that $\mathbf{g}_{I_\star}\in \mathcal{M}_{\dagger,J}$, so
\begin{equation*}
h'_{I_\star} \mathbf{g}_{I_\star}\in \mathcal{M}_{\dagger,J},
\quad h'_{I_\star\setminus \set{i_\star}} \mathbf{g}_{I_\star\setminus \set{i_\star}}
= h_{I_\star\setminus \set{i_\star}} \mathbf{g}_{I_\star\setminus \set{i_\star}}
+ h_{I_\star} \mathbf{g}_{I_\star} - h'_{I_\star} \mathbf{g}_{I_\star}\in \mathcal{M}_{\dagger,J}.
\end{equation*}
Now replace $\mathscr{I}$ with $\mathscr{I}'\defeq \mathscr{I}\setminus \set{I_\star}$, and go back to step~(b).

\item Suppose instead that $I_\star\belongs J_3$.
Then $1/\mathbf{g}_{I_\star}\in \mathcal{H}^\CC_J$, which together with $h_{I_\star} \mathbf{g}_{I_\star}\in \mathcal{M}_{\dagger,J}$ implies $h_{I_\star}\in \mathcal{M}_{\dagger,J}\cap \mathcal{H}_J = \mathcal{H}_{\dagger,J}$.\footnote{One can use Cauchy's integral formula to show that $\mathcal{M}_{\dagger,J}\cap \mathcal{H}_J = \mathcal{H}_{\dagger,J}$.}
Replace $\mathscr{I}$ with $\mathscr{I}'\defeq \mathscr{I}\setminus \set{I_\star}$, and repeat (b).
\end{enumerate}
The algorithm terminates after $1 + \card{2^P}$ occurrences of (b).
For the final output $(h_I)_{I\belongs P}$,
the conditions of (1) hold by inspection of (c) and (d).



(2):
For the first part, multiply \eqref{decompose-f11} by $\prod_{\jboundary\in P} z_\jboundary$ and plug in $\bm{z}=\bm{0}$.
For the second part, use the $\CC\div(a)$-invariance of $h_I$ when $\card{I\setminus J_3}\ge 1$, and of $\mathbf{g}_I$ when $I\belongs J_3$.

(3):
This follows from \cite{chambert2001fonctions}*{Proposition~3.1.9}
if $\card{I\setminus J_3}\ge 2$ and $\mathcal{X}_{\Lambda_I(X)} \ne 0$
(in which case $\RR_{\ge 0}^I \cap (\RR\div(a)\vert_I) = 0$,
by Remark~\ref{RMK:X-function-vanishing-criteria}),
and from the identity
\begin{equation}
\label{EQN:purely-positive-shifted-integral-vanishes}
\int_\RR \frac{dt}{\prod_{1\le k\le r} (z'_k+it)} = 0
\quad\textnormal{for $r\ge 2$ and $\Re(z'_k)>0$}
\end{equation}
if $\card{I\setminus J_3} = 1$ or $\mathcal{X}_{\Lambda_I(X)} = 0$.
(To prove \eqref{EQN:purely-positive-shifted-integral-vanishes},
write $s=it$ to view the integral over $\RR$
as a contour integral over a vertical line $\Re(s)=0$ in $\CC$,
and shift the contour to $\Re(s)\to \infty$.)

(4):
This follows from \eqref{EQN:define-shifted-integral}
and \eqref{INEQ:growth-decay-condition-for-f-in-H_dagger,J},
since $\int_\RR (1+t^2)^{-1}\, dt
\ll 1$.
\end{proof}

\subsection{Peyre's constant}

Cf.~\cites{peyre1995hauteurs,salberger1998tamagawa}.
First, $$\alpha_\textnormal{Pey}(X)
\defeq \mathcal{X}_{\Pic(X)}(K_X^{-1})=\mathcal{X}_{\Lambda_J(X)}(K_X^{-1});$$
by Remark~\ref{RMK:saturated-case-dual-measure-interpretation}, the measure $d\varkappa$ in $\mathcal{X}_{\Lambda_J(X)}$ is correctly normalized in terms of $\Pic(X)$.
Here $\alpha_\textnormal{Pey}(X)>0$ by Proposition~\ref{PROP:poles-of-left-invariant-top-form}.
Second, $\beta_\textnormal{Pey}(X) = 1$ since $X$ is rational; see e.g.~\cite{tanimoto2012distribution}*{two lines before Theorem~48}.
Let $$\abs{\omega}_v\defeq \abs{a_v}_v^{-1} \, da_v\, db_v$$
be the measure on $G(\QQ_v)$ associated to
the top form $\omega = a^{-1}\, da\, db$.
Then $$\int_{\mathscr{G}(\ZZ_p)} \abs{\omega}_p
= \int_{\ZZ_p^\times} \abs{a_p}_p^{-1} \, da_p \int_{\ZZ_p} db_p
= \int_{\ZZ_p^\times} da_p
= 1-p^{-1}$$
for any prime $p$,
where the $\mathscr{G}$ is the obvious integral model of $G$ over $\ZZ$.
(As a scheme, $\mathscr{G}\cong \GG_m \times \Aff^1_\ZZ$.)
Finally,
define the local Tamagawa measure
\begin{equation}
\label{EQN:define-local-Tamagawa-measure}
d\tau_v=d\tau_{X,\mathsf{H},v}\defeq H_{\div(\omega),v} \cdot \abs{\omega}_v
\end{equation}
on $X(\QQ_v)$;
initially the formula \eqref{EQN:define-local-Tamagawa-measure} only makes sense on $G(\QQ_v)$, but it extends to $X(\QQ_v)$ by Definition~\ref{DEFN:standard-local-Weil-height-for-boundary-divisors}.
Let $L(s,\Pic(X_{\ol{\QQ}}))$ be the Artin $L$-function associated to the representation $\Pic(X_{\ol{\QQ}})\otimes \CC$ of $\Gal(\ol{\QQ}/\QQ)$.
It is known that $L(s,\Pic(X_{\ol{\QQ}}))$ converges absolutely on $\Re(s)>1$, is meromorphic on $\CC$, and has a pole of order $\rank(\Pic(X))\ge 1$ at $s=1$;
let $$L^\ast(1,\Pic(X_{\ol{\QQ}}))
\defeq \lim_{s\to 1}{\frac{L(s,\Pic(X_{\ol{\QQ}}))}{(s-1)^{\rank(\Pic(X))}}}.$$
Define on $X(\adele_\QQ)$ the measure
\begin{equation*}
d\tau\defeq L^\ast(1,\Pic(X_{\ol{\QQ}}))\, d\tau_\infty
\prod_p \frac{d\tau_p}{L_p(1,\Pic(X_{\ol{\QQ}}))}.
\end{equation*}
It is known that $d\tau$
(unlike the local measures $d\tau_v$)
depends only on $X$ and $H_{K_X^{-1}}=\mathsf{H}$,
and not on the top form $\omega$ or the local heights $H_{\div(\omega),v}$.
Let $$\tau(X,\mathsf{H}) \defeq \int_{X(\adele_\QQ)} d\tau,$$
which (by the Weil conjectures) factors as an absolutely convergent Euler product;
see e.g.~\cite{chambert2010igusa}*{Corollary~2.4 and Theorem~2.5}.
Peyre's constant is
\begin{equation}
\label{EQN:Peyre-constant}
\mathcal{A}_{X,\mathsf{H}}
\defeq \frac{\alpha_\textnormal{Pey}(X) \tau(X,\mathsf{H})}{(\rank(\Pic(X))-1)!}
> 0.
\end{equation}

\subsection{Spectral expansion}

For $\bm{s} = \sum_{\jboundary\in J} s_\jboundary D_\jboundary
\in \Pic^G(X) \otimes \CC$
and $g=(g_v)_v\in G(\adele_\QQ)$, let
\begin{equation*}
H_v(\bm{s},g_v) \defeq \prod_{\jboundary\in J} H_{D_\jboundary,v}(g_v)^{s_\jboundary},
\quad H(\bm{s},g)\defeq \prod_v H_v(\bm{s},g_v),
\quad \mathsf{Z}(\bm{s},g) \defeq \sum_{g'\in G(\QQ)} H(\bm{s},g' g)^{-1}.
\end{equation*}
For each prime $p$, choose a maximal open subgroup $\Kbad_p$ of $\mathscr{G}(\ZZ_p)$ such that $H_{D_\jboundary,p}$ is right $\Kbad_p$-invariant for all $\jboundary\in J$;
then $\Kbad_p = \mathscr{G}(\ZZ_p)$ for all but finitely many $p$.
Let $\Kbad \defeq \prod_p \Kbad_p$.
Following \cite{tanimoto2012height}*{\S5}, we will decompose $\mathsf{Z}(\bm{s},g)$
using the automorphic machinery of \cite{tanimoto2012height}*{\S3}.
Since we work more generally than \cite{tanimoto2012height}*{\S5}, it seems appropriate to provide some details;
but the key formula \eqref{EQN:difficult-part-of-spectral-expansion} is not new (only our subsequent analysis of it is).


Let $dg=\prod_v dg_v$ be the measure on $G(\adele_\QQ)$ given by
$dg_\infty = \abs{\omega}_\infty$
and $dg_p = \abs{\omega}_p/(1-p^{-1})$.
Let $d^\times{a}=\prod_v d^\times{a_v}$
be the Haar measure on $\adele_\QQ^\times$
with $d^\times{a_\infty} = a_\infty/\abs{a_\infty}_\infty$
and $\int_{\ZZ_p^\times} d^\times{a_p} = 1$.
Then $dg=d^\times{a}\,db$
and $\int_{\mathscr{G}(\ZZ_p)} dg_p = 1$.
Also, for $p$ sufficiently large,
Definition~\ref{DEFN:standard-local-Weil-height-for-boundary-divisors} implies
\begin{equation}
\label{large-p-GZp-height}
H_p(\bm{s},g_p) = 1
\quad\textnormal{for all $g_p\in \mathscr{G}(\ZZ_p)$},
\end{equation}
whence $\int_{\mathscr{G}(\ZZ_p)} H_p(\bm{s},g_p)^{-1}\, dg_p = 1$.

For the rest of \S\ref{SEC:background}, let
$A_5>0$ be a large real constant
(depending on $\div(a)$ and $\div_\infty(b)$).
If $\Re(\bm{s})\in [A_5,2A_5]^J$
and $g\in G(\adele_\QQ)$, then
\begin{equation}
\begin{split}
\label{INEQ:crude-lower-bound-on-H}
\abs{H(\bm{s},g)}
&\gg H(\div_0(a)+\div_\infty(a)+\div_\infty(b),g)^{A_5^{1/2}} \\
&\gg \prod_v \max(1,\abs{a(g_v)}_v^{-1},\abs{a(g_v)}_v,\abs{b(g_v)}_v)^{A_5^{1/2}},
\end{split}
\end{equation}
by Definition~\ref{DEFN:standard-local-Weil-height-for-boundary-divisors};
cf.~ \cite{tanimoto2012height}*{proof of Lemma~5.2}.
For unitary characters $\chi$ of $\QQ^\times\backslash \adele_\QQ^\times$, let
\begin{equation*}
H^\ast(\bm{s},\chi)
\defeq \int_{G(\adele_\QQ)} H(\bm{s},g)^{-1} \ol{\chi}(\det(g))\, dg,
\quad H^\ast_v(\bm{s},\chi)
\defeq \int_{G(\QQ_v)} H_v(\bm{s},g_v)^{-1} \ol{\chi}_v(\det(g_v))\, dg_v.
\end{equation*}
If $\Re(\bm{s})\in [A_5,2A_5]^J$, then by \eqref{INEQ:crude-lower-bound-on-H}, we have $H^\ast_v(\bm{s},\chi) \ll 1$ for all $v$, and $H^\ast_p(\bm{s},\chi) = 1+O(p^{-2})$ for all $p$ large enough in terms of the conductor of $\chi$;
so $H^\ast(\bm{s},\chi) = \prod_v H^\ast_v(\bm{s},\chi) \ll_\chi 1$.

For the rest of \S\ref{SEC:background}, assume $X$ is split.
(Splitness simplifies many formulas.)

\begin{lemma}
\label{LEM:snc-implies-nice-densities-implies-Tamagawa-extraction}
Let $\delta>0$ be small.
Then each of the following implies the next:
\begin{enumerate}
\item The divisor $D$ has strict normal crossings (see \S\ref{SUBSEC:conventions}).

\item Let $\Re(\bm{s}+\div(\omega)) \in [-\delta,\delta^{-1}]^J$.
Then for each $v$, the integral $H^\ast_v(\bm{s},1)$ converges absolutely uniformly over $\bm{s}$.
Moreover, $H^\ast_p(\bm{s},1) / \prod_{\jboundary\in J} \zeta_p(s_\jboundary-\mathsf{d}_\jboundary+1) = 1+O(p^{-1-\delta})$.

\item We have $\lim_{\bm{s}\to -\div(\omega)}{H^\ast(\bm{s},1)
\prod_{\jboundary\in J}(s_\jboundary-\mathsf{d}_\jboundary)}
= \tau(X,\mathsf{H})$.
\end{enumerate}
\end{lemma}

\begin{proof}

(1)$\Rightarrow$(2):
See \cite{chambert2010igusa}*{\S4.3.2 (or Lemma~4.1)} for the convergence of individual integrals,
and \cite{chambert2010igusa}*{\S4.3.3 with $\mathscr{B}=\emptyset$} (based on a formula of Denef) for the asymptotic.

(2)$\Rightarrow$(3):
Write $H^\ast = \prod_v H^\ast_v$;
switch $\lim$, $\prod$ using (2).
Then express $\tau(X,\mathsf{H})$ in terms of $H^\ast_v$ via \eqref{EQN:define-local-Tamagawa-measure}.
Cf.~\cite{chambert2002distribution}*{proof of Proposition~6.2} or \cite{chambert2010igusa}*{proof of Proposition~4.10}.
\end{proof}

For the rest of \S\ref{SEC:background}, assume $D$ has strict normal crossings.
(This condition makes some technical calculations more uniform.)
To decompose $\mathsf{Z}(\bm{s},g)$ requires not just $H^\ast(\bm{s},\chi)$ but other integral transforms as well, which we now define alongside other relevant notions.

First, let $\Sbad\belongs \Spec(\ZZ)$ be a finite set such that the following hold:
\begin{enumerate}
\item If $p\notin \Sbad$ and $\jboundary\in J$, then $H_{D_\jboundary,p}$ satisfies Definition~\ref{DEFN:standard-local-Weil-height-for-boundary-divisors} with $\Delta_p=1$.

\item If $p\notin \Sbad$ and $\jboundary\in J$, then $H_{D_\jboundary,p}$ is $\mathscr{G}(\ZZ_p)$-invariant (on the right), i.e.~$\Kbad_p=\mathscr{G}(\ZZ_p)$.

\item If $I\belongs J$, then $\bigcap_{\jboundary\in I} \mathscr{D}_\jboundary$ is smooth over $\ZZ[1/\Sbad]$, the ring of $\Sbad$-integers.
Here $\mathscr{D}_\jboundary$ denotes the closure of $D_\jboundary$ in $\mathscr{X}$,
and for $I=\emptyset$ we let $\bigcap_{\jboundary\in I} \mathscr{D}_\jboundary \defeq \mathscr{X}$.
\end{enumerate}
(We can arrange for (3) to hold since $X$ is smooth and $D$ has strict normal crossings.)

For $p\in \Sbad$, let $\rbad_p(\Kbad)
\defeq \min{\set{r\in \ZZ_{\ge 1}:
(1+p^r\ZZ_p) \times p^r\ZZ_p \belongs \Kbad_p}}$.
Let $\Nbad \defeq \prod_{p\in \Sbad} p^{\rbad_p(\Kbad)}$, and let $\Mbad$ be the set of characters
\begin{equation}
\label{relevant-badness-characters}
\lambda\maps (\RR^\times/(\RR^\times)^2)
\times \prod_{p\in \Sbad} (\ZZ_p^\times/(1+p^{\rbad_p(\Kbad)}\ZZ_p))
\to \CC^\times.
\end{equation}
For $a'=(a'_v)_v\in \RR^\times \times \prod_p \ZZ_p^\times$,
write $\lambda(a') = \prod_v \lambda_v(a'_v)$,
where $\lambda_\infty\maps \RR^\times/(\RR^\times)^2\to \CC^\times$
and $\lambda_p\maps \ZZ_p^\times\to \CC^\times$ are characters
such that $\lambda_p=1$ for all $p\notin \Sbad$.
For $a_p\in \QQ_p^\times$,
let $\sgn_p(a_p) \defeq a_p/p^{v_p(a_p)}$.
Also, for $a=(a_v)_v\in \adele_\QQ^\times$,
let $\lambda_\Sbad(a) \defeq \prod_p \lambda_{\Sbad,p}(a_p)$, where
\begin{equation}
\label{lambda-Sp-definition}
\lambda_{\Sbad,p}(a_p) \defeq \prod_{p'\ne p} \lambda_{p'}(p^{v_p(a_p)})
= \prod_{p'\in \Sbad\setminus \{p\}} \lambda_{p'}(p^{v_p(a_p)}).
\end{equation}
We are now prepared to define the key components $\theta_{m,\lambda,t}(g)$
of a non-abelian Fourier transform on $G(\QQ)\backslash G(\adele_\QQ)$.
For $(m,\lambda,t)\in \ZZ_{\ge 1} \times \Mbad \times \RR$
and $g=(a,b)\in G(\adele_\QQ)$, let
\begin{equation}
\label{clever-theta-sum-over-Qcross}
\theta_{m,\lambda,t}(g) \defeq \sum_{\alpha\in \QQ^\times} \psi(\alpha b)
\mathbf{v}_{m,\lambda}(\alpha a)\, \abs{\alpha a_\infty}_\infty^{it},
\end{equation}
where $\psi$ is the additive character from \S\ref{SUBSEC:conventions},
and where for $a=(a_v)_v\in \adele_\QQ^\times$ we let
\begin{equation}
\label{EQN:define-finite-vector-v_m,lambda}
\mathbf{v}_{m,\lambda}(a)
\defeq \lambda_\infty(a_\infty)
\prod_p \lambda_p(\sgn_p(a_p)) \bm{1}_{\Nbad a_p/m\in \ZZ_p^\times}.
\end{equation}
By additive reciprocity ($\psi(\QQ) = 1$),
the function $\theta_{m,\lambda,t}$ is left $G(\QQ)$-invariant.
Furthermore, $\abs{\theta_{m,\lambda,t}(g)} \le 2$ for all $g$,
since the sum in \eqref{clever-theta-sum-over-Qcross} is supported on
$\alpha = \pm \prod_p p^{v_p(m/\Nbad a_p)} \in \QQ^\times$.
Moreover, at $1_G\defeq (1,0)$ we have
\begin{equation*}
\begin{split}
\theta_{m,\lambda,t}(1_G)
&= \sum_{\alpha = \pm m/\Nbad }
\mathbf{v}_{m,\lambda}(\alpha)\, \abs{\alpha}_\infty^{it} \\
&= \abs{m/\Nbad }_\infty^{it} (1+\lambda(-1))
\prod_p \lambda_p(\sgn_p(m/\Nbad )) \\
&= \abs{m/\Nbad }_\infty^{it} (1+\lambda(-1))
\prod_p \prod_{p'\ne p} \lambda_p((p')^{v_{p'}(m/\Nbad )}) \\
&= \abs{m/\Nbad }_\infty^{it} (1+\lambda(-1))
\prod_{p'} \lambda_{\Sbad,p'}(m/\Nbad ),
\end{split}
\end{equation*}
by the definition of $\lambda_{\Sbad,p'}$.
By the definition of $\lambda_\Sbad$, it follows that
\begin{equation}
\label{EQN:formula-for-theta-at-1}
\theta_{m,\lambda,t}(1_G)
= \abs{m/\Nbad }_\infty^{it} (1+\lambda(-1))
\lambda_\Sbad(m/\Nbad )
= \frac{(1+\lambda(-1)) \lambda_\Sbad(m/\Nbad )}{\prod_p \abs{m/\Nbad }_p^{it}},
\end{equation}
where the last step follows from the product formula
$\prod_v \abs{m/\Nbad }_v^{it} = 1$.
Let
\begin{equation}
\label{EQN:define-non-abelian-Fourier-coefficient-H^ast}
H^\ast(\bm{s},m,\lambda,t)
\defeq \int_{G(\adele_\QQ)} H(\bm{s},g)^{-1} \ol{\theta}_{m,\lambda,t}(g)\, dg.
\end{equation}


For $\alpha\in \QQ^\times$,
define $H^\vee_p(\bm{s},\lambda,t,\alpha)$ to be
\begin{equation}
\label{EXPR:formula-for-H^vee_p}
\int_{G(\QQ_p):\, \Nbad \alpha a_p\in \ZZ_p} H_p(\bm{s},g_p)^{-1}
\phase(-\alpha b_p\bmod{\ZZ_p}) \ol{\lambda}_p(\sgn_p(\alpha a_p))
\lambda_{\Sbad,p}(\alpha a_p) \abs{a_p}_p^{-it}\, dg_p.
\end{equation}
Let $H^\vee_\infty(\bm{s},\lambda,t,\alpha)
\defeq \int_{G(\RR)} H_\infty(\bm{s},g_\infty)^{-1} \phase(\alpha b_\infty)
\lambda_\infty(\alpha a_\infty) \abs{a_\infty}_\infty^{-it} \, dg_\infty$.
(Note that $\lambda_\infty=\ol{\lambda}_\infty$.)

Let $H'_\infty(\bm{s},\lambda,t,\alpha)
\defeq \abs{\alpha}_\infty^{-it} H^\vee_\infty(\bm{s},\lambda,t,\alpha)$.
For $m\in \ZZ_{\ge 1}$, define $H'_p(\bm{s},\lambda,m,\alpha)$ to be
\begin{equation}
\label{EXPR:formula-for-H'_p}
\int_{G(\QQ_p):\, \Nbad \alpha a_p\in m\ZZ_p^\times} H_p(\bm{s},g_p)^{-1}
\phase(-\alpha b_p\bmod{\ZZ_p}) \ol{\lambda}_p(\sgn_p(\alpha a_p))\, dg_p.
\end{equation}

Before proceeding, we need some foundational bounds.
Let $\norm{\bm{s}}\defeq \max_{\jboundary\in J}{\abs{s_\jboundary}}$.

\begin{proposition}
[Cf.~\cite{tanimoto2012height}*{Lemma~5.2}]
\label{PROP:large-s-absolute-convergence-properties}
Let $\Re(\bm{s})\in [A_5,2A_5]^J$ and $g\in G(\adele_\QQ)$.
Then $\mathsf{Z}(\bm{s},g)$ converges absolutely to a bounded function of $(\bm{s},g)$, holomorphic in $\bm{s}$ and smooth in $g$.
Also, $H(\bm{s},g)^{-1}\in L^1(G(\adele_\QQ))$ and $\mathsf{Z}(\bm{s},g)\in L^q(G(\QQ)\backslash G(\adele_\QQ))^{\Kbad}$ for all $\bm{s}$, for all $q\in [1,\infty]$.
\end{proposition}

\begin{proof}

Except for smoothness in $g$, this follows from \eqref{INEQ:crude-lower-bound-on-H} as in \cite{tanimoto2012height}*{proof of Lemma~5.2}.
For smoothness in $g$, use $\Kbad$-invariance at finite places,
and at the infinite place use the fact that
for every composition $\Top$ of operators in $\set{a\,\frac{\partial}{\partial a}, a\,\frac{\partial}{\partial b}}$, we have
\begin{equation}
\label{T-stable-derivatives}
\Top(H_{D_\jboundary,\infty}^{-1})
\ll_\Top H_{D_\jboundary,\infty}^{-1},
\end{equation}
thanks to Lemma~\ref{LEM:localized-regular-and-height-derivative-estimates}(2)
(or \cite{tanimoto2012distribution}*{proof of Lemma~5.9}),
together with left invariance of $\Top$,
i.e.~the property $\Top(f(g' g)) = (\Top{f})(g' g)$ for $g'\in G(\RR)$.
\end{proof}

\begin{lemma}
\label{LEM:large-s-real-IBP-bounds}
Suppose $\Re(\bm{s})\in [A_5,2A_5]^J$ and $\alpha, t\in \RR$.
Then the following hold:
\begin{enumerate}
\item For $\varsigma\in \set{\pm 1}$, we have
$$\int_{G(\RR):\, \sgn(a_\infty)=\varsigma} H_\infty(\bm{s},g_\infty)^{-1}
\phase(\alpha b_\infty) \abs{a_\infty}_\infty^{-it}\, dg_\infty
\ll \frac{1+\norm{\bm{s}}^6}{(1+\alpha^4)(1+t^2)}.$$

\item We have $\int_{\RR^\times} \abs{\int_\RR H_\infty(\bm{s},g_\infty)^{-1}
\phase(\alpha b_\infty)\, db_\infty}\, d^\times{a_\infty}
\ll (1+\norm{\bm{s}}^4) / (1+\alpha^4)$.
\end{enumerate}
\end{lemma}

\begin{proof}
For (1) and (2), integrate by parts over $b_\infty$ four times if $\abs{\alpha} \ge 1$.
For (1), further integrate by parts over $\log{\abs{a_\infty}}$ twice if $\abs{t} \ge 1$.
To bound derivatives, use \eqref{T-stable-derivatives};
to bound $H^{-1}$ itself, use \eqref{INEQ:crude-lower-bound-on-H} and the largeness of $A_5$.
\end{proof}

\begin{lemma}
\label{LEM:large-s-absolute-p-adic-bounds}
Suppose $\Re(\bm{s})\in [A_5,2A_5]^J$ and $(m, \alpha) \in \ZZ_{\ge 1}\times \QQ^\times$.
Then the following hold:
\begin{enumerate}
\item We have $H'_p(\bm{s},\lambda,m,\alpha)
= \bm{1}_{v_p(m)=v_p(\alpha)} + O((p+\abs{m/\alpha}_p+\abs{\alpha/m}_p)^{-4})$.

\item We have $\int_{G(\QQ_p):\, \Nbad \alpha a_p\in \ZZ_p}
\abs{H_p(\bm{s},g_p)}^{-1}\, dg_p
= \bm{1}_{v_p(\alpha)\ge 0} + O((p+\abs{\alpha}_p)^{-4})$.
\end{enumerate}
\end{lemma}

\begin{proof}
Use \eqref{INEQ:crude-lower-bound-on-H} and the fact \eqref{large-p-GZp-height} for large $p$.
\end{proof}

\begin{lemma}
\label{LEM:large-s-technical-adelic-convergence-bounds}
Suppose $\Re(\bm{s})\in [A_5,2A_5]^J$ and $t\in \RR$.
Then the following hold:
\begin{enumerate}
\item We have $${\textstyle
\sum_{\alpha\in \QQ^\times} \sum_{\lambda\in \Mbad} \sum_{m\ge 1}
\abs{H'_\infty(\bm{s},\lambda,t,\alpha)
\prod_p H'_p(\bm{s},\lambda,m,\alpha)}
\ll (1+\norm{\bm{s}}^6) / (1+t^2).}$$

\item We have $${\textstyle \sum_{\alpha\in \QQ^\times}
(\int_{\RR^\times} \abs{\int_\RR H_\infty(\bm{s},g_\infty)^{-1}
\phase(\alpha b_\infty)\, db_\infty}\, d^\times{a_\infty})
\prod_p (\int_{G(\QQ_p):\, \Nbad \alpha a_p\in \ZZ_p}
\abs{H_p(\bm{s},g_p)}^{-1}\, dg_p)
< \infty.}$$
\end{enumerate}
\end{lemma}

\begin{proof}
(1):
Let $\alpha'\defeq \alpha/m$.
Now plug in Lemmas~\ref{LEM:large-s-real-IBP-bounds}(1)
and~\ref{LEM:large-s-absolute-p-adic-bounds}(1).
It then suffices to prove $\sum_{m\ge 1} \sum_{\alpha'=r_1/r_2\in \QQ^\times} (1+(m\alpha')^4)^{-1} \cdot \abs{r_1r_2}^{\eps-4} < \infty$, where we write $\alpha' = r_1/r_2$ in lowest terms (with $r_2\ge 1$).
But $(1+(m\alpha')^4) \cdot (r_1r_2)^4 \ge (r_2^4+m^4) \cdot r_1^4$, which suffices.


(2):
Use Lemmas~\ref{LEM:large-s-real-IBP-bounds}(2) and~\ref{LEM:large-s-absolute-p-adic-bounds}(2), and the bound $\sum_{\alpha=r_1/r_2\in \QQ^\times} (1+\alpha^4)^{-1} \abs{r_2}^{\eps-4} < \infty$.
\end{proof}

For the rest of \S\ref{SEC:background}, assume $\Re(\bm{s})\in [A_5,2A_5]^J$.
We will decompose $\mathsf{Z}(\bm{s},g)$ into two pieces,
based on an averaging process over $\QQ\backslash \adele_\QQ$,
using the short exact sequence of \emph{pointed sets}
(noting that $G(\QQ)\backslash G(\adele_\QQ)$ is not actually a group)
\begin{equation}
\label{adelic-SES}
1 \to \QQ\backslash \adele_\QQ
\xrightarrow{b\mapsto (1,b)} G(\QQ)\backslash G(\adele_\QQ)
\xrightarrow{\det\maps (a,b)\mapsto a} \QQ^\times\backslash \adele_\QQ^\times
\to 1,
\end{equation}
where the map $\QQ\backslash \adele_\QQ \to G(\QQ)\backslash G(\adele_\QQ)$
is well-defined since $(1,c+b) = (1,c)(1,b)$
for all $b\in \adele_\QQ$ and $c\in \QQ$.
In addition, if $\phi\in C(G(\QQ)\backslash G(\adele_\QQ))$,
then $\phi(a,b)$ is $\QQ$-translation invariant in the $b$ entry,
thanks to the translation identity $(1,c)(a,b) = (a,c+b)$ for $c\in \QQ$.
For $g\in G(\adele_\QQ)$, following \cite{tanimoto2012distribution}*{\S3.2.2}
(or \cite{tanimoto2012height}*{Lemma~3.3}), we now let
\begin{equation}
\label{EQN:define-Z_0-as-b-average}
\mathsf{Z}_0(\bm{s},g)
\defeq \int_{\QQ\backslash \adele_\QQ} \mathsf{Z}(\bm{s},(\det(g),b'))\, db',
\end{equation}
where $\mathsf{Z}(\bm{s},(\det(g),b'))$ is well-defined
because $\mathsf{Z}$ is left $G(\QQ)$-invariant.
For any $g'\in G(\QQ)$,
the scaling identity $(\det(g'),1)(\det(g),b') = (\det(g'g),b'\det(g'))$ implies that
\begin{equation*}
\mathsf{Z}_0(\bm{s},g)
= \int_{\QQ\backslash \adele_\QQ} \mathsf{Z}(\bm{s},(\det(g'g),b'\det(g')))\, db'
= \int_{\QQ\backslash \adele_\QQ} \mathsf{Z}(\bm{s},(\det(g'g),b'))\, db'
= \mathsf{Z}_0(\bm{s},g'g),
\end{equation*}
since $\mathsf{Z}$ is left $G(\QQ)$-invariant
and multiplication by $\det(g')\in \QQ$
on the $\QQ$-vector space $\QQ\backslash \adele_\QQ$
is measure-preserving.
Thus $\mathsf{Z}_0$ is left $G(\QQ)$-invariant.
In fact, $\mathsf{Z}_0$ is also right $\Kbad$-invariant,
because if $k=(k_1,k_2)\in \Kbad$,
then the identity $(\det(g),b')k = (\det(gk),\det(g)k_2+b')$
and right $\Kbad$-invariance of $\mathsf{Z}$ imply that
\begin{equation*}
\mathsf{Z}_0(\bm{s},g)
= \int_{\QQ\backslash \adele_\QQ} \mathsf{Z}(\bm{s},(\det(g),b')k)\, db'
= \int_{\QQ\backslash \adele_\QQ} \mathsf{Z}(\bm{s},(\det(gk),b'))\, db'
= \mathsf{Z}_0(\bm{s},gk).
\end{equation*}

We have $\mathsf{Z}_0\in (C^\infty\cap L^\infty)(G(\adele_\QQ))$,
because $\mathsf{Z}\in (C^\infty\cap L^\infty)(G(\adele_\QQ))$
(by Proposition~\ref{PROP:large-s-absolute-convergence-properties})
and $\QQ\backslash \adele_\QQ$ is compact.
Also, $\mathsf{Z}_0$ is left $\adele_\QQ$-invariant, and thus descends to a (right) $\det(\Kbad)$-invariant function of $\det(g)\in \QQ^\times\backslash \adele_\QQ^\times$,
where $\det(\Kbad)$ denotes the image of $\Kbad$
under the map $\det\maps G(\adele_\QQ) \to \adele_\QQ^\times$.
Fourier analysis on $C(\QQ^\times\backslash \adele_\QQ^\times)^{\det(\Kbad)}$,
as in \cite{tanimoto2012distribution}*{\S3.2.2},
gives
\begin{equation}
\label{EQN:Fourier-expand-Z_0}
\mathsf{Z}_0(\bm{s},g)
= \mathsf{Z}_0(\bm{s},(\det(g),0))
= \int H^\ast(\bm{s},\chi) \chi(\det(g))\, d\chi,
\end{equation}
where $\chi$ runs over characters of $\adele_\QQ^\times$ lying in $L^\infty(\QQ^\times\backslash \adele_\QQ^\times)^{\det(\Kbad)}$.
Before defining $d\chi$ and justifying \eqref{EQN:Fourier-expand-Z_0},
we first explain why $H^\ast(\bm{s},\chi)$ is a Fourier coefficient of $\mathsf{Z}_0$.
By Proposition~\ref{PROP:large-s-absolute-convergence-properties},
$H(\bm{s},g)^{-1} \in L^1(G(\adele_\QQ))$,
so Fubini's theorem allows us to
rewrite the integral $H^\ast(\bm{s},\chi)$,
by folding $G(\adele_\QQ)$ first over $G(\QQ)$
and then over $\QQ\backslash \adele_\QQ$,
as follows:
\begin{equation*}
\begin{split}
H^\ast(\bm{s},\chi)
= \int_{G(\adele_\QQ)} H(\bm{s},g)^{-1} \ol{\chi}(\det(g))\, dg
&= \int_{G(\QQ)\backslash G(\adele_\QQ)}
\sum_{g'\in G(\QQ)} H(\bm{s},g' g)^{-1} \ol{\chi}(\det(g))\, dg \\
&= \int_{G(\QQ)\backslash G(\adele_\QQ)} \mathsf{Z}(\bm{s},g) \ol{\chi}(\det(g))\, dg \\
&= \int_{\QQ^\times\backslash \adele_\QQ^\times}
\int_{\QQ\backslash \adele_\QQ} \mathsf{Z}(\bm{s},(a,b))\ol{\chi}(a)
\, db\, d^\times{a} \\
&= \int_{\QQ^\times\backslash \adele_\QQ^\times} \mathsf{Z}_0(\bm{s},(a,0)) \ol{\chi}(a)\, d^\times{a},
\end{split}
\end{equation*}
by \eqref{EQN:define-Z_0-as-b-average},
where $\int_{\QQ\backslash \adele_\QQ} \mathsf{Z}(\bm{s},(a,b))\ol{\chi}(a)\, db$ is viewed as a $\QQ^\times$-invariant function of $a\in \adele_\QQ^\times$,
even though $\mathsf{Z}(\bm{s},(a,b))$ is not $\QQ^\times$-invariant.
This is what we wanted.


The characters $\chi$ in \eqref{EQN:Fourier-expand-Z_0} are of the form $\chi_0\, \abs{\cdot}^{it}$ with $(\chi_0,t)\in \mathcal{U} \times \RR$, where $\mathcal{U}$ is a finite group, because we work over $\QQ$.
(Over number fields, the story is more complicated
and we refer the interested reader to \cite{tanimoto2012distribution}*{Lemma~47}.)
The measure $d\chi$ is the product of the counting measure on $\mathcal{U}$
with the measure $\frac{dt}{2\pi}$ on $\RR$.
The Fourier expansion \eqref{EQN:Fourier-expand-Z_0} is justified, since
$\mathsf{Z}_0\in L^1$ by Proposition~\ref{PROP:large-s-absolute-convergence-properties},
and $\int \abs{H^\ast(\bm{s},\chi)}\, d\chi < \infty$ by Lemma~\ref{LEM:large-s-real-IBP-bounds}(1) with $\alpha=0$.

\begin{proposition}
\label{PROP:Z_0-partial-canonical-main-terms}
Fix $I\belongs J$
with $I\cap J_1 \ne \emptyset$ and $I\cap J_2 \ne \emptyset$.
Fix a vector $$\bm{\kappa}=(\kappa_\jboundary)_{\jboundary\in J}
\in -\div(\omega) + \sum_{\jboundary\in J\setminus I} \RR_{>0} D_\jboundary
\belongs \Pic^G(X)\otimes \RR = \RR^J.$$
Then $(s-1)^{\card{I}-1} \mathsf{Z}_0(s\bm{\kappa},1_G)$ is holomorphic, with at most polynomial growth in vertical strips, on $\Re(s)\ge 1-\delta$.
Also, $$\lim_{s\to 1}{(s-1)^{\card{I}-1} \mathsf{Z}_0(s\bm{\kappa},1_G)}
= \mathcal{X}_{\Lambda_I(X)}(\bm{\kappa})
\lim_{\bm{s}'\to \bm{\kappa}}{H^\ast(\bm{s}',1)
\prod_{\jboundary\in I}(s'_\jboundary-\kappa_\jboundary)}.$$
\end{proposition}

\begin{proof}
Use \eqref{EQN:Fourier-expand-Z_0}, Lemma~\ref{LEM:snc-implies-nice-densities-implies-Tamagawa-extraction}(1)$\Rightarrow$(2),
and Lemma~\ref{LEM:shifted-integral-polar-and-residue-structure} with $P=I$
and $$\digamma(\bm{z})=H^\ast(\bm{z}+\bm{\kappa},\chi_0),$$
where $\chi_0\in \mathcal{U}$.
Then specialize to $\bm{z} = (s-1)\bm{\kappa}$ and let $s\to 1$.
The main contribution to $\mathsf{Z}_0$ comes from $\chi_0=1$.
Cf.~\cite{tanimoto2012height}*{proof of Theorem~2.1} or \cite{tanimoto2012distribution}*{proof of Theorem~48}.
\end{proof}

Proposition~\ref{PROP:Z_0-partial-canonical-main-terms} and Lemma~\ref{LEM:snc-implies-nice-densities-implies-Tamagawa-extraction} give a satisfactory understanding of $\mathsf{Z}_0$ for us.
It remains to discuss $\mathsf{Z}_1\defeq \mathsf{Z} - \mathsf{Z}_0$.
We follow \cite{tanimoto2012height}*{proofs of Lemma~5.3 and Proposition~5.4}.
Clearly $\mathsf{Z}_1$ is left $G(\QQ)$-invariant,
since $\mathsf{Z}$ and $\mathsf{Z}_0$ are.
For $a\in \adele_\QQ^\times$, let
\begin{equation}
\label{EQN:define-h-as-twisted-Z_1-integral}
\mathsf{h}(\bm{s},a) \defeq \int_{\QQ\backslash \adele_\QQ} \mathsf{Z}_1(\bm{s},(a,b)) \ol{\psi}(b)\, db,
\end{equation}
where the quantities $\mathsf{Z}_1(\bm{s},(a,b))$
and $\psi(b)$ (the character from \S\ref{SUBSEC:conventions})
are well-defined.
We have $\mathsf{h}\in (C^\infty\cap L^\infty)(\adele_\QQ^\times)$, since $\mathsf{Z}_1=\mathsf{Z}-\mathsf{Z}_0\in (C^\infty\cap L^\infty)(G(\adele_\QQ))$ and $\QQ\backslash \adele_\QQ$ is compact.
Since $\vol(\QQ\backslash \adele_\QQ)=1$,
and $\mathsf{Z}_0(\bm{s},(a,b))$ is independent of $b$,
we have
\begin{equation*}
\int_{\QQ\backslash \adele_\QQ} \mathsf{Z}_0(\bm{s},(a,b'))\, db'
=  \mathsf{Z}_0(\bm{s},(a,b))
= \int_{\QQ\backslash \adele_\QQ} \mathsf{Z}(\bm{s},(a,b'))\, db'
\end{equation*}
by \eqref{EQN:define-Z_0-as-b-average},
whence $\int_{\QQ\backslash \adele_\QQ} \mathsf{Z}_1(\bm{s},(a,b))\, db = 0$.
By Fourier expansion of $\mathsf{Z}_1(\bm{s},(a,b))$ in $b\in \QQ\backslash \adele_\QQ$,
as in \cite{tanimoto2012height}*{proof of $\Theta I = 1$
in the final display of the proof of Lemma~3.4},
we find that
\begin{equation*}
\begin{split}
\mathsf{Z}_1(\bm{s},(a,b))
&= \sum_{\alpha\in \QQ}
\int_{\QQ\backslash \adele_\QQ} \mathsf{Z}_1(\bm{s},(a,b+y)) \ol{\psi}(\alpha y)\, dy \\
&= \sum_{\alpha\in \QQ^\times}
\int_{\QQ\backslash \adele_\QQ} \mathsf{Z}_1(\bm{s},(a,y))
\ol{\psi}(\alpha (y-b))\, dy \\
&= \sum_{\alpha\in \QQ^\times} \psi(\alpha b)
\int_{\QQ\backslash \adele_\QQ} \mathsf{Z}_1(\bm{s},(\alpha a,\alpha y))
\ol{\psi}(\alpha y)\, dy,
\end{split}
\end{equation*}
by the scaling identity $(\alpha,0)(a,b) = (\alpha a,\alpha b)$ for $\alpha\in \QQ^\times$.
Multiplication by $\alpha$ on $\QQ\backslash \adele_\QQ$ is measure-preserving,
so by \eqref{EQN:define-h-as-twisted-Z_1-integral},
we conclude that
\begin{equation}
\label{EQN:Z_1-prelim-psi-h-expansion}
\mathsf{Z}_1(\bm{s},g)
= \sum_{\alpha\in \QQ^\times} \psi(\alpha b) \mathsf{h}(\bm{s}, \alpha a).
\end{equation}

But $\int_{\QQ\backslash \adele_\QQ} \mathsf{Z}_0(\bm{s},(a,b)) \ol{\psi}(b)\, db
= \mathsf{Z}_0(\bm{s},(a,0)) \int_{\QQ\backslash \adele_\QQ} \ol{\psi}(b)\, db
= 0$, so \eqref{EQN:define-h-as-twisted-Z_1-integral} simplifies to
\begin{equation}
\label{EQN:rewrite-h-as-twisted-Z-integral}
\mathsf{h}(\bm{s},a)
= \int_{\QQ\backslash \adele_\QQ} \mathsf{Z}(\bm{s},(a,b)) \ol{\psi}(b)\, db.
\end{equation}
The following lemma is essentially
\cite{tanimoto2012height}*{first paragraph of the proof of Lemma~5.3}:
\begin{lemma}
\label{h-is-in-L^1}
We have $\mathsf{h}\in L^1(\adele_\QQ^\times)$.
\end{lemma}

\begin{proof}
For each $b\in \adele_\QQ$, the definition of $\mathsf{Z}$ yields
\begin{equation*}
\mathsf{Z}(\bm{s},(a,b))
= \sum_{(\alpha^{-1},y\alpha^{-1})\in G(\QQ)} H(\bm{s},(\alpha^{-1},y\alpha^{-1})(a,b))^{-1}
= \sum_{\alpha\in \QQ^\times} \sum_{y\in \QQ} H(\bm{s},(\alpha^{-1} a,\alpha^{-1}(b+y)))^{-1}.
\end{equation*}
Plugging this into \eqref{EQN:rewrite-h-as-twisted-Z-integral}
and unfolding $\QQ\backslash \adele_\QQ$ over $\QQ$
(by Fubini's theorem), we get
\begin{equation*}
\mathsf{h}(\bm{s},a)
= \sum_{\alpha\in \QQ^\times}
\int_{\adele_\QQ} H(\bm{s},(\alpha^{-1} a,\alpha^{-1} b))^{-1} \ol{\psi}(b)\, db.
\end{equation*}
Using the triangle inequality, and moving $\sum_\alpha$ to the outside, we get
\begin{equation*}
\begin{split}
\int_{\adele_\QQ^\times} \abs{\mathsf{h}(\bm{s},a)}\, d^\times{a}
&\le \sum_{\alpha\in \QQ^\times} \int_{\adele_\QQ^\times}
\biggl\lvert{\int_{\adele_\QQ} H(\bm{s},(\alpha^{-1} a,\alpha^{-1} b))^{-1} \ol{\psi}(b)\, db}\biggr\rvert
\, d^\times{a} \\
&= \sum_{\alpha\in \QQ^\times} \int_{\adele_\QQ^\times}
\biggl\lvert{\int_{\adele_\QQ} H(\bm{s},(a,b))^{-1} \ol{\psi}(\alpha b)\, db}\biggr\rvert
\, d^\times{a}
= \sum_{\alpha\in \QQ^\times} \int_{\adele_\QQ^\times}
\abs{I(a,\alpha)}
\, d^\times{a},
\end{split}
\end{equation*}
say.
With analogous local notation, we have
$I(a,\alpha) = \prod_v I_v(a_v,\alpha)$ (by Fubini's theorem).
However, by the right $\Kbad$-invariance of $H$,
we have $I_p(a_p,\alpha) = I_p(a_p,\alpha) \psi_p(\alpha a_py_p)$
for all $y_p\in \Nbad \ZZ_p$.
Taking $y_p=\Nbad$, it follows that $I_p(a_p,\alpha)=0$
unless $\alpha a_p\Nbad\in \ZZ_p$.
By the triangle inequality, we get
$\int_{\QQ_p^\times} \abs{I_p(a_p,\alpha)}\, d^\times{a_p}
\le \int_{G(\QQ_p):\, \Nbad \alpha a_p\in \ZZ_p}
\abs{H_p(\bm{s},g_p)}^{-1}\, dg_p$.
Since
\begin{equation*}
\int_{\adele_\QQ^\times} \abs{I(a,\alpha)}\, d^\times{a}
= \int_{\adele_\QQ^\times} \prod_v \abs{I_v(a_v,\alpha)}\, d^\times{a}
= \prod_v \int_{\QQ_v^\times} \abs{I_v(a_v,\alpha)}\, d^\times{a_v},
\end{equation*}
it now follows from Lemma~\ref{LEM:large-s-technical-adelic-convergence-bounds}(2)
that $\mathsf{h}\in L^1(\adele_\QQ^\times)$.
\end{proof}

By Lemma~\ref{h-is-in-L^1},
we now have $\mathsf{h}\in L^q(\adele_\QQ^\times)$ for all $q\in [1,\infty]$
(since $\mathsf{h}\in L^\infty$).
This lets us Fourier-expand $\mathsf{h}(\bm{s},a)$ in the variable $a$.
We do this first at finite places $v$, using $L^2$ theory,
and then at $v=\infty$, using $L^1$ theory.
(We roughly follow \cite{tanimoto2012height}*{proof of Proposition~3.6}.)
Let $\adele_{\QQ,f}^\times$ be the finite part of $\adele_\QQ^\times$.
Recall the characters $\lambda\in \Mbad$ from \eqref{relevant-badness-characters},
and the functions $\mathbf{v}_{m,\lambda}(a)$ from \eqref{EQN:define-finite-vector-v_m,lambda}.
Let $L^q(\RR^\times)^{\lambda_\infty}$ consist of
even (resp.~odd) functions in $L^q(\RR^\times)$
if $\lambda_\infty$ is trivial (resp.~nontrivial).
Then $\mathbf{v}_{m,\lambda}(a_f,a_\infty)\in L^\infty(\RR^\times)^{\lambda_\infty}$
as a function of $a_\infty\in \RR^\times$.
Let
\begin{equation}
\label{discrete-fourier-coefficient-hmlambda}
\mathsf{h}_{m,\lambda}(\bm{s},a_\infty)
\defeq \frac12 \sum_{\varsigma\in \set{\pm 1}} \int_{\adele_{\QQ,f}^\times}
\mathsf{h}(\bm{s},(a_f,\varsigma a_\infty))
\ol{\mathbf{v}}_{m,\lambda}(a_f,\varsigma a_\infty)\, d^\times{a_f}.
\end{equation}
Then $\mathsf{h}_{m,\lambda}\in L^1(\RR^\times)
\cap L^\infty(\RR^\times)$,
where $\mathsf{h}_{m,\lambda}\in L^1$
because $\mathsf{h}\in L^1$ and $\mathbf{v}_{m,\lambda}\in L^\infty$,
and where $\mathsf{h}_{m,\lambda}\in L^\infty$
because $\mathsf{h},\mathbf{v}_{m,\lambda}\in L^\infty$ and $\mathbf{v}_{m,\lambda}$ is supported on $\RR^\times\times \prod_p \Nbad^{-1}m\ZZ_p^\times$.

\begin{lemma}
\label{finite-fourier-expand-h}
We have
$\mathsf{h}(\bm{s},a) = \sum_{\lambda\in \Mbad} \sum_{m\ge 1}
\mathbf{v}_{m,\lambda}(a) \mathsf{h}_{m,\lambda}(\bm{s},a_\infty)$.
\end{lemma}

\begin{proof}
Since $\mathsf{Z}$ is right $\Kbad$-invariant,
\eqref{EQN:rewrite-h-as-twisted-Z-integral} implies
that $\mathsf{h}(\bm{s},a\ucord) = \psi(a\vcord) \mathsf{h}(\bm{s},a)$ for all $(\ucord,\vcord)\in \Kbad$.
Taking $\ucord=1$ and $\vcord\in \Nbad\ZZ_p$,
we find that $\mathsf{h}(\bm{s},a)$ is supported on
$$a\in \adele_\QQ^\times \cap \left(\RR\times \prod_p \Nbad^{-1}\ZZ_p\right)
= \bigcup_{m\ge 1} \left(\RR^\times\times \prod_p \Nbad^{-1}m\ZZ_p^\times\right).$$
Taking $\ucord\in \ZZ_p^\times\cap (1+\Nbad\ZZ_p)$,
we find that if $a\in \RR^\times\times \prod_p \Nbad^{-1}m\ZZ_p^\times$,
then $\mathsf{h}(\bm{s},a)$ depends only on
$\bm{s}$, $a_\infty$, and the class of $\sgn_p(a_p)=a_p/p^{v_p(a_p)}$ in $\ZZ_p^\times/(1+\Nbad\ZZ_p)$ for each $p\in \Sbad$.
Consider the restriction $\mathsf{h}'_m(\bm{s},a)\defeq \mathsf{h}(\bm{s},a)
\cdot \bm{1}_{a\in \RR^\times\times \prod_p \Nbad^{-1}m\ZZ_p^\times}$
of $\mathsf{h}$ to $\RR^\times\times \prod_p \Nbad^{-1}m\ZZ_p^\times$.
Then $\mathsf{h}(\bm{s},a) = \sum_{m\ge 1} \mathsf{h}'_m(\bm{s},a)$.
Additionally, by Fourier analysis on the finite group
$(\RR^\times/(\RR^\times)^2)
\times \prod_{p\in \Sbad} (\ZZ_p^\times/(1+\Nbad\ZZ_p))$,
we have $\mathsf{h}'_m(\bm{s},a)
= \sum_{\lambda\in \Mbad} \mathbf{v}_{m,\lambda}(a)
\mathsf{h}_{m,\lambda}(\bm{s},a_\infty)$,
since the Fourier coefficient of
the vector $\mathbf{v}_{m,\lambda}(a)
=\lambda_\infty(a_\infty)\prod_p \lambda_p(\sgn_p(a_p))$
in $\mathsf{h}'_m(\bm{s},a)$ is
\begin{equation}
\label{h'-versus-h-fourier-coefficient}
\frac12 \sum_{\varsigma\in \set{\pm 1}} \int_{\adele_{\QQ,f}^\times}
\mathsf{h}'(\bm{s},(a_f,\varsigma a_\infty))
\ol{\mathbf{v}}_{m,\lambda}(a_f,\varsigma a_\infty)\, d^\times{a_f}
= \mathsf{h}_{m,\lambda}(\bm{s},a_\infty).
\end{equation}
The equality \eqref{h'-versus-h-fourier-coefficient}
holds by \eqref{discrete-fourier-coefficient-hmlambda},
since $\mathsf{h}'(\bm{s},a)\ol{\mathbf{v}}_{m,\lambda}(a)
= \mathsf{h}(\bm{s},a)\ol{\mathbf{v}}_{m,\lambda}(a)$,
since $\ol{\mathbf{v}}_{m,\lambda}$ is supported on $\RR^\times\times \prod_p \Nbad^{-1}m\ZZ_p^\times$.
The lemma now immediately follows from Fubini's theorem,
since for any $a\in \adele_\QQ^\times$,
there are only finitely many pairs $(\lambda,m)$ with $\mathbf{v}_{m,\lambda}(a)\ne 0$.
\end{proof}

\begin{lemma}
\label{infinite-fourier-coeff-h}
For any $(m,\lambda,t)\in \ZZ_{\ge 1} \times \Mbad \times \RR$, we have
\begin{equation*}
\int_{\RR^\times} \mathsf{h}_{m,\lambda}(\bm{s},a_\infty)
\abs{a_\infty}_\infty^{-it}\, d^\times{a_\infty}
= \int_{G(\QQ)\backslash G(\adele_\QQ)} \mathsf{Z}(\bm{s},g) \ol{\theta}_{m,\lambda,t}(g)\, dg.
\end{equation*}
\end{lemma}

\begin{proof}
Multiplication by $-1$ on $\RR^\times$ is measure-preserving.
So by \eqref{discrete-fourier-coefficient-hmlambda} and Fubini's theorem,
\begin{equation*}
I\defeq \int_{\RR^\times} \mathsf{h}_{m,\lambda}(\bm{s},a_\infty)
\abs{a_\infty}_\infty^{-it}\, d^\times{a_\infty}
= \int_{\adele_\QQ^\times} \mathsf{h}(\bm{s},a) \ol{\mathbf{v}}_{m,\lambda}(a)
\abs{a_\infty}_\infty^{-it}\, d^\times{a}.
\end{equation*}
Plugging in \eqref{EQN:rewrite-h-as-twisted-Z-integral},
and folding $\adele_\QQ^\times$ over $\QQ^\times$,
we get
\begin{equation*}
\begin{split}
I
&= \int_{\adele_\QQ^\times} \int_{\QQ\backslash \adele_\QQ} \mathsf{Z}(\bm{s},(a,b)) \ol{\psi}(b)\ol{\mathbf{v}}_{m,\lambda}(a) \abs{a_\infty}_\infty^{-it}\, db\, d^\times{a} \\
&= \int_{\QQ^\times\backslash \adele_\QQ^\times} \sum_{\alpha\in \QQ^\times}
\int_{\QQ\backslash \adele_\QQ} \mathsf{Z}(\bm{s},(\alpha a,\alpha b)) \ol{\psi}(\alpha b)\ol{\mathbf{v}}_{m,\lambda}(\alpha a) \abs{\alpha a_\infty}_\infty^{-it}\, db\, d^\times{a}
\end{split}
\end{equation*}
by Fubini's theorem, which applies since $\mathsf{Z}\in L^1(G(\QQ)\backslash G(\adele_\QQ))$ and $\ol{\mathbf{v}}_{m,\lambda}(\alpha a)$ is supported on $\alpha = \pm \prod_p p^{v_p(m/\Nbad a_p)} \in \QQ^\times$ (if $a\in \adele_\QQ^\times$ is fixed).
Thanks to
the invariance $\mathsf{Z}(\bm{s},(\alpha a,\alpha b))=\mathsf{Z}(\bm{s},(a,b))$ for $\alpha\in \QQ^\times$, we can move $\sum_\alpha$ inside to get
\begin{equation*}
\begin{split}
I
&= \int_{\QQ^\times\backslash \adele_\QQ^\times}
\int_{\QQ\backslash \adele_\QQ} \mathsf{Z}(\bm{s},(a,b))
\sum_{\alpha\in \QQ^\times} \ol{\psi}(\alpha b)\ol{\mathbf{v}}_{m,\lambda}(\alpha a)
\abs{\alpha a_\infty}_\infty^{-it}\, db\, d^\times{a}.
\end{split}
\end{equation*}
By \eqref{clever-theta-sum-over-Qcross},
we have $\sum_{\alpha\in \QQ^\times} \ol{\psi}(\alpha b)\ol{\mathbf{v}}_{m,\lambda}(\alpha a) \abs{\alpha a_\infty}_\infty^{-it} = \ol{\theta}_{m,\lambda,t}(a,b)$,
for any $a\in \adele_\QQ^\times$ and $b\in \QQ\backslash \adele_\QQ$.
So $I = \int_{\QQ^\times\backslash \adele_\QQ^\times} \int_{\QQ\backslash \adele_\QQ}
\mathsf{Z}(\bm{s},(a,b)) \ol{\theta}_{m,\lambda,t}(a,b)\, db\, d^\times{a}$,
which suffices by \eqref{adelic-SES}.
\end{proof}

\begin{lemma}
\label{fourier-coeff-Z}
For any $(m,\lambda,t)\in \ZZ_{\ge 1} \times \Mbad \times \RR$, we have
\begin{equation*}
\int_{G(\QQ)\backslash G(\adele_\QQ)} \mathsf{Z}(\bm{s},g) \ol{\theta}_{m,\lambda,t}(g)\, dg
= H^\ast(\bm{s},m,\lambda,t)
= \sum_{\alpha\in \QQ^\times} H'_\infty(\bm{s},\lambda,t,\alpha) \prod_p H'_p(\bm{s},\lambda,m,\alpha).
\end{equation*}
\end{lemma}

\begin{proof}
Unfolding $G(\QQ)\backslash G(\adele_\QQ)$ over $G(\QQ)$ proves the first equality,
by the definition \eqref{EQN:define-non-abelian-Fourier-coefficient-H^ast} of $H^\ast$.
Fubini's theorem here
is justified by Proposition~\ref{PROP:large-s-absolute-convergence-properties}.
For the second equality, we first expand definitions,
using \eqref{clever-theta-sum-over-Qcross}
and \eqref{EQN:define-finite-vector-v_m,lambda},
to write $H^\ast(\bm{s},m,\lambda,t)$ as
\begin{equation*}
\int_{G(\adele_\QQ)} H(\bm{s},g)^{-1}
\sum_{\alpha\in \QQ^\times} \ol{\psi}(\alpha b)
\abs{\alpha a_\infty}_\infty^{-it} \ol{\lambda}_\infty(\alpha a_\infty)
\prod_p \ol{\lambda}_p(\sgn_p(\alpha a_p)) \bm{1}_{\Nbad \alpha a_p/m\in \ZZ_p^\times}
\, dg.
\end{equation*}
Moving $\sum_\alpha$ to the outside,
and writing $\int_{G(\adele_\QQ)} = \prod_v \int_{G(\QQ_v)}$,
completes the proof,
by the definition of $H'_\infty$ and $H'_p$ from \eqref{EXPR:formula-for-H'_p}.
The manipulations are justified by Lemma~\ref{LEM:large-s-technical-adelic-convergence-bounds}(1).
\end{proof}

We have
$\sum_{\lambda\in \Mbad} \sum_{m\ge 1} \abs{H^\ast(\bm{s},m,\lambda,t)} < \infty$,
by Lemma~\ref{LEM:large-s-technical-adelic-convergence-bounds}(1)
and the second equality in Lemma~\ref{fourier-coeff-Z}.
Therefore, Fourier expansion of $\mathsf{h}_{m,\lambda}\in L^1(\RR^\times)$ gives
\begin{equation}
\mathsf{h}(\bm{s},a) = \sum_{\lambda\in \Mbad} \sum_{m\ge 1} \mathbf{v}_{m,\lambda}(a) (4\pi)^{-1} \int_\RR H^\ast(\bm{s},m,\lambda,t) \abs{a_\infty}_\infty^{it}\, dt,
\end{equation}
by Lemmas~\ref{finite-fourier-expand-h},
\ref{infinite-fourier-coeff-h},
and~\ref{fourier-coeff-Z}.
Plugging this into \eqref{EQN:Z_1-prelim-psi-h-expansion} gives
\begin{equation}
\mathsf{Z}_1(\bm{s},g) = \sum_{\lambda\in \Mbad} \sum_{m\ge 1} (4\pi)^{-1} \int_{\RR} H^\ast(\bm{s},m,\lambda,t) \theta_{m,\lambda,t}(g)\, dt,
\end{equation}
upon recalling the definition of $\theta_{m,\lambda,t}(g)$ from \eqref{clever-theta-sum-over-Qcross};
the required use of Fubini's theorem
is justified by Lemma~\ref{LEM:large-s-technical-adelic-convergence-bounds}(1).
Using \eqref{EQN:formula-for-theta-at-1} to evaluate $\theta_{m,\lambda,t}(g)$ at $g=1_G$, and using the second equality in Lemma~\ref{fourier-coeff-Z} to rewrite $H^\ast(\bm{s},m,\lambda,t)$, we get
\begin{equation*}
\mathsf{Z}_1(\bm{s},1_G)
= \sum_{\lambda\in \Mbad} \sum_{m\ge 1} (4\pi)^{-1} \int_{\RR}
\frac{(1+\lambda(-1)) \lambda_\Sbad(m/\Nbad )}{\prod_p \abs{m/\Nbad }_p^{it}}
\sum_{\alpha\in \QQ^\times} H'_\infty(\bm{s},\lambda,t,\alpha) \prod_p H'_p(\bm{s},\lambda,m,\alpha)\, dt.
\end{equation*}
By \eqref{lambda-Sp-definition},
the factor $\lambda_{\Sbad,p}(a_p)$ of $\lambda_\Sbad$ depends only on $v_p(a_p)$.
Move $\sum_\alpha$ outside,
write
\begin{equation*}
\begin{split}
&\sum_{m=p^k:\, k\ge 0} H'_p(\bm{s},\lambda,m,\alpha)
\lambda_{\Sbad,p}(m/\Nbad ) \abs{m/\Nbad }_p^{-it} \\
&= \sum_{m=p^k:\, k\ge 0}
\int_{G(\QQ_p):\, \Nbad \alpha a_p\in m\ZZ_p^\times} H_p(\bm{s},g_p)^{-1}
\phase(-\alpha b_p\bmod{\ZZ_p}) \ol{\lambda}_p(\sgn_p(\alpha a_p))
\lambda_{\Sbad,p}(\alpha a_p) \abs{\alpha a_p}_p^{-it}\, dg_p \\
&= \abs{\alpha}_p^{-it} H^\vee_p(\bm{s},\lambda,t,\alpha)
\end{split}
\end{equation*}
using \eqref{EXPR:formula-for-H'_p} and \eqref{EXPR:formula-for-H^vee_p},
and recall $H'_\infty = \abs{\alpha}_\infty^{-it} H^\vee_\infty$,
to get
\begin{equation}
\label{EQN:difficult-part-of-spectral-expansion}
\mathsf{Z}_1(\bm{s},1_G)
= \sum_{\lambda\in \Mbad:\, \lambda(-1)=1}\,
\sum_{\alpha\in \QQ^\times} (2\pi)^{-1}
\int_\RR \prod_v H^\vee_v(\bm{s},\lambda,t,\alpha)\, dt
\end{equation}
(cf.~\cite{tanimoto2012height}*{Proposition~5.4}).
The required manipulations are justified by Lemma~\ref{LEM:large-s-technical-adelic-convergence-bounds}(1).

Both additive and multiplicative harmonics appear above; the sum over $\alpha\in \QQ^\times$ somehow reflects the non-abelian nature of $G$.
The multiplicative harmonics in \eqref{EQN:difficult-part-of-spectral-expansion} in fact reflect a symmetry in $\bm{s}$.
By \eqref{EQN:local-character-height-consistency-relation} we have $\abs{a(g_v)}_v = H_v(\div(a),g_v)^{-1}$, so by inspection of \eqref{EXPR:formula-for-H^vee_p},
\begin{equation}
\label{EQN:change-of-t-identity}
H^\vee_v(\bm{s},\lambda,t,\alpha)
= H^\vee_v(\bm{s}-it\div(a), \lambda, 0, \alpha).
\end{equation}
Thus it will suffice to study $H^\vee_v$ for $t=0$ (without the factor $\abs{a}_v^{-it}$).

\section{New geometric and parametric observations}
\label{SEC:new-geometry}


In this section, we make crucial observations for later.
Until specified otherwise,
we do not assume $X$ is split or that $D$ has strict normal crossings.
(This generality will be useful in \S\ref{SEC:final-reductions}.)
In particular, the $D_\jboundary$ are irreducible but not necessarily smooth or geometrically irreducible.

\begin{proposition}
\label{PROP:upper-bound-on-anticanonical-Weil-divisor}
Let $\jboundary\in J$ and $c\in \QQ$.
Then $\mathsf{d}_\jboundary \le 1-\ord_{D_\jboundary}(b-c)$.
\end{proposition}

\begin{proof}
Let $\rord = \ord_{D_\jboundary}(a)$ and $\sord = \ord_{D_\jboundary}(b-c)$.
Now work over $\CC$.
Locally near a \emph{general} point $x\in D_\jboundary$,
choose coordinates $y,z\ll 1$
such that $$a = y^{\rord} f_3(y,z),
\quad b-c = y^{\sord} f_4(y,z),$$
where $f_3$, $f_4$ are analytic with $f_3(0,0) f_4(0,0) \ne 0$.
(We choose $y$ so that $y=0$ cuts out $D_\jboundary$ near $x$.)
Thus $$da\,db = O(y^{\rord+\sord-1})\,dy\,dz.$$
So $\omega = db\, da/a = O(y^{\sord-1})\, dz\, dy$.
So $-\mathsf{d}_\jboundary = \ord_{D_\jboundary}(\omega) \ge \sord-1$.
\end{proof}

\begin{proof}
[Proof of Proposition~\ref{PROP:apply-lower-bound-on-anticanonical-Weil-divisor}]
By Proposition~\ref{PROP:poles-of-left-invariant-top-form}, $\mathsf{d}_\jboundary\ge 1-\ord_{D_\jboundary}(a)$.
Now use Proposition~\ref{PROP:upper-bound-on-anticanonical-Weil-divisor}.
\end{proof}

Recall $J_1$, $J_2$, $J_3$ from \S\ref{SUBSEC:polar-combinatorics}.
Inspired by Proposition~\ref{PROP:apply-lower-bound-on-anticanonical-Weil-divisor}
and the condition \eqref{INEQ:mysterious-numerical-condition}, let
\begin{align}
J^c_1 &\defeq \set{\jboundary\in J_1:
\ord_{D_\jboundary}(b-c) = \ord_{D_\jboundary}(a)},
\label{Jc1} \\
J^\ast_2 &\defeq \set{\jboundary\in J_2:
\ord_{D_\jboundary}(b/a) = 0},
\label{Jast2} \\
I^c &\defeq \set{\jboundary\in J:
\ord_{D_\jboundary}(b-c) > 0},
\label{Ic}
\end{align}
for each $c\in \QQ$.
Then in particular,
\begin{equation}
\label{Jc1-Ic-J1-containments}
J^c_1\belongs I^c\belongs J_1,
\end{equation}
where the latter containment holds by Proposition~\ref{PROP:apply-lower-bound-on-anticanonical-Weil-divisor}.

\begin{definition}
For each $\jboundary\in J$, let
\begin{equation}
\label{C0j}
{\textstyle \mathcal{C}^{(0)}(\jboundary)
\defeq \argmax_{c\in \QQ} \ord_{D_\jboundary}(b-c)},
\end{equation}
where in general for a function $f_5\maps V\to \RR$,
we let
\begin{equation}
\label{argmax}
{\textstyle\argmax_V f_5
\defeq \argmax_{x\in V} f_5(x)
\defeq \set{x\in V: f_5(x) \ge f_5(y)\;\textnormal{for all $y\in V$}}}.
\end{equation}
\end{definition}

\begin{proposition}
\label{PROP:compute-arg-max-set-C0j}
If $\jboundary\in I^c$ for some $c\in \QQ$,
then $\mathcal{C}^{(0)}(\jboundary)=\set{c}$;
otherwise,
$\mathcal{C}^{(0)}(\jboundary)=\QQ$.
\end{proposition}

\begin{proof}
For $d\in \QQ$, let $f_6(d) \defeq \ord_{D_\jboundary}(b-d)$.
If $\jboundary\in I^c$, then for all $d\ne c$, we have
$$f_6(d) = 0 < f_6(c),$$
so $\mathcal{C}^{(0)}(\jboundary)=\set{c}$.
If $\jboundary\notin \bigcup_{c\in \QQ} I^c$
and $f_6(0) \ge 0$, then $f_6(c) = 0$ for all $c\in \QQ$,
so $\mathcal{C}^{(0)}(\jboundary)=\QQ$.
If $f_6(0) < 0$, then $f_6(c) = f_6(0)$ for all $c\in \QQ$,
so $\mathcal{C}^{(0)}(\jboundary)=\QQ$.
\end{proof}

Thus $I^c\cap I^{c'} = \emptyset$ if $c\ne c'$.
In addition, for $\jboundary\in J$, we find that
$D_\jboundary$ is special (Definition~\ref{DEFN:special-divisors})
if and only if
$\jboundary\in J^\ast_2$ or $\jboundary\in \bigcup_{c\in \QQ} J^c_1$.

\begin{proposition}
\label{PROP:critical-index-upper-bound}
Let $c\in \QQ$.
Then $\card{J^c_1 \cup J^\ast_2} \le \rank(\Pic(X))$.
\end{proposition}

\begin{proof}
Since $X$ admits no nonconstant morphism to $\Aff^1$, we have $\div_\infty((b-c)/a) \ne 0$ (by the algebraic Hartogs lemma).
But $\ord_{D_\jboundary}((b-c)/a) \ge 0$ for all $\jboundary\in J^c_1 \cup J^\ast_2$.
Since $\div_\infty((b-c)/a)$ is supported on $D$,
it follows that $\card{J^c_1 \cup J^\ast_2} \le \card{J} - 1 = \rank(\Pic(X))$,
where the final equality holds by Proposition~\ref{PROP:compute-Picard-group-in-terms-of-boundary}.
(This holds even if $X$ is non-split, because $J$ labels irreducible divisors over $\QQ$ rather than $\ol{\QQ}$.)
\end{proof}


Precisely understanding the $D_\jboundary$ seems tricky in general, but the following will suffice:

\begin{proposition}
\label{PROP:G-induced-nonconstancy-of-a-rational-function-on-Dj}
Let $\jboundary\in J$ and
$c\in \mathcal{C}^{(0)}(\jboundary)$.
Let $\rord = \ord_{D_\jboundary}(a)$, $\sord = \ord_{D_\jboundary}(b-c)$,
and $F_7 = (b-c)^{\rord}/a^{\sord}$.
Then $F_7\in \Gamma(U_1,\mathcal{O}_X)^\times$ for some dense open set $U_1\belongs X$ with $U_1\cap D_\jboundary \ne \emptyset$.
The rational map $f_8\maps D_\jboundary \dashrightarrow \PP^1_\QQ$ extending $F_7\vert_{U_1\cap D_\jboundary}$ is nonconstant.
\end{proposition}

\begin{proof}
Since $\ord_{D_\jboundary}(F_7) = 0$, the first statement is clear.
In particular, $f_8^{-1}(\GG_m) \ne \emptyset$.

\emph{Case~1: $\sord=0$.}
If $f_8$ were constant, then we would have $b\vert_{D_\jboundary} = c'$
for some constant $c'\in \QQ$,
and thus $1\le \ord_{D_\jboundary}(b-c')
\le \ord_{D_\jboundary}(b-c) = \sord = 0$,
where the second inequality holds because $c\in \mathcal{C}^{(0)}(\jboundary)$.
That is absurd;
so instead, $f_8$ is nonconstant.

\emph{Case~2: $\rord=\sord\ne 0$.}
The formula \eqref{EQN:group-law-on-G} is ``homogeneous'' in $a$, $b$,
so it induces a right $G$-action
$$y\cdot (\ucord,\vcord) \defeq (\vcord+y)/\ucord$$
on $y=b/a\in \Aff^1$,
which extends to a $G$-action on $y\in \PP^1$
(with a fixed point at $\infty$).
Moreover, for this $G$-action on $\PP^1$,
the rational map $b/a$ from $X$ to $\PP^1$ is $G$-equivariant.
Similarly,
the rational map $(b-c)/a$ from $X$ to $\PP^1$ is $G$-equivariant;
one can check this using \eqref{EQN:group-law-on-G},
or by factoring $(b-c)/a$ through the $G$-equivariant map
$(a,b)\mapsto (1,-c)(a,b) = (a,b-c)$ on $G$.
Thus $\frac{b-c}{a}\vert_{D_\jboundary}$
is a $G$-equivariant rational map $D_\jboundary\dashrightarrow \PP^1$.
Since $f_8^{-1}(\GG_m) \ne \emptyset$,
and $G$ acts transitively on $\Aff^1(\ol{\QQ})$,
it follows that $f_8$ is nonconstant.

\emph{Case~3: $\sord\ne 0$ and $\rord\ne \sord$.}
Then $\rord>\sord$
by Proposition~\ref{PROP:apply-lower-bound-on-anticanonical-Weil-divisor},
so $a/(b-c)$ vanishes on a dense open set $V_1\belongs D_\jboundary$.
But \eqref{EQN:group-law-on-G} yields,
in $\Rat(X\times G)$,
the identity
\begin{equation*}
F_7(x\cdot (\ucord,\vcord))
= (a\vcord+b-c)^{\rord}/(a\ucord)^{\sord}
= (a\vcord/(b-c) + 1)^{\rord}
\cdot F_7(x)/\ucord^{\sord}
\end{equation*}
(for $x\in X$ and $(\ucord,\vcord)\in G$).
Since $F_7(x)$,
$F_7(x\cdot (\ucord,\vcord))$,
and $(a\vcord/(b-c) + 1)^{\rord}$
are regular on $X\times G$ near a \emph{general} point $(x,g)\in D_\jboundary\times G$,
restriction to $D_\jboundary$ gives
$$f_8(x\cdot (\ucord,\vcord)) = f_8(x)/\ucord^{\sord}$$
in $\Rat(D_\jboundary\times G)$.
Thus $f_8$ is $G$-equivariant for
a nontrivial multiplicative action of $G$ on $\PP^1$
(where $G$ fixes $0$ and $\infty$, and acts transitively on $\GG_m(\ol{\QQ})$).
Since $f_8^{-1}(\GG_m)$ is nonempty, it follows that $f_8$ is nonconstant.
\end{proof}

From now on, assume $X$ is split.
Thus, each $D_\jboundary$ is \emph{geometrically} irreducible.
In particular, a nonconstant function on $D_\jboundary$
will remain nonconstant on any nonempty open subset of $D_\jboundary(\efield)$,
for any local field $\efield$ of characteristic $0$.
This observation, applied to the function $F_7$ from Proposition~\ref{PROP:G-induced-nonconstancy-of-a-rational-function-on-Dj}, leads to new insight on local parameterizations of $X$ near $D_\jboundary$.

\begin{lemma}
\label{LEM:convenient-local-analytic-coordinates-generically}
Let $\jboundary$, $c$, $\rord$, $\sord$, and $F_7$
be as in Proposition~\ref{PROP:G-induced-nonconstancy-of-a-rational-function-on-Dj}.
Let $U_2\belongs X$ be the largest open set such that
$D_\jboundary\cap U_2$ is smooth and
$a, b-c\in \Gamma(U_2\setminus D_\jboundary, \mathcal{O}_X)^\times$.
Let $x\in (D_\jboundary\cap U_2)(\RR)$.
\begin{enumerate}
\item Assume $\jboundary\in J_1 \cup J_2$.
Then locally near $x$,
there are real-analytic coordinates $y,z\ll 1$
with $a = \unitone y^{\rord}$ and $(b-c)/y^{\sord} = \unittwo+z$,
where $\unitone, \unittwo\in \RR^\times$.

\item Assume $\jboundary\in J_3$.
Then locally near $x$,
there are real-analytic coordinates $y,z\ll 1$
with $b-c = \unitone y^{\sord}$ and $a = \unittwo+z$,
where $\unitone, \unittwo\in \RR^\times$.

\item The functions $(\Top{y})/y$
and $\Top{z}$ are regular near $x$,
for any $\Top
\in \set{(b-c)\,\frac{\partial}{\partial b},
a\,\frac{\partial}{\partial a},
a\,\frac{\partial}{\partial b}}$.
\end{enumerate}
\end{lemma}

\begin{proof}
(1):
There exists $\tloc\in \Rat(X)$ (defining $D_\jboundary$ locally)
such that $a/\tloc^{\rord}, (b-c)/\tloc^{\sord} \in \mathcal{O}_{X,x}^\times$.
Let $\uloc\in \Rat(X)$ be a regular local coordinate complementary to $\tloc$;
then analytically we have $$a/\tloc^{\rord} = \unitone+f_9(\tloc,\uloc),
\quad (b-c)/\tloc^{\sord} = \unittwo+f_{10}(\tloc,\uloc),$$
say,
where $\unitone,\unittwo\in \RR^\times$ and $f_9(0,0)=f_{10}(0,0)=0$.
Since $\rord\ne 0$
(we have $\rord>0$ if $\jboundary\in J_1$, and $\rord<0$ if $\jboundary\in J_2$),
we may change variables from $\tloc$
to $y\defeq \tloc (1+f_9(\tloc,\uloc)/\unitone)^{1/\rord}$.
More precisely,
we are changing variables from $(\tloc,\uloc)\in \RR^2$ to $(y,\uloc)\in \RR^2$,
using the inverse function theorem for analytic functions.
Then $$a = \unitone y^{\rord},
\quad (b-c)/y^{\sord} = \unittwo+f_{11}(y,\uloc),
\quad F_7(a,b)
= (\unittwo+f_{11})^{\rord} / \unitone^{\sord}
= \unitthree+f_{12}(y,\uloc),$$
say,
where $\unitthree = \unittwo^{\rord}/\unitone^{\sord} \in \RR^\times$.

Differentiating $a=\unitone y^{\rord}$ in $b$
gives $\partial{y}/\partial{b} = 0$,
since $\rord\ne 0$ and $\partial{a}/\partial{b} = 0$.
Differentiating $F_7=\unitthree+f_{12}$ in $b$ gives
$$\rord F_7 = (b-c)\partial{F_7(a,b)}/\partial{b}
= (b-c)\partial{f_{12}(y,\uloc)}/\partial{b}
= (b-c)(\partial{\uloc}/\partial{b}) (\partial{f_{12}}/\partial{\uloc}),$$
since $\partial{y}/\partial{b} = 0$.
Yet the function $\partial{\uloc}/\partial{b}$ near $x$
is regular away from $D_\jboundary$
(because $\uloc\in \mathcal{O}_{X,x}$,
and $(a,b)$ is a pair of regular \emph{coordinates} away from $D_\jboundary$).
Therefore, $(b-c)\partial{\uloc}/\partial{b}$ is either regular near $x$, or polar along $D_\jboundary$.
Since $\rord F_7\in \Gamma(U_2, \mathcal{O}_X)^\times$,
we conclude that
the \emph{regular} function $\partial{f_{12}}/\partial{\uloc}$ near $x$
is either invertible, or zero along $D_\jboundary$.
But by Proposition~\ref{PROP:G-induced-nonconstancy-of-a-rational-function-on-Dj},
the map $F_7 = \unitthree+f_{12}$ is nonconstant along $D_\jboundary$,
so the latter case is impossible
(since $f_{12}(y,\uloc)\vert_{D_\jboundary} = f_{12}(0,\uloc)$ locally).
Instead, $\partial{f_{12}}/\partial{\uloc}$ must be invertible.
So by the analytic inverse function theorem,
we may change variables from $\uloc$ to $\uloc'\defeq f_{12}(y,\uloc)$.
Yet $$(1 + f_{11}(y,\uloc)/\unittwo)^{\rord}
= 1 + f_{12}(y,\uloc)/\unitthree
= 1 + \uloc'/\unitthree.$$
Taking $\rord$th roots lets us change variables
from $\uloc'$ to $z\defeq f_{11}(y,\uloc)$.
Thus $a = \unitone y^{\rord}$ and $(b-c)/y^{\sord} = \unittwo+z$,
as desired.

(2):
This is similar.
First, note that $\rord = 0$ (since $\jboundary\in J_3$) and write
$$(b-c)/\tloc^{\sord} = \unitone+f_{13}(\tloc,\uloc),
\quad a = \unittwo+f_{14}(\tloc,\uloc),$$
say, where $D_\jboundary$ is locally cut out by $\tloc=0$.
Next, change variables from $\tloc$ to a suitable $y$
so that $b-c = \unitone y^{\sord}$ and $a = \unittwo+f_{15}(y,\uloc)$, say.
Here $\partial{y}/\partial{a}=0$, since $\sord\ne 0$ and $\partial{b}/\partial{a}=0$.
Now analyze the equation
$$1 = \partial{a}/\partial{a}
= \partial{f_{15}(y,\uloc)}/\partial{a}
= (\partial{\uloc}/\partial{a}) (\partial{f_{15}}/\partial{\uloc})$$
using the nonconstancy of $a\vert_{D_\jboundary}$ (guaranteed by Proposition~\ref{PROP:G-induced-nonconstancy-of-a-rational-function-on-Dj})
and the local identity $f_{15}(y,\uloc)\vert_{D_\jboundary} = f_{15}(0,\uloc)$,
to change variables from $\uloc$ to $z\defeq f_{15}(y,\uloc)$.

(3):
Given (1) and (2),
this is routine.
For $\Top\in \set{(b-c)\,\frac{\partial}{\partial b}, a\,\frac{\partial}{\partial a}}$,
use the chain rule and the fact that
$(\Top a)/a\in \ZZ$, $(\Top b)/(b-c)\in \ZZ$
are constant.
For $\Top = a\,\frac{\partial}{\partial b}$,
use $a/(b-c)\in \Gamma(U_2, \mathcal{O}_X)$ (given by Proposition~\ref{PROP:apply-lower-bound-on-anticanonical-Weil-divisor})
to reduce to the $\Top = (b-c)\,\frac{\partial}{\partial b}$ case.
(We note here that (3) is only new when $\Top = (b-c)\,\frac{\partial}{\partial b}$;
the other cases already appeared in \cite{chambert2002distribution}.)
\end{proof}

\begin{remark}
\label{RMK:generalize-basic-coordinates-to-local-fields}
The lemma generalizes to any
local field $\efield$ of characteristic $0$.
This gives a map $(y,z)\mapsto (a,b)$ that locally parameterizes $G(\efield)$.
Being a parameterization, this map must be injective away from $D$.
Taking $\efield=\CC$, we find that injectivity forces $\gcd(\rord,\sord)=1$.
\end{remark}

For the rest of \S\ref{SEC:new-geometry}, assume $D$ has strict normal crossings.
Using the inverse function theorem for formal power series with integral coefficients,\footnote{The particular statement we need is that if $\varphi_1\in R[[x,y]]$, $\varphi_1(0,0)=0$, and $\partial{\varphi_1}/\partial{y}\in R[[x,y]]^\times$, where $R$ is a ring, then there exists a $\varphi_2\in R[[x,y]]$ such that $\varphi_2(0,0)=0$ and $\varphi_1(x,\varphi_2(x,y))=\varphi_2(x,\varphi_1(x,y))=y$.
The proof is standard:
we first construct a right inverse $\varphi_2$ of $\varphi_1(x,\cdot)$,
degree by degree, by Hensel's lemma in $R[[x,y]]$;
iterating, we get a right inverse $\varphi_3$ of $\varphi_2(x,\cdot)$;
finally, a composition trick shows that $\varphi_3=\varphi_1$.}
instead of the inverse function theorem for analytic functions,
we can refine Lemma~\ref{LEM:convenient-local-analytic-coordinates-generically} over $\QQ_p$.
The result is Lemma~\ref{LEM:large-p-local-analytic-coordinates-generically} below.
Let $\QQ_p^{\map{ur}} \belongs \ol{\QQ}_p$ be the maximal unramified extension of $\QQ_p$ (in an algebraic closure $\ol{\QQ}_p$ of $\QQ_p$),
and let $\ZZ_p^{\map{ur}}$ be the integral closure of $\ZZ_p$ in $\QQ_p^{\map{ur}}$.
Explicitly, $\ZZ_p^{\map{ur}}$ (resp.~$\QQ_p^{\map{ur}}$)
is generated over $\ZZ_p$ (resp.~$\QQ_p$)
by the set of all roots of unity in $\ol{\QQ}_p$ of order coprime to $p$.
Moreover, the residue field $\ZZ_p^{\map{ur}}/p\ZZ_p^{\map{ur}}$
is an algebraic closure $\ol{\FF}_p$ of
the residue field $\ZZ_p/p\ZZ_p = \FF_p$.

Let $\redp\maps \mathscr{X}(\ZZ_p) \to \mathscr{X}(\FF_p)$
and $\redpur\maps \mathscr{X}(\ZZ_p^{\map{ur}}) \to \mathscr{X}(\ol{\FF}_p)$
be reduction modulo $p$.
If $\mathscr{V}$ is a scheme and $R$ is a ring,
then $\mathscr{V}\times R$, or $\mathscr{V}_R$,
denotes the \emph{base change} $\mathscr{V}\times \Spec(R)$.
If moreover $R\in \{\ZZ_p, \ZZ_p^{\map{ur}}\}$,
then we may identify $\mathscr{V}(R/pR)$
with the set of $R/pR$-valued closed points of
the special fiber $\mathscr{V}_{R/pR} \belongs \mathscr{V}_R$.
Finally, elements $f\in \Gamma(\mathscr{V},\mathcal{O}_{\mathscr{V}})$
will often be written as $f(x)$,
in order to emphasize that $f$ induces a map $\mathscr{V}(R)\to R,\;x\mapsto f(x)$.
This will all be convenient below,
for describing and studying certain local rings of schemes.

\begin{lemma}
\label{LEM:large-p-local-analytic-coordinates-generically}
Let $\jboundary$, $c$, $\rord$, $\sord$,
and $U_2$ be as in Lemma~\ref{LEM:convenient-local-analytic-coordinates-generically}.
Let $\mathscr{U}_2$ be the complement of the closure of $X\setminus U_2$ in $\mathscr{X}$.
Let $p$ be a large prime.
Let $R\in \{\ZZ_p, \ZZ_p^{\map{ur}}\}$
and $x_0\in \mathscr{U}_2(R/pR)$.
Then there exist an analytic isomorphism
$$\redR^{-1}(x_0) \to pR^2,
\quad x\mapsto (y,z),$$
and units $\unitone,\unittwo\in R^\times$ depending on $x_0$,
such that for all $x\in \redR^{-1}(x_0)$ the following hold:
\begin{enumerate}
\item if $x_0\in \mathscr{D}_\jboundary(R/pR)$
and $\jboundary\in J_1\cup J_2$,
then $a(x) = \unitone y^{\rord}$ and $b(x)-c = y^{\sord} (\unittwo+z)$;

\item if $x_0\in \mathscr{D}_\jboundary(R/pR)$
and $\jboundary\in J_3$,
then $b(x)-c = \unitone y^{\sord}$ and $a(x) = \unittwo+z$;
and

\item if $x_0\in \mathscr{G}(R/pR)$,
then $a(x) = \unitone+y$ and $b(x)-c = \unittwo+z$.
\end{enumerate}
\end{lemma}

\begin{proof}
Uniformity over $p$ and $x_0$ takes some care.
We use the quasi-compactness of $U_2$.
Let $V_2\belongs \mathscr{U}_2$ be a small affine open neighborhood of a point of $U_2$.
Then there exist two elements $\tloc(x),\uloc(x)\in \Gamma(V_2,\mathcal{O}_{V_2})$
such that
for every prime $p\gg 1$
and point $x'\in V_2(R)$,
the maximal ideal of the local ring $\mathcal{O}_{V_2\times R,\redR(x')}$
is generated by $p$, $\tloc(x)-\tloc(x')$, and $\uloc(x)-\uloc(x')$.\footnote{Thus
$\tloc(x)-\tloc(x')$ and $\uloc(x)-\uloc(x')$
are regular local coordinates for $(V_2)_{R/pR}$ at $\redR(x')\in V_2(R/pR)$.
This condition, when cast in terms of differentials,
is equivalent to the nonvanishing of some determinant modulo $p$,
and thus can be guaranteed by removing from $V_2$ the zero locus of the determinant if necessary.}
Moreover, since $D_\jboundary$ is smooth, we may choose $\tloc$ and $\uloc$
so that $\div(\tloc\vert_{V_2}) = \mathscr{D}_\jboundary\vert_{V_2}$
(if $V_2$ is small enough).

In general,
given a point $x_0 \in \mathscr{U}_2(R/pR)$,
we choose $V_2$ and $x'\in V_2(R)$
such that $\redR(x') = x_0$.
This is possible by Hensel's lemma,
provided that $p$ is large enough.

Now $a/\tloc^{\rord}$ and $(b-c)/\tloc^{\sord}$
are invertible regular functions over $\QQ$
(i.e.~morphisms $(V_2)_\QQ\to \GG_m$);
so are $\rord F_7$ and $(b-c)\partial{\uloc}/\partial{b}$
(resp.~$a$ and $\partial{\uloc}/\partial{a}$)
if $\jboundary\in J_1\cup J_2$
(resp.~$\jboundary\in J_3$).
So they are invertible regular
over some $\ZZ[1/A_6]$ with $A_6\in \ZZ_{\ge 1}$.
By the inverse function theorem in
$$R[[\tloc(x)-\tloc(x'),\uloc(x)-\uloc(x')]],$$
if $p\nmid A_6$,
the method of Lemma~\ref{LEM:convenient-local-analytic-coordinates-generically}
thus yields (1) and (2)
for each $x_0 \in (\mathscr{D}_\jboundary\cap V_2)(R/pR)$.

If $x_0 \in (\mathscr{G}\cap V_2)(R/pR)$
and $p\nmid A_6$,
then $a(x'), b(x')-c\in R^\times$,
so for (3) we can simply take
$\unitone \defeq a(x')$,
$\unittwo \defeq b(x')-c$,
$y \defeq a(x)-a(x')$,
and $z \defeq b(x)-b(x')$.
\end{proof}

We end with some local constancy analysis useful for Lemmas~\ref{LEM:local-constancy} and~\ref{LEM:localized-regular-and-height-derivative-estimates}.
Roughly speaking, we will cover $X$ by open sets on which all divisors $D_\jboundary$ behave compatibly.
On each such open set, we will prove two smoothness results for local height functions:
Lemma~\ref{LEM:large-p-localized-height-symmetry-constancy} for large primes $p$,
and Lemma~\ref{LEM:fixed-p-localized-height-symmetry-constancy} for small primes $p$.

\begin{definition}
\label{DEFN:C1U-maximizes-for-any-Dj-intersecting-U}
For any set $U\belongs X$, not necessarily open,
let $$\mathcal{C}^{(1)}(U)
\defeq \bigcap_{\jboundary\in J:\, D_\jboundary\cap U \ne \emptyset}
\mathcal{C}^{(0)}(\jboundary) \belongs \QQ,$$
where $\mathcal{C}^{(0)}(\jboundary)$ is the set defined in \eqref{C0j}.
\end{definition}

\begin{proposition}
\label{PROP:U-with-nonempty-c(U)-cover-X}
Let $\mathcal{C}^{(2)}
\defeq \set{\textnormal{open }U\belongs X:
\mathcal{C}^{(1)}(U)\ne \emptyset}$.
Then $\bigcup_{U\in \mathcal{C}^{(2)}} U = X$.
\end{proposition}

\begin{proof}
If $x\in X$, then $\mathcal{C}^{(1)}(U)=\mathcal{C}^{(1)}(\set{x})$
for any small enough open set $U\ni x$.
To conclude, it suffices to show that
$\mathcal{C}^{(1)}(\set{x})\ne \emptyset$ for all $x\in X$.
Suppose for contradiction that $\mathcal{C}^{(1)}(\set{x}) = \emptyset$
for some $x\in X$.
Then by Proposition~\ref{PROP:compute-arg-max-set-C0j},
we must have $x\in D_\iboundary\cap D_\jboundary$
for some $\iboundary\in I^c$ and $\jboundary\in I^{c'}$ with $c\ne c'$.
Now $\ord_{D_\iboundary}(b-c),\ord_{D_\jboundary}(b-c') > 0$,
by \eqref{Ic}.
But $D$ has strict normal crossings,
and $\iboundary\ne \jboundary$ (since $\mathcal{C}^{(0)}(\iboundary)
= \{c\} \ne \{c'\} = \mathcal{C}^{(0)}(\jboundary)$),
so $D_\iboundary\cap D_\jboundary\cap \div_\infty(b) = \emptyset$.
Thus $x\notin \div_\infty(b)$,
whence $b\in \mathcal{O}_{X,x}$.
So $c = b(x) = c'$, a contradiction.
\end{proof}

It turns out that
sets $U\in \mathcal{C}^{(2)}$ possess extra analytic symmetries.
To facilitate proofs, let $\Gamma_p(\mathfrak{M},\mathfrak{R})$ be the set of $\mathfrak{R}$-valued analytic functions on a $p$-adic analytic manifold $\mathfrak{M}$.

\begin{lemma}
\label{LEM:large-p-localized-height-symmetry-constancy}
Let $U\in \mathcal{C}^{(2)}$ and $c\in \mathcal{C}^{(1)}(U)$.
Let $\mathscr{U}$ be the complement of the closure of $X\setminus U$ in $\mathscr{X}$.
Then for all sufficiently large primes $p$, the following hold for all $x_1,x_2\in 1+p\ZZ_p$:
\begin{enumerate}
\item The map $\phi\maps (a,b)\mapsto (ax_1,c+(b-c)x_2)$ on $G(\QQ_p)$
maps $G(\QQ_p)\cap \mathscr{U}(\ZZ_p)$ to itself.

\item If $\jboundary\in J$, then $H_{D_\jboundary,p}(g) = H_{D_\jboundary,p}(\phi(g))$ for all $g\in G(\QQ_p)\cap \mathscr{U}(\ZZ_p)$.
\end{enumerate}
\end{lemma}

\begin{proof}


Let $B$ be the closure in $\mathscr{X}$ of the set of pairwise intersections of irreducible components of $D\cup \div_0(b-c)$.
Importantly for Hartogs-style arguments, $B_\QQ$ has dimension $0$.

Let $p$ be large.
Let $V\belongs \mathscr{U}_{\ZZ_p}$ be an affine open set
such that for each $\jboundary\in J$, there exist $\tloc_\jboundary,\uloc_\jboundary\in \Gamma(V,\mathcal{O}_V)$ with $\div(\tloc_\jboundary\vert_V) = \mathscr{D}_\jboundary\vert_V$ such that $\tloc_\jboundary$, $\uloc_\jboundary$ are regular coordinates for $V_{\FF_p}$ (as in the proof of Lemma~\ref{LEM:large-p-local-analytic-coordinates-generically}; note that $D_\jboundary$ is smooth).
By calculation in terms of the analytic local coordinates of Lemma~\ref{LEM:large-p-local-analytic-coordinates-generically} (applicable by Definition~\ref{DEFN:C1U-maximizes-for-any-Dj-intersecting-U}),
we find that for all $f_{16}\in \Gamma(V,\mathcal{O}_V)$
and $x'\in (V\setminus B)(\ZZ_p^{\map{ur}})$,
the rational functions
$(\tloc_\jboundary-\phi^\ast{\tloc_\jboundary})/(p\tloc_\jboundary)$,
$(f_{16} - \phi^\ast{f_{16}})/p$
on $V$ lie in
\begin{equation}
\label{EQN:rational-power-series-lies-in-local-ring}
\Rat(V)
\cap \ZZ_p^{\map{ur}}[[\tloc_\jboundary(x)-\tloc_\jboundary(x'),
\uloc_\jboundary(x)-\uloc_\jboundary(x')]]
= \Rat(V)\cap \mathcal{O}^\wedge_{V\times \ZZ_p^{\map{ur}},\redpur(x')}
= \mathcal{O}_{V,\redpur(x')},
\end{equation}
where $\mathcal{O}^\wedge$ denotes completion.
(The first equality in \eqref{EQN:rational-power-series-lies-in-local-ring} comes from
the maximal ideal $(p,
\tloc_\jboundary(x)-\tloc_\jboundary(x'),
\uloc_\jboundary(x)-\uloc_\jboundary(x'))$
of $\mathcal{O}_{V\times \ZZ_p^{\map{ur}},\redpur(x')}$.
For the second,
note that $\Rat(V) = \Frac(\mathcal{O}_{V,\redpur(x')})$,
and $\mathcal{O}_{V,\redpur(x')}\to \mathcal{O}^\wedge_{V\times \ZZ_p^{\map{ur}},\redpur(x')}$
is a faithfully flat injection of regular local rings.)
Thus $$\frac{\tloc_\jboundary-\phi^\ast{\tloc_\jboundary}}{p\tloc_\jboundary},
\frac{f_{16} - \phi^\ast{f_{16}}}{p}
\in \bigcap_{x'\in (V\setminus B)_{\FF_p}} \mathcal{O}_{V,x'},$$
since $(V\setminus B)(\ZZ_p^{\map{ur}})$ maps onto $(V\setminus B)(\ol{\FF}_p)$.
Yet $\dim(B_{\FF_p}) = 0 = \dim(V_{\FF_p})-2$,
so every prime divisor $\mathfrak{P}\belongs V$ that intersects $B_{\FF_p}$
must pass through some point of $(V\setminus B)_{\FF_p}$.
Upon writing each $\mathcal{O}_{V,x'}$ as the intersection of
its localizations at height $1$ primes,
we conclude that
$$\frac{\tloc_\jboundary-\phi^\ast{\tloc_\jboundary}}{p\tloc_\jboundary},
\frac{f_{16} - \phi^\ast{f_{16}}}{p}
\in \bigcap_{x'\in V_{\FF_p}} \mathcal{O}_{V,x'}.$$

Crucially, each function in $\bigcap_{x'\in V_{\FF_p}} \mathcal{O}_{V,x'}$
induces on $\ZZ_p$-points a function in $\Gamma_p(V(\ZZ_p),\ZZ_p)$.
By converting between
the local analytic coordinates in \eqref{EQN:rational-power-series-lies-in-local-ring}
and the global algebraic coordinates on $V$ given by $\Gamma(V,\mathcal{O}_V)$,
it follows that for every $x'\in V(\FF_p)$, the analytic map
\begin{equation*}
\phi\vert_{\redp^{-1}(x')\cap \phi^{-1}(V(\QQ_p))}
\maps \redp^{-1}(x')\cap \phi^{-1}(V(\QQ_p)) \to G(\QQ_p)\cap V(\QQ_p)
\end{equation*}
(where $\phi^{-1}(V(\QQ_p)) \defeq \set{g\in G(\QQ_p): \phi(g)\in V(\QQ_p)}$)
extends uniquely to an analytic map
$$\redp^{-1}(x')\to \redp^{-1}(x'),$$
since $\redp^{-1}(x')\cap \phi^{-1}(V(\QQ_p))$ is dense in $\redp^{-1}(x')$,
in the $p$-adic topology.
(Density follows from
the fact that $X(\QQ_p)\setminus \phi^{-1}(V(\QQ_p))$
is a proper algebraic subset of $X(\QQ_p)$.)

Thus $\phi$ extends (uniquely) to a map $V(\ZZ_p)\to V(\ZZ_p)$.
Covering $\mathscr{U}_{\ZZ_p}$ by $V$ yields (1).\footnote{A referee
showed us that for $p\gg 1$,
the map in (1) in fact induces a $\ZZ_p$-scheme morphism
$\mathscr{U}_{\ZZ_p}\to \mathscr{U}_{\ZZ_p}$.
The main additional observation needed is that
$(\tloc_\jboundary-\phi^\ast{\tloc_\jboundary})/(p\tloc_\jboundary),
(f_{16} - \phi^\ast{f_{16}})/p
\in \Gamma(\mathscr{G}_{\ZZ_p}\cap V,\mathcal{O}_V)$.}
Also, $(\phi^\ast{\tloc_\jboundary})/\tloc_\jboundary
\in 1+p\Gamma_p(V(\ZZ_p),\ZZ_p)$
implies $$\abs{\tloc_\jboundary(\phi(x))}_p
= \abs{\tloc_\jboundary(x)}_p$$ for all $x\in V(\ZZ_p)$,
whence (2) holds by Definition~\ref{DEFN:standard-local-Weil-height-for-boundary-divisors}.
\end{proof}




Lemma~\ref{LEM:large-p-localized-height-symmetry-constancy} leaves out small primes $p$, which we address by
a $p$-adic analytic Hartogs.
Let $\ord_0(f_{17})\in \ZZ_{\ge 0}\cup \set{\infty}$
be the degree of the leading homogeneous term of a power series $f_{17}$.

\begin{lemma}
\label{LEM:construct-Hartogs-type-poles}
Fix a finite extension $K/\QQ_p$ and an integer $n\ge 2$.
Fix $P,Q\in \mathcal{O}_K[[\bm{x}]]=\mathcal{O}_K[[x_1,\dots,x_n]]$.
Suppose $\gcd(P,Q)=1$ and
$\ord_0(Q) = \ord_0(Q\bmod{\pi_K}) \ge 1$,
where $\pi_K$ is a uniformizer of $\mathcal{O}_K$.
Then for some finite extension $L/K$,
the set $$\set{\bm{x}\in p\mathcal{O}_L^n: Q(\bm{x})=0, P(\bm{x})\ne 0}$$
has a limit point lying in
$(p\mathcal{O}_L^2\setminus 0) \times \set{0}^{n-2}$,
where $p\mathcal{O}_L^2\setminus 0 \defeq p\mathcal{O}_L^2\setminus \set{(0,0)}$.
\end{lemma}

\begin{proof}


If $Q=0$, then $P\in \mathcal{O}_K[[\bm{x}]]^\times$, so $P(\bm{x})\in \mathcal{O}_K^\times$ for all $\bm{x}\in p\mathcal{O}_K^n$, and the result is obvious.
Now suppose $Q\ne 0$.
Since $Q\notin \mathcal{O}_K[[\bm{x}]]^\times$, we must have $P\ne 0$.
Also, $Q$ is primitive (i.e.~$\pi_K\nmid Q$), since $\ord_0(Q\bmod{\pi_K}) = \ord_0(Q) < \infty$.
By removing factors of $\pi_K$ in $P$, we may assume $P$ is primitive too.
Let $P_0,Q_0\in \mathcal{O}_K[\bm{x}]$ be the lowest-degree \emph{primitive} homogeneous terms in $P$, $Q$, respectively.
Then in particular,
$$\deg(Q_0) = \ord_0(Q\bmod{\pi_K}) = \ord_0(Q).$$

After replacing $K$ with a finite extension if necessary,
there exists a $\bm{k}\in \mathcal{O}_K^{n-1} \times \set{0}$ such that $P_0(k_1,\dots,k_{n-1},1) Q_0(k_1,\dots,k_{n-1},1) \in \mathcal{O}_K^\times$.
Via the $\mathcal{O}_K$-linear automorphism $\bm{x}\mapsto \bm{x}+x_n\bm{k}$ of $\mathcal{O}_K[[\bm{x}]]$,
we may then assume $$P_0(0,\dots,0,1) Q_0(0,\dots,0,1) \in \mathcal{O}_K^\times.$$
Thus the power series $P,Q\bmod{(\pi_K,x_1,\dots,x_{n-1})}
\in (\mathcal{O}_K/\pi_K)[[x_n]]$
have orders $\ord_0(\cdot)$ equal to $\deg{P_0}$, $\deg{Q_0}$, respectively.
So by the Weierstrass preparation theorem
in $R[[x_n]]$,
where $R \defeq \mathcal{O}_K[[x_1,\dots,x_{n-1}]]$,
there exist unique \emph{monic} polynomials $P_1,Q_1\in R[x_n]$
of degrees $\deg{P_0}$, $\deg{Q_0}$, respectively,
with $$P/P_1, Q/Q_1\in \mathcal{O}_K[[\bm{x}]]^\times.$$

Let $P_{1,j}, Q_{1,j}\in R$ be the $x_n^j$ coefficients of $P_1$, $Q_1$, respectively.
Reduction modulo $(\pi_K,x_1,\dots,x_{n-1})$ forces $\pi_K\mid P_{1,j}(\bm{0})$ if $j<\deg{P_1}$, and $\pi_K\mid Q_{1,j}(\bm{0})$ if $j<\deg{Q_1}$.
Therefore, any divisor of $P_1Q_1$ in $R[x_n]$ either lies in $R^\times$ or the ideal $(\pi_K,x_1,\dots,x_n)$.
From $\gcd(P,Q)=1$ in $\mathcal{O}_K[[\bm{x}]]$,
we thus get $\gcd(P_1,Q_1)=1$ in $R[x_n]$.
So by Gauss' lemma,
we have $\gcd(P_1,Q_1)=1$ in $\Frac(R)[x_n]$,
a Euclidean domain, whence
$$A P_1 + B Q_1 = C$$
for some $A,B,C\in R$ with $C\ne 0$.
On the other hand,
$Q/Q_1\in \mathcal{O}_K[[\bm{x}]]^\times$ implies
$$\ord_0(Q_1) = \ord_0(Q) = \deg(Q_0) = \deg(Q_1)$$
(if we define $\ord_0(Q_1)$ viewing $Q_1$ as a power series in $\bm{x}$),
whence $$\ord_0(Q_{1,j})\ge \deg(Q_1) - j \ge 1$$
(so $Q_{1,j}(\bm{0})=0$)
for all $j<\deg{Q_1}$.
Let $$f_{18,j}\defeq Q_{1,j}\vert_{x_3=\dots=x_{n-1}=0}\in \mathcal{O}_K[[x_1,x_2]].$$

The idea now is to choose points $\bm{x}$ at which $Q_1=0$ and $C\ne 0$.
Let $L(d)$ be the compositum of all degree $\le d$ extensions of $K$ in $\ol{K}$;
then $L(d)/K$ is finite (by Krasner's lemma).

\emph{Case~1: $n\ge 3$, and $f_{18,j} = 0$ for all $j<\deg{Q_1}$.}
For each large integer $r\ge 1$,
choose $$x_1,x_2\in p+p^r\mathcal{O}_K^\times,
\quad x_3,\dots,x_{n-1}\in p^r\mathcal{O}_K^\times$$
with $C\ne 0$; then choose $x_n\in \ol{K}$ with $Q_1=0$.
(Roughly, $v_p(x_3),\dots,v_p(x_{n-1})$ ensure $Q_1(x_1,\dots,x_n)$ resembles a monomial in $x_n$ of degree $\deg{Q_1}$.)
Letting $r\to \infty$ produces an infinite sequence of $\bm{x}\in p\mathcal{O}_{L(\deg{Q_1})}^n$ tending to $(p,p,0,\dots,0)$.

\emph{Case~2: $n\ge 3$, and $f_{18,j} \ne 0$ for some $j<\deg{Q_1}$.}
Take $j<\deg{Q_1}$ minimal with $f_{18,j} \ne 0$.
Via a linear automorphism $(x_1,x_2)\mapsto (x_1,x_2+kx_1)$ with $k\in \mathcal{O}_K$, assume $f_{18,j}\vert_{x_2=0} \ne 0$.
Choose $c\in \ZZ_{\ge 0}$ such that
$f_{18,j}(p^c y,0)\in \mathcal{O}_K[[y]]$ is divisible by its own leading term.
Given large integers $r,d\ge 1$,
choose $$x_2\in p^d+p^r\mathcal{O}_K^\times,
\quad x_3,\dots,x_{n-1}\in p^{dr}\mathcal{O}_K^\times$$
with $C\in \mathcal{O}_K[[x_1]]\setminus 0$.
Then $C$ has only finitely many zeros $x_1\in \ol{K}$ with $v_p(x_1)>0$,
and for each such $x_1$ the set $\set{x_n\in \ol{K}: Q_1=0}$ is finite.
Thus there exists $x_n\in p^r\mathcal{O}_K^\times$ such that
$$\set{x_1\in \ol{K}: v_p(x_1)>0,\; Q_1=C=0} = \emptyset.$$
Next, if $r$ and $d$ are large enough, then by Weierstrass preparation in $\mathcal{O}_K[[y]]$ (applied to $Q_1(p^c y,x_2,\dots,x_n)/x_n^j$) and the fact $Q_{1,j}(\bm{0})=0$, there exists $x_1\in p^{c+1}\mathcal{O}_{L(d)}$ with $Q_1=0$.
(Roughly, $v_p(x_2),\dots,v_p(x_n)$ and the minimality of $j$ ensure
$Q_1/x_n^j$ resembles a monomial in $y$ of degree $\ord_0(f_{18,j}\vert_{x_2=0})$.)
Taking $r\to \infty$ produces $\bm{x}\in p\mathcal{O}_{L(d)}^n$ with $1\ll \max(\abs{x_1}_p, \abs{x_2}_p)\ll 1$ and $x_3,\dots,x_n\to 0$.
Now pass to a convergent subsequence (by compactness).


\emph{Case~3: $n=2$.}
Here $\#\set{x_1\in \ol{K}: v_p(x_1)>0,\; C=0} < \infty$.
So there are infinitely many $x_1\in p^{\deg{Q_1}}\mathcal{O}_K^\times$ with $C\ne 0$.
For each such $x_1$, there exists $x_2\in p\mathcal{O}_{L(\deg{Q_1})}$ with $Q_1=0$.
(It is important here that $Q_{1,j}(0)=0$ for all $j<\deg{Q_1}$.)
By compactness, this suffices.
\end{proof}

\begin{proposition}
\label{PROP:p-adic-analytic-Hartogs}
Let $f_{19}\in \Frac(\ZZ_p[[\xi_1,\dots,\xi_N,y,z]])$,
where $N\in \ZZ_{\ge 0}$.
Suppose that for every $l\in \ZZ_{\ge 1}$ and finite extension $\efield/\QQ_p$,
there exists $m=m(l,\efield)\in \ZZ_{\ge 1}$ such that
\begin{equation}
\label{f29-integrality-condition}
f_{19}\in \Gamma_p(p^m\mathcal{O}_{\efield}^N
\times (p^l\mathcal{O}_{\efield}^2 \setminus p^{l+1}\mathcal{O}_{\efield}^2),
\mathcal{O}_{\efield}).
\end{equation}
Then there exists $l\in \ZZ_{\ge 1}$ such that
$$f_{19}\in \Gamma_p(p^l\ZZ_p^{N+2},\ZZ_p).$$
\end{proposition}

\begin{proof}
Write $f_{19}=\varsigma P/Q$ with $P,Q\in \ZZ_p[[\xi_1,\dots,z]]$, $Q\ne 0$, and $\varsigma\in \QQ_p^\times$.
By a suitable scaling $(\xi_1,\dots,z)\mapsto (p^d\xi_1,\dots,p^d z)$, we may assume
$\ord_0(Q)=\ord_0(Q\bmod{p})$.
After dividing out $\gcd(P,Q)$, and applying Gauss' lemma to the leading homogeneous terms of $Q$, $\gcd(P,Q)$, $Q/\gcd(P,Q)$, we may assume that $\gcd(P,Q)=1$ and $\ord_0(Q)=\ord_0(Q\bmod{p})$.

\emph{Case~1: $Q(\bm{0}) \ne 0$.}
Then $Q(\bm{0}) \in \ZZ_p^\times$,
so $f_{19}\in \ZZ_p[[\xi_1,\dots,z]] \otimes \QQ_p$.
Hence $f_{19} \in f_{19}(\bm{0}) + \Gamma_p(p^l\ZZ_p^{N+2},\ZZ_p)$
if $l$ is large enough.
Yet $f_{19}(0,\dots,0,p^l)\in \ZZ_p$ for all $l\ge 1$,
by \eqref{f29-integrality-condition}.
We conclude that $f_{19}(\bm{0})\in \ZZ_p$,
and thus $f_{19} \in \Gamma_p(p^l\ZZ_p^{N+2},\ZZ_p)$ for any sufficiently large $l$.

\emph{Case~2: $Q(\bm{0}) = 0$.}
Then $\ord_0(Q) \ge 1$.
So by the $(K,n) = (\QQ_p,N+2)$ case of Lemma~\ref{LEM:construct-Hartogs-type-poles},
there exist a finite extension $L/\QQ_p$,
an integer $l\ge 1$,
and a sequence of points
$\bm{x}^{(m)}\in p^m\mathcal{O}_L^N
\times (p^l\mathcal{O}_L^2 \setminus p^{l+1}\mathcal{O}_L^2)$
indexed by integers $m\ge 1$,
such that $Q(\bm{x}^{(m)}) = 0$ and $P(\bm{x}^{(m)})\ne 0$
for all $m\ge 1$.
This contradicts the $m=m(l,L)$ case of \eqref{f29-integrality-condition}.
\end{proof}

\begin{lemma}
\label{LEM:fixed-p-localized-height-symmetry-constancy}
Let $U\in \mathcal{C}^{(2)}$
and $c\in \mathcal{C}^{(1)}(U)$.
Fix a prime $p$
and a compact open set $\oset_p\belongs U(\QQ_p)$.
Then there exists an integer $A_7=A_7(U,c,p,\oset_p)\ge 1$
such that the following hold for all $x_1,x_2\in 1+p^{A_7}\ZZ_p$:
\begin{enumerate}
\item The map $\phi\maps (a,b)\mapsto (ax_1,c+(b-c)x_2)$ on $G(\QQ_p)$
maps $G(\QQ_p)\cap \oset_p$ to itself.

\item If $\jboundary\in J$, then $H_{D_\jboundary,p}(g) = H_{D_\jboundary,p}(\phi(g))$ for all $g\in G(\QQ_p)\cap \oset_p$.
\end{enumerate}
\end{lemma}

\begin{proof}
Given a point $x'\in \oset_p$,
let $\oset'_p\cong p\ZZ_p^2$
denote a small compact open neighborhood of $x'$ in $\oset_p$,
where $x'$ is identified with the origin $(0,0)\in p\ZZ_p^2$.
Let $V\belongs U_{\QQ_p}$ be an affine open set such that
for each $\jboundary\in J$ there exists $\tloc_\jboundary\in \Gamma(V,\mathcal{O}_V)$
with $\div(\tloc_\jboundary\vert_V) = D_\jboundary\vert_V$.
Let $f_{20,1},\dots,f_{20,k}\in \Gamma(V,\mathcal{O}_V)$ generate $\Gamma(V,\mathcal{O}_V)$ as a $\QQ_p$-algebra.
Assume $\oset'_p\belongs V(\QQ_p)$.
If $\oset'_p\cong p\ZZ_p^2$ is sufficiently small,
then for each finite extension $\efield/\QQ_p$,
we may view $p\mathcal{O}_{\efield}^2$ as an open neighborhood of $\oset'_p$ in $V(\efield)$.
Let $B\belongs \mathscr{X}$
be as in the proof of Lemma~\ref{LEM:large-p-localized-height-symmetry-constancy};
then $B(\ol{\QQ}_p)$ is a finite set.
By shrinking $\oset'_p\cong p\ZZ_p^2$ if necessary,
we may assume that
$$p\mathcal{O}_{\ol{\QQ}_p}^2\cap B(\ol{\QQ}_p)
\belongs \set{x'}.$$
That is, we assume that $\oset'_p$
(and its neighborhood $p\mathcal{O}_{\efield}^2$ for every $\efield/\QQ_p$)
is disjoint from $B(\ol{\QQ}_p)$,
except in the case where $x'\in B(\QQ_p)$ to begin with.

Algebraically, view $\phi$ now as a rational map $\Aff^2 \times X \dashrightarrow X$
defined by the formula $$(x_1,x_2,a,b)\mapsto (ax_1,c+(b-c)x_2)$$ on $(a,b)\in G$.
A $p$-adic Lemma~\ref{LEM:convenient-local-analytic-coordinates-generically} (see Remark~\ref{RMK:generalize-basic-coordinates-to-local-fields})
shows that for each $\efield/\QQ_p$ and $l\in \ZZ_{\ge 1}$, there exists $m\in \ZZ_{\ge 1}$ such that
the rational functions
$(\tloc_\jboundary-\phi^\ast{\tloc_\jboundary})/(p\tloc_\jboundary)$
and $(f_{20,i} - \phi^\ast{f_{20,i}})/p$ on $\Aff^2 \times V$
lie in $\Gamma_p((1+p^m\mathcal{O}_\efield)^2
\times (p^l\mathcal{O}_\efield^2\setminus p^{l+1}\mathcal{O}_\efield^2),
\mathcal{O}_\efield)$.
By the $N=2$ case of Proposition~\ref{PROP:p-adic-analytic-Hartogs},
we get $$(\tloc_\jboundary-\phi^\ast{\tloc_\jboundary})/(p\tloc_\jboundary),
(f_{20,i} - \phi^\ast{f_{20,i}})/p
\in \Gamma_p((1+p^l\ZZ_p)^2 \times p^l\ZZ_p^2, \ZZ_p)$$
for some $l=l(x')\ge 1$.
In particular, if $\oset''_p\cong p^l\ZZ_p^2\belongs p\ZZ_p^2\cong \oset'_p$,
then $\phi$ extends uniquely to an analytic map
$\Phi\maps (1+p^l\ZZ_p)^2 \times \oset''_p\to \oset''_p$,
since $((1+p^l\ZZ_p)^2 \times \oset''_p)\cap \phi^{-1}(V(\QQ_p))$
is dense in $(1+p^l\ZZ_p)^2 \times \oset''_p$.
In addition,
by Definition~\ref{DEFN:standard-local-Weil-height-for-boundary-divisors},
we have
$H_{D_\jboundary,p} = \Phi^\ast{H_{D_\jboundary,p}}$ on $(1+p^{l'}\ZZ_p)^2 \times \oset''_p$
for some integer $l'=l'(x')\ge l$.
So (1) and (2) hold with $A_7\defeq \max_{x'\in \oset^\ast_p} l'(x')$,
where $\oset^\ast_p\belongs \oset_p$ is a finite set
of points $x'\in \oset_p$ whose associated neighborhoods $\oset''_p\ni x'$
cover $\oset_p$.
\end{proof}

\section{Non-archimedean local calculations}
\label{SEC:non-archimedean-local-calculations}


Recall the definition $\mathsf{u}_\jboundary\defeq \ord_{D_\jboundary}(a)$
from \S\ref{SUBSEC:polar-combinatorics}.
Throughout \S\ref{SEC:non-archimedean-local-calculations},
assume $X$ is split and $D$ has strict normal crossings,
fix $\delta \in \openint(0,1)$ and $A_8\in \openint(1,\infty)$,
and let $\bm{s}\in \CC^J$ be such that
\begin{equation}
\label{INEQ:key-new-region}
-\delta \le \Re(s_\jboundary-\mathsf{d}_\jboundary-2\mathsf{u}_\jboundary) \le A_8
\qquad\textnormal{for all $\jboundary\in J$}.
\end{equation}
In particular,
$\Re(s_\jboundary-\mathsf{d}_\jboundary+1) > 0$ for $\jboundary\in J_1\cup J_3$,
but not always for $\jboundary\in J_2$.
This asymmetry turns out to be natural,
due to the integrality condition $\Nbad \alpha a_p\in \ZZ_p$ in \eqref{EXPR:formula-for-H^vee_p}.
Ultimately, the translate $2\mathsf{u}_\jboundary$ in \eqref{INEQ:key-new-region}
will ensure satisfactory bounds in
Lemmas~\ref{LEM:denominator-bias},
\ref{LEM:numerator-bias},
and~\ref{LEM:generic-local-factor-estimate}.

For any $\delta>0$, the region \eqref{INEQ:key-new-region} goes beyond the region $\Omega'_\eps$ considered in \cite{tanimoto2012height}*{Lemma~5.11}.
Our eventual success thus relies on precise calculations, revealing leading-order structure like in \cite{wang2024_isolating_special_solutions}, that we carry out using
a new source of cancellation in the local integral \eqref{EXPR:formula-for-H^vee_p}: the nonconstancy identified in Proposition~\ref{PROP:G-induced-nonconstancy-of-a-rational-function-on-Dj}, fed into \cite{bombieri1966exponential}.

Let $(\lambda, \alpha) \in \Mbad \times \QQ^\times$.
Over \eqref{INEQ:key-new-region}, it turns out that the most important factors $H^\vee_p(\bm{s},\lambda,t,\alpha)$ in \eqref{EQN:difficult-part-of-spectral-expansion} are those with $p$ large and $v_p(\alpha)$ small but nonzero.
But we also need reasonable control on other factors.
To get our bearings,
we first prove a useful bound
(applicable to \eqref{EXPR:formula-for-H^vee_p} after taking absolute values)
valid for arbitrary $p$ and $\alpha$,
which does not require any subtle cancellations:

\begin{lemma}
\label{LEM:general-fixed-p-bound}
Then $\int_{G(\QQ_p):\, \Nbad \alpha a_p\in \ZZ_p} \abs{H_p(\bm{s},g_p)}^{-1}\, dg_p$
is (for all $\eps>0$)
\begin{equation*}
\ll_{p,\eps} \abs{\alpha}_p^{-\eps} \bm{1}_{\abs{\alpha}_p \le 1}
+ \abs{\alpha}_p^\eps \sum_{\jboundary\in J_1} \abs{\alpha}_p^{-\Re(s_\jboundary-\mathsf{d}_\jboundary+1)/\mathsf{u}_\jboundary} \bm{1}_{\abs{\alpha}_p > 1}
+ \abs{\alpha}_p^{-\eps} \sum_{\jboundary\in J_2} \abs{\alpha}_p^{\Re(s_\jboundary-\mathsf{d}_\jboundary+1)/\abs{\mathsf{u}_\jboundary}} \bm{1}_{\abs{\alpha}_p < 1}.
\end{equation*}
\end{lemma}


\begin{proof}
Since $\abs{H_p(\bm{s},g)} = H_p(\Re(\bm{s}),g)$, we may assume $\bm{s}$ is real.
Let $U\belongs X(\QQ_p)$ be small open sets $\cong \ZZ_p^2$ forming a cover $\mathscr{C}_p$ of $X(\QQ_p)$.
Writing $U\cap D \belongs D_\kboundary \cup D_\lboundary$
for some $\kboundary,\lboundary\in J$ depending on $U$
(possible since $D$ has strict normal crossings),
we get\footnote{In \S\ref{SEC:non-archimedean-local-calculations},
we let $dg\defeq dg_p$
and generally omit subscripts of $p$ from variables and measures.}
\begin{equation*}
\int_{G(\QQ_p):\, \Nbad \alpha a\in \ZZ_p} \frac{dg}{\abs{H_p(\bm{s},g)}}
\ll_p \sum_{U\in \mathscr{C}_p}
\int_{\ZZ_p^2} \abs{y}_p^{(s_\kboundary-\mathsf{d}_\kboundary)\varsigma_\kboundary} \abs{z}_p^{(s_\lboundary-\mathsf{d}_\lboundary)\varsigma_\lboundary}
\bm{1}_{\Nbad'\alpha y^{\mathsf{u}_\kboundary\varsigma_\kboundary} z^{\mathsf{u}_\lboundary\varsigma_\lboundary} \in \ZZ_p} \, dy\, dz,
\end{equation*}
where $\varsigma_\jboundary\defeq \bm{1}_{U\cap D_\jboundary \ne \emptyset}$ for $\jboundary\in \set{\kboundary,\lboundary}$,
where $y$ (resp.~$z$) is a local parameter for $D_\kboundary$ (resp.~$D_\lboundary$) if $\varsigma_\kboundary=1$ (resp.~$\varsigma_\lboundary=1$),
and where $\Nbad'\ll_p \Nbad$ is a nonzero multiple of $\Nbad$ such that
the local analytic function $\Nbad' y^{\mathsf{u}_\kboundary\varsigma_\kboundary} z^{\mathsf{u}_\lboundary\varsigma_\lboundary}/(\Nbad a)$ is $\ZZ_p$-valued.
Given $U$, $\kboundary$, $\lboundary$, call the integral on the right $I(U,\kboundary,\lboundary)$.
The contribution to $I(U,\kboundary,\lboundary)$ from $\abs{y}_p = Y$ and $\abs{z}_p = Z$,
where $Y,Z\in \RR_{>0}$,
is $$(1-p^{-1})^2 \cdot f_{21}(Y,Z)
\cdot \bm{1}_{(Y,Z)\in \mathfrak{S}_\ZZ(\Nbad'\alpha)},$$
where $f_{21}(Y,Z) \defeq
Y^{(s_\kboundary-\mathsf{d}_\kboundary)\varsigma_\kboundary+1} Z^{(s_\lboundary-\mathsf{d}_\lboundary)\varsigma_\lboundary+1}$
and for any set $R\belongs \RR$ we let
\begin{equation}
\label{envelope-R}
\mathfrak{S}_R(\alpha) \defeq \set{(Y,Z)\in (0,1]^2:
\abs{\alpha}_p\, Y^{\mathsf{u}_\kboundary\varsigma_\kboundary}Z^{\mathsf{u}_\lboundary\varsigma_\lboundary} \le 1,
\; (\log_p{Y},\log_p{Z})\in R^2}.
\end{equation}

We now apply the $(d,e) = (2,3)$ case of Lemma~\ref{LP-axiom} below.
(In particular, $V=\RR^2$.)
Take half-planes $\hbar_1,\hbar_2,\hbar_3\belongs \RR^2$ corresponding to
the half-regions $\{Y\le 1\}$, $\{Z\le 1\}$, and $\{\abs{\Nbad'\alpha}_p\, Y^{\mathsf{u}_\kboundary\varsigma_\kboundary}Z^{\mathsf{u}_\lboundary\varsigma_\lboundary} \le 1\}$ in $\RR_{>0}^2$.
Then $\mathfrak{S}'_R = \mathfrak{S}_R(\Nbad'\alpha)$.
Using \eqref{lebesgue-sum} and \eqref{R-vertex-bounding} if $\mathfrak{S}'_\RR\ne \emptyset$, or the trivial fact that $I(U,\kboundary,\lboundary) = 0$ if $\mathfrak{S}'_\RR=\emptyset$, we find that
\begin{equation*}
I(U,\kboundary,\lboundary)
\ll_\eps (\abs{\Nbad'\alpha}_p^{-1}+\abs{\Nbad'\alpha}_p)^{-1/\eps}
+ (\abs{\Nbad'\alpha}_p^{-1}+\abs{\Nbad'\alpha}_p)^\eps \cdot (M_1+M_2+M_3),
\end{equation*}
where $M_1 = f_{22}(1,1)$,
$M_2 = f_{22}(\abs{\Nbad'\alpha}_p^{-1/\mathsf{u}_\kboundary},1) \bm{1}_{\mathsf{u}_\kboundary\varsigma_\kboundary\ne 0}$,
and $M_3 = f_{22}(1,\abs{\Nbad'\alpha}_p^{-1/\mathsf{u}_\lboundary}) \bm{1}_{\mathsf{u}_\lboundary\varsigma_\lboundary\ne 0}$,
where $f_{22}(x)\defeq f_{21}(x) \bm{1}_{x\in \mathfrak{S}_\RR(\Nbad'\alpha)}$.
Here $M_1 = \bm{1}_{\abs{\Nbad'\alpha}_p\le 1}$.
Also, $M_2$ (resp.~$M_3$) equals
\begin{equation*}
\abs{\Nbad'\alpha}_p^{-(s_\jboundary-\mathsf{d}_\jboundary+1)/\mathsf{u}_\jboundary}
(\bm{1}_{\jboundary\in J_1} \bm{1}_{\abs{\Nbad'\alpha}_p\ge 1}
+ \bm{1}_{\jboundary\in J_2} \bm{1}_{\abs{\Nbad'\alpha}_p\le 1})
\bm{1}_{\mathsf{u}_\jboundary\varsigma_\jboundary\ne 0}
\end{equation*}
for $\jboundary=\kboundary$ (resp.~$\jboundary=\lboundary$).
This suffices, since $\card{\mathscr{C}_p},\Nbad' \ll_p 1$ (and $\bm{s}\ll 1$) and
\begin{equation*}
\begin{split}
(\abs{\alpha}_p^{-1}+\abs{\alpha}_p)^{-1/\eps}
&\le \bm{1}_{\abs{\alpha}_p \le 1}
+ \sum_{i\in J_1} \abs{\alpha}_p^{-(s_i-\mathsf{d}_i+1)/\mathsf{u}_i} \bm{1}_{\abs{\alpha}_p > 1}, \\
\bm{1}_{\abs{\Nbad'\alpha}_p \le 1}
&\le \bm{1}_{\abs{\alpha}_p \le 1}
+ \sum_{i\in J_1} \abs{\Nbad'\alpha}_p^{-(s_i-\mathsf{d}_i+1)/\mathsf{u}_i} \bm{1}_{\abs{\alpha}_p > 1}, \\
\abs{\alpha}_p^{-(s_\jboundary-\mathsf{d}_\jboundary+1)/\mathsf{u}_\jboundary} \bm{1}_{\jboundary\in J_1} \bm{1}_{\abs{\Nbad'\alpha}_p\ge 1}
&\le \bm{1}_{\abs{\alpha}_p = 1}
+ \sum_{i\in J_1} \abs{\alpha}_p^{-(s_i-\mathsf{d}_i+1)/\mathsf{u}_i} \bm{1}_{\abs{\alpha}_p > 1}, \\
\abs{\alpha}_p^{-(s_\jboundary-\mathsf{d}_\jboundary+1)/\mathsf{u}_i} \bm{1}_{\jboundary\in J_2} \bm{1}_{\abs{\Nbad'\alpha}_p\le 1}
&\ll_p \bm{1}_{\abs{\alpha}_p \le 1}
+ \sum_{i\in J_1} \abs{\alpha}_p^{-(s_i-\mathsf{d}_i+1)/\mathsf{u}_i} \bm{1}_{\abs{\alpha}_p > 1}
+ \sum_{i\in J_2} \abs{\alpha}_p^{(s_i-\mathsf{d}_i+1)/\abs{\mathsf{u}_i}} \bm{1}_{\abs{\alpha}_p < 1}.
\end{split}
\end{equation*}
(These final displayed inequalities follow from \eqref{INEQ:key-new-region} and the fact $J_1\ne \emptyset$.)
\end{proof}

The goal of the following lemma
is to bound a certain class of $p$-adic integrals
via linear programming over a suitable region in $\RR^2$.

\begin{lemma}
\label{LP-axiom}
Assume $\bm{s}$ is real.
Let $\kboundary,\lboundary\in J$.
Let $\varsigma_\kboundary,\varsigma_\lboundary\in \{0,1\}$.
Let $\mathfrak{S}_R(\alpha)$ be as in \eqref{envelope-R}.
Let $d,e\ge 1$ be integers.
Let $V\belongs \RR^2$ be an affine space of dimension $d$.
Let $\hbar_1,\dots,\hbar_e\belongs \RR^2$ be closed half-planes,
whose boundaries are the lines $\ell_1,\dots,\ell_e\belongs \RR^2$, respectively.
Let $$\mathfrak{S}'_R \defeq \set{(Y,Z)\in \RR_{>0}^2:
(\log_p{Y},\log_p{Z})\in R^2\cap V\cap \hbar_1\cap \dots\cap \hbar_e}.$$
Assume $\mathfrak{S}'_\RR\belongs \mathfrak{S}_\RR(\alpha)$
and $\mathfrak{S}'_\RR\ne \emptyset$.
Then $\max_{\mathfrak{S}'_\RR}{f_{21}}$ exists,
and for all $\eps>0$ we have
\begin{equation}
\label{lebesgue-sum}
\int_{(\abs{y}_p,\abs{z}_p)\in \mathfrak{S}'_\ZZ} \abs{y}_p^{(s_\kboundary-\mathsf{d}_\kboundary)\varsigma_\kboundary} \abs{z}_p^{(s_\lboundary-\mathsf{d}_\lboundary)\varsigma_\lboundary} \, dy\, dz
\ll_\eps (\abs{\alpha}_p^{-1}+\abs{\alpha}_p)^{-1/\eps}
+ (\abs{\alpha}_p^{-1}+\abs{\alpha}_p)^\eps
\max_{\mathfrak{S}'_\RR}{f_{21}}.
\end{equation}
For any set $\mathcal{E}\belongs \{1,\dots,e\}$,
let $\ell_{\mathcal{E}}\defeq \bigcap_{e'\in \mathcal{E}} \ell_{e'}\belongs \RR^2$.
Call $\mathcal{E}$ \emph{extreme} if
\begin{equation*}
\begin{split}
(V\cap \hbar_1\cap \dots\cap \hbar_e)\cap \ell_{\mathcal{E}} &\ne \emptyset, \\
(V\cap \hbar_1\cap \dots\cap \hbar_e)\cap \ell_{\mathcal{E}\cup \{e''\}} &= \emptyset,
\end{split}
\end{equation*}
for all $e''\in \{1,\dots,e\}\setminus \mathcal{E}$.
Then
\begin{equation}
\label{R-vertex-bounding}
\max_{\mathfrak{S}'_\RR}{f_{21}}
\le \sum_{\textnormal{extreme }\mathcal{E}} f_{21}(p^{x_{\mathcal{E}}}),
\end{equation}
for some choice of points
$x_{\mathcal{E}}\in V\cap \ell_{\mathcal{E}}$.
\end{lemma}

\begin{proof}
Let $\mathfrak{S}_R\defeq \mathfrak{S}_R(\alpha)$ for convenience.
We first prove that if $L\in \RR_{>0}$ and $(Y,Z)\in \mathfrak{S}_\RR$
with $L\le f_{21}\le 2L$, then
\begin{equation}
\label{compact-envelope}
\log_p{Y},\log_p{Z}
\ll_\eps (1+\abs{\alpha}_p^{-1}+L^{-1})^\eps.
\end{equation}
Order $\kboundary$, $\lboundary$ so that $\mathsf{u}_\kboundary\varsigma_\kboundary \ge \mathsf{u}_\lboundary\varsigma_\lboundary$.
We prove \eqref{compact-envelope} by casework.

\emph{Case~1: $\mathsf{u}_\kboundary\varsigma_\kboundary<0$.}
Then $\mathfrak{S}_\RR = \emptyset$ unless $\abs{\alpha}_p\le 1$.
Also, $\log_p{Y},\log_p{Z}\ll \abs{v_p(\alpha)}$ for $(Y,Z)\in \mathfrak{S}_\RR$.
The claim \eqref{compact-envelope} follows,
since if $\abs{\alpha}_p\le 1$,
then $$\abs{v_p(\alpha)} = \log_p(1/\abs{\alpha}_p)
\ll \log(1/\abs{\alpha}_p)
\ll_\eps \abs{\alpha}_p^{-\eps}.$$

\emph{Case~2: $\mathsf{u}_\kboundary\varsigma_\kboundary\ge 0$.}
Then $(s_\kboundary-\mathsf{d}_\kboundary)\varsigma_\kboundary\ge -\delta$,
since either $\varsigma_\kboundary=0$ or $\kboundary\in J_1\cup J_3$.
Now suppose $f_{21}\asymp L$ for some $(Y,Z)\in \mathfrak{S}_\RR$;
then $Y \asymp L^{e_1}/Z^{e_2}$
and $Y^{\mathsf{u}_\kboundary\varsigma_\kboundary}Z^{\mathsf{u}_\lboundary\varsigma_\lboundary} \asymp L^{e_3}/Z^{e_4}$,
where
\begin{equation*}
e_1 = ((s_\kboundary-\mathsf{d}_\kboundary)\varsigma_\kboundary+1)^{-1},
\; e_2 = ((s_\lboundary-\mathsf{d}_\lboundary)\varsigma_\lboundary+1) e_1,
\; e_3 = \mathsf{u}_\kboundary\varsigma_\kboundary e_1,
\; e_4 = \mathsf{u}_\kboundary\varsigma_\kboundary e_2 - \mathsf{u}_\lboundary\varsigma_\lboundary.
\end{equation*}
If $\mathsf{u}_\lboundary\varsigma_\lboundary\ge 0$,
then $(s_\lboundary-\mathsf{d}_\lboundary)\varsigma_\lboundary\ge -\delta$,
so the constraint $Y\le 1$ in $\mathfrak{S}_\RR$
implies $Z\gg L^{e_1/e_2}$,
where $0<\frac{e_1}{e_2}\le \frac{1}{1-\delta}$.
Meanwhile, if $\mathsf{u}_\lboundary\varsigma_\lboundary<0$, then $\varsigma_\lboundary=1$ and (using \eqref{INEQ:key-new-region} for $\jboundary\in \set{\kboundary,\lboundary}$)
\begin{equation*}
\begin{split}
\tfrac{e_4}{e_1}
= \mathsf{u}_\kboundary\varsigma_\kboundary \tfrac{e_2}{e_1} + \abs{\mathsf{u}_\lboundary} e_1^{-1}
&\ge \mathsf{u}_\kboundary\varsigma_\kboundary (2\mathsf{u}_\lboundary+1-\delta) + \abs{\mathsf{u}_\lboundary} ((2\mathsf{u}_\kboundary-\delta)\varsigma_\kboundary+1) \\
&= \mathsf{u}_\kboundary\varsigma_\kboundary (1-\delta) + \abs{\mathsf{u}_\lboundary} (-\delta\varsigma_\kboundary+1) \\
&\ge \mathsf{u}_\kboundary\varsigma_\kboundary (1-\delta) + \abs{\mathsf{u}_\lboundary} (1-\delta)
= (\mathsf{u}_\kboundary\varsigma_\kboundary + \abs{\mathsf{u}_\lboundary}) (1-\delta),
\end{split}
\end{equation*}
so the constraint $Y^{\mathsf{u}_\kboundary\varsigma_\kboundary}Z^{\mathsf{u}_\lboundary\varsigma_\lboundary}
\le \abs{\alpha}_p^{-1}$
implies $Z\gg \abs{\alpha}_p^{1/e_4} L^{e_3/e_4}$,
where $0 < \frac{1}{e_4} \ll \frac{1}{e_1} \ll A_8$
and $0\le \frac{e_3}{e_4} \le \frac{1}{1-\delta}$.
Either way, we now have a lower bound on $Z$,
which together with the upper bound $Z\le 1$
and the estimate $Y \asymp L^{e_1}/Z^{e_2}$
gives \eqref{compact-envelope}.

Having established \eqref{compact-envelope} (in both Case~1 and Case~2),
we now use our assumptions on $\mathfrak{S}'_\RR$
to prove the existence of $\max_{\mathfrak{S}'_\RR}{f_{21}}$,
and the inequality \eqref{R-vertex-bounding}.
The function $\log{f_{21}}$ is linear (hence convex)
in $(\log_p{Y},\log_p{Z})\in \RR^2$.
However, the closed set $\mathfrak{S}'_\RR$ might not be bounded.
To address this, let $M\in \RR_{>0}$
and note that by the $L\ge M$ case of \eqref{compact-envelope},
we have
\begin{equation*}
\{(Y,Z)\in \mathfrak{S}_\RR:
f_{21}\ge M\}
\belongs \{(Y,Z)\in \mathfrak{S}_\RR:
\log_p{Y},\log_p{Z}
\ll_\eps (1+\abs{\alpha}_p^{-1}+M^{-1})^\eps\}.
\end{equation*}
Since $\mathfrak{S}'_\RR\belongs \mathfrak{S}_\RR$,
it follows that for every $M\in \RR_{>0}$, the closed set
$\set{x\in \mathfrak{S}'_\RR: f_{21}(x)\ge M}$
is bounded, and therefore compact.
Choose a point $x_0\in \mathfrak{S}'_\RR$
(possible since $\mathfrak{S}'_\RR\ne \emptyset$),
and let $M = f_{21}(x_0)/2 > 0$.
Then $f_{21}$ achieves a maximum $M_\star\ge f_{21}(x_0)=2M>M$
on the set $\set{x\in \mathfrak{S}'_\RR: f_{21}(x)\ge M}$,
which must then be a global maximum over the set $\mathfrak{S}'_\RR$.
Moreover, recall the definition of $\argmax(\cdot)$ from \eqref{argmax}.
By convexity of $\log{f_{21}}$, the set $\argmax_{\mathfrak{S}'_\RR}{f_{21}}$
contains a vertex $x'$ of $\set{x\in \mathfrak{S}'_\RR: f_{21}\ge M}$,
with $f_{21}(x')=M_\star\ne M$.
(Here we think of
$\set{x\in \mathfrak{S}'_\RR: f_{21}\ge M}$
as the image under $\exp\maps \RR^2\to \RR_{>0}^2$
of a compact polytope.)

We have $\log_p(x')\in \log_p(\mathfrak{S}'_\RR) =  (V\cap \hbar_1\cap \dots\cap \hbar_e)\cap \ell_{\emptyset}$, so there exists a set $\mathcal{E}\belongs \{1,\dots,e\}$ such that $\log_p(x')\in (V\cap \hbar_1\cap \dots\cap \hbar_e)\cap \ell_{\mathcal{E}}$.
Take $\mathcal{E}$ to have maximal cardinality among such sets.
Then $\log_p(x')\in \hbar_{e'}\setminus \ell_{e'}$
for all $e'\in \{1,\dots,e\}\setminus \mathcal{E}$.
Suppose for contradiction that for some $e''\in \{1,\dots,e\}\setminus \mathcal{E}$,
the set $(V\cap \hbar_1\cap \dots\cap \hbar_e)\cap \ell_{\mathcal{E}\cup \{e''\}}$
contains a point, say $\log_p(x'')$.
Then for all $\varepsilon\in \RR$ sufficiently small (not necessarily positive),
the point $$\log_p(x'''_\varepsilon)
\defeq (1-\varepsilon) \log_p(x') + \varepsilon \log_p(x'')$$
lies in $(V\cap \hbar_1\cap \dots\cap \hbar_e)\cap \ell_{\mathcal{E}}$
and satisfies $f_{21}(x'''_\varepsilon) > M$,
so $x'''_\varepsilon\in \set{x\in \mathfrak{S}'_\RR: f_{21}\ge M}$.
Since $x'$ is a vertex,
this can only happen if $x'''_\varepsilon=x'$,
i.e.~$\log_p(x'') = \log_p(x')$.
But $\mathcal{E}$ is maximal, so $x''\ne x'$.
Thus our assumption was wrong,
so in fact $\mathcal{E}$ is extreme,
giving \eqref{R-vertex-bounding}.

Finally, we turn to \eqref{lebesgue-sum}.
Call the integral $I$.
Then
\begin{equation*}
I = (1-p^{-1})^2 \sum_{(Y,Z)\in \mathfrak{S}'_\ZZ} f_{21}(Y,Z)
= (1-p^{-1})^2 \sum_{\log_2{L}\in \ZZ} L
\cdot \#\{(Y,Z)\in \mathfrak{S}'_\ZZ: L\le f_{21}<2L\}
\end{equation*}
by summing over level sets of $f_{21}$, a la Lebesgue.
By \eqref{compact-envelope},
since $\mathfrak{S}'_\ZZ\belongs \mathfrak{S}'_\RR\belongs \mathfrak{S}_\RR$,
we have $$\#\{(Y,Z)\in \mathfrak{S}'_\ZZ: L\le f_{21}<2L\}
\ll_\eps (1+\abs{\alpha}_p^{-1}+L^{-1})^\eps
\ll_\eps 1+\abs{\alpha}_p^{-\eps}+L^{-\eps},$$
where the last step follows from the estimate
$1+\abs{\alpha}_p^{-1}+L^{-1} \asymp \max(1,\abs{\alpha}_p^{-1},L^{-1})$.
So
\begin{equation*}
I \ll_\eps \sum_{\log_2{L}\in \ZZ} (L+L\abs{\alpha}_p^{-\eps}+L^{1-\eps})
\ll_\eps M_\star+M_\star\abs{\alpha}_p^{-\eps}+M_\star^{1-\eps},
\end{equation*}
by summing geometric series over $L\ll M_\star$.
Finally, note that $\max(\abs{\alpha}_p^{-1},\abs{\alpha}_p)\ge 1$.
If $M_\star \ge \max(\abs{\alpha}_p^{-1},\abs{\alpha}_p)^{-1/\eps^{1/2}}$,
then $M_\star^{1-\eps}\le \max(\abs{\alpha}_p^{-1},\abs{\alpha}_p)^{\eps^{1/2}} M_\star$,
so $I\ll_\eps \max(\abs{\alpha}_p^{-1},\abs{\alpha}_p)^{\eps^{1/2}} M_\star$.
If $M_\star \le \max(\abs{\alpha}_p^{-1},\abs{\alpha}_p)^{-1/\eps^{1/2}}$,
then $I\ll_\eps \max(\abs{\alpha}_p^{-1},\abs{\alpha}_p)^{-0.9/\eps^{1/2}}$.
Either way, we have $$I\ll_\eps \max(\abs{\alpha}_p^{-1},\abs{\alpha}_p)^{-0.9/\eps^{1/2}}
+ \max(\abs{\alpha}_p^{-1},\abs{\alpha}_p)^{\eps^{1/2}} M_\star.$$
Replacing $\eps$ with $\frac14\eps^2$, we conclude that \eqref{lebesgue-sum} holds.
\end{proof}

For subsequent calculations,
we choose $c(\jboundary)\in \mathcal{C}^{(0)}(\jboundary)$
and let $\mathsf{v}_\jboundary \defeq \ord_{D_\jboundary}(b-c(\jboundary))$
for $\jboundary\in J$.
Also, we make the following definition for convenience in \S\ref{SEC:archimedean-endgame} (when we apply Lemma~\ref{LEM:denominator-bias}).

\begin{definition}
\label{DEFN:bad-set-S}
Let $\ol{\Sbad}$ be a finite set containing
the set $\Sbad$ defined after Lemma~\ref{LEM:snc-implies-nice-densities-implies-Tamagawa-extraction},
such that
$\set{q\in \QQ: I^q\ne \emptyset} \belongs \ZZ[1/\ol{\Sbad}]$
and $\set{q_1-q_2: q_1,q_2\in \QQ,\; I^{q_1},I^{q_2}\ne \emptyset}
\belongs \ZZ[1/\ol{\Sbad}]^\times \cup \{0\}$,
where $I^q$ is the set defined in \eqref{Ic}.
\end{definition}

\begin{lemma}
[Denominator bias]
\label{LEM:denominator-bias}
Suppose $v_p(\alpha) = -k < 0$.
Let
\begin{equation*}
\mathfrak{D}=\mathfrak{D}_p(\bm{s},\lambda,\alpha)\defeq H^\vee_p(\bm{s},\lambda,0,\alpha)
- \bm{1}_{p\notin \ol{\Sbad}} \sum_{c\in \QQ}
\sum_{\jboundary\in J^c_1} \phase(-c\alpha \bmod{\ZZ_p}) p^{-k (s_\jboundary-\mathsf{d}_\jboundary+1)/\mathsf{u}_\jboundary} \bm{1}_{\mathsf{u}_\jboundary\mid k}.
\end{equation*}
Then $\mathfrak{D} \ll_\eps p^{(\eps-2)k} \sum_{\jboundary\in J_1} (p^{\delta-1}
+ p^{-1/2} \bm{1}_{\mathsf{u}_\jboundary\mid k} + \bm{1}_{k>\mathsf{u}_\jboundary}) p^{(\delta-1)k/\mathsf{u}_\jboundary}$.
\end{lemma}

\begin{proof}
Here $\abs{\alpha}_p = p^k$, so by Lemma~\ref{LEM:general-fixed-p-bound}
(and the bound $p^{-k\Re(s_\jboundary-\mathsf{d}_\jboundary+1)/\mathsf{u}_\jboundary}
\le p^{-2k} p^{(\delta-1)k/\mathsf{u}_\jboundary}$ for $\jboundary\in J_1$,
using $\mathsf{u}_\jboundary>0$),
we may assume $p$ is large.
Following \cite{tanimoto2012height}*{proof of Lemma~5.6}, let $\rho$ be the reduction map $X(\QQ_p) = \mathscr{X}(\ZZ_p) \to \mathscr{X}(\FF_p)$;
then $H^\vee_p = \sum_{x_0\in \bigcup_{\jboundary\in J_1} \mathscr{D}_\jboundary(\FF_p)} \mathcal{I}_p(x_0)$, where
\begin{equation*}
\mathcal{I}_p(x_0)
\defeq \int_{\rho^{-1}(x_0):\, \alpha a\in \ZZ_p}
H_p(\bm{s},g)^{-1} \phase(-\alpha b\bmod{\ZZ_p}) \lambda_{\Sbad,p}(\alpha a)\, dg.
\end{equation*}
Let $\mathcal{I}^0_p(x_0)$ be the contribution to $\mathcal{I}_p(x_0)$ from $\alpha a\in \ZZ_p^\times$, and let $\mathcal{I}^1_p(x_0)$ be the contribution from $\alpha a\in p\ZZ_p$; then $\mathcal{I}_p(x_0) = \mathcal{I}^0_p(x_0) + \mathcal{I}^1_p(x_0)$, and we have
\begin{equation*}
\mathcal{I}^0_p(x_0)
= \int_{\rho^{-1}(x_0):\, \alpha a\in \ZZ_p^\times}
\frac{\phase(-\alpha b\bmod{\ZZ_p}) \, dg}{H_p(\bm{s},g)},
\quad
\abs{\mathcal{I}^1_p(x_0)}
\le \int_{\rho^{-1}(x_0):\, \alpha a\in p\ZZ_p}
\frac{dg}{\abs{H_p(\bm{s},g)}}.
\end{equation*}

For each $\jboundary\in J$,
let $U_\jboundary\belongs X$ be the largest open set on which $a, b-c(\jboundary)\in \Gamma(U_\jboundary\setminus D_\jboundary, \mathcal{O}_X)^\times$.
Let $\mathscr{Z}_\jboundary$ be the closure of $X\setminus U_\jboundary$ in $\mathscr{X}$.
Then let $\mathscr{D}^\ast_\jboundary\defeq \mathscr{D}_\jboundary\setminus \mathscr{Z}_\jboundary$.
We claim the following:
\begin{enumerate}
\item If $\jboundary\in J_1$, then
\begin{equation*}
\sum_{x_0\in \mathscr{D}^\ast_\jboundary(\FF_p)} \mathcal{I}_p(x_0)
= (\phase(-c(\jboundary)\alpha \bmod{\ZZ_p})
\bm{1}_{\mathsf{u}_\jboundary\mid k} \bm{1}_{\jboundary\in J^{c(\jboundary)}_1}
+ O(p^{-1/2} \bm{1}_{\mathsf{u}_\jboundary\mid k}) + O(p^{-1}))
\cdot p^{-(s_\jboundary-\mathsf{d}_\jboundary+1)k/\mathsf{u}_\jboundary}.
\end{equation*}

\item If $x_0\in \bigcup_{\jboundary\in J_1} \mathscr{D}_\jboundary(\FF_p)$, then
\begin{equation*}
\mathcal{I}_p(x_0)
\ll_\eps p^{(\eps-2)k} \sum_{\jboundary\in J_1:\, x_0\in \mathscr{D}_\jboundary(\FF_p)} (p^{\delta-1} + \bm{1}_{k>\mathsf{u}_\jboundary}) p^{(\delta-1)k/\mathsf{u}_\jboundary}.
\end{equation*}
\end{enumerate}
These two claims would imply the lemma,
since on the one hand
\begin{equation*}
\biggl\lvert{H^\vee - \sum_{\jboundary\in J_1}
\sum_{x_0\in \mathscr{D}^\ast_\jboundary(\FF_p)} \mathcal{I}_p(x_0)}\biggr\rvert
\le \sum_{\jboundary\in J_1} \sum_{x_0\in (\mathscr{D}_\jboundary\setminus \mathscr{D}^\ast_\jboundary)(\FF_p)} \abs{\mathcal{I}_p(x_0)}
\ll \max_{x_0\in \bigcup_{\jboundary\in J_1} \mathscr{D}_\jboundary(\FF_p)} \abs{\mathcal{I}_p(x_0)},
\end{equation*}
which may be bounded using (2);
and on the other hand for each $\jboundary\in J_1$, by (1),
\begin{equation*}
\begin{split}
\sum_{x_0\in \mathscr{D}^\ast_\jboundary(\FF_p)} \mathcal{I}_p(x_0)
&= \phase(-c(\jboundary)\alpha \bmod{\ZZ_p})
\bm{1}_{\mathsf{u}_\jboundary\mid k} \bm{1}_{\jboundary\in J^{c(\jboundary)}_1}
p^{-(s_\jboundary-\mathsf{d}_\jboundary+1)k/\mathsf{u}_\jboundary}
+ O(p^{-1/2} \bm{1}_{\mathsf{u}_\jboundary\mid k} + p^{-1})
p^{-2k} p^{(\delta-1)k/\mathsf{u}_\jboundary} \\
&= \sum_{c\in \QQ} \bm{1}_{\jboundary\in J^c_1}
\phase(-c\alpha \bmod{\ZZ_p}) p^{-k (s_\jboundary-\mathsf{d}_\jboundary+1)/\mathsf{u}_\jboundary} \bm{1}_{\mathsf{u}_\jboundary\mid k}
+ O(p^{-1/2} \bm{1}_{\mathsf{u}_\jboundary\mid k} + p^{-1})
p^{-2k} p^{(\delta-1)k/\mathsf{u}_\jboundary},
\end{split}
\end{equation*}
where the last equality follows from Proposition~\ref{PROP:compute-arg-max-set-C0j}.
Indeed, if $\jboundary\in J^c_1$ for some $c\in \QQ$,
then $\mathcal{C}^{(0)}(\jboundary) = \{c\}$
(by Proposition~\ref{PROP:compute-arg-max-set-C0j}
and the containment $J^c_1\belongs I^c$ from \eqref{Jc1-Ic-J1-containments}),
so $c(\jboundary) = c$.

We prove (1) first.
Suppose $\jboundary\in J_1$ and $c(\jboundary)=c$, and let $x_0\in \mathscr{D}^\ast_\jboundary(\FF_p)$.
By Lemma~\ref{LEM:large-p-local-analytic-coordinates-generically},
the set $\rho^{-1}(x_0) \belongs X(\QQ_p)$
has analytic coordinates $y,z\in p\ZZ_p$
such that $a = \unitone y^{\mathsf{u}_\jboundary}$
and $b-c = y^{\mathsf{v}_\jboundary} (\unittwo+z)$,
where $\unitone,\unittwo\in \ZZ_p^\times$.
Upon writing
$dg=\abs{\omega}_p/(1-p^{-1})$ in terms of $d\tau_p$
using \eqref{EQN:define-local-Tamagawa-measure},
as in \cite{tanimoto2012height}*{proof of Lemma~5.6}, we get
\begin{align*}
\mathcal{I}^0_p(x_0) &= \int_{y,z\in p\ZZ_p:\, v_p(y) = k/\mathsf{u}_\jboundary}
\abs{y}_p^{s_\jboundary-\mathsf{d}_\jboundary} \phase(-c\alpha-y^{\mathsf{v}_\jboundary}(\unittwo+z)\alpha \bmod{\ZZ_p})\, \frac{dy\,dz}{1-p^{-1}} \\
&= \phase(-c\alpha\bmod{\ZZ_p})
\int_{y\in p\ZZ_p:\, v_p(y) = k/\mathsf{u}_\jboundary} \abs{y}_p^{s_\jboundary-\mathsf{d}_\jboundary+1}
\phase(-y^{\mathsf{v}_\jboundary}\unittwo\alpha \bmod{\ZZ_p}) \frac{\bm{1}_{y^{\mathsf{v}_\jboundary}p\alpha\in \ZZ_p}}{p} \, \frac{dy/\abs{y}_p}{1-p^{-1}}.
\end{align*}
Thus $p\mathcal{I}^0_p(x_0) = p^{-(s_\jboundary-\mathsf{d}_\jboundary+1)k/\mathsf{u}_\jboundary} \phase(-c\alpha\bmod{\ZZ_p}) \bm{1}_{\mathsf{u}_\jboundary\mid k}$ if $\mathsf{v}_\jboundary\ge \mathsf{u}_\jboundary$;
and if $\mathsf{v}_\jboundary<\mathsf{u}_\jboundary$, then summing over $x_0\in \mathscr{D}^\ast_\jboundary(\FF_p)$, $y\in p^{k/\mathsf{u}_\jboundary}\ZZ_p^\times$ (integrating over $y$ first if $\mathsf{v}_\jboundary\ne 0$, and $x_0$ first if $\mathsf{v}_\jboundary=0$) gives
\begin{equation*}
\sum_{x_0\in \mathscr{D}^\ast_\jboundary(\FF_p)} p\mathcal{I}^0_p(x_0)
\ll p^{-\Re(s_\jboundary-\mathsf{d}_\jboundary+1)k/\mathsf{u}_\jboundary} \bm{1}_{\mathsf{u}_\jboundary\mid k}
\cdot p^{1/2} \bm{1}_{\mathsf{v}_\jboundary k/\mathsf{u}_\jboundary = k-1}
\end{equation*}
(by the bound \cite{bombieri1966exponential}*{Theorem~6} on exponential sums over curves;
this applies because $\sgn_p(y)^{\mathsf{v}_\jboundary}\bmod{p}$ is nonconstant if $\mathsf{v}_\jboundary\ne 0$, and if $\mathsf{v}_\jboundary=0$ then $\unittwo\equiv b-c\bmod{p}$ is nonconstant by Proposition~\ref{PROP:G-induced-nonconstancy-of-a-rational-function-on-Dj}).
Furthermore, if $\mathsf{v}_\jboundary\ge \mathsf{u}_\jboundary$, then (by Proposition~\ref{PROP:apply-lower-bound-on-anticanonical-Weil-divisor}) $\mathsf{v}_\jboundary=\mathsf{u}_\jboundary$ and $\jboundary\in J^c_1$.
On the other hand, for $\mathcal{I}^1_p(x_0)$, summing a geometric series over $v_p(y) \in \ZZ$ yields
\begin{equation*}
\mathcal{I}^1_p(x_0) \ll \int_{y,z\in p\ZZ_p:\, v_p(y) \ge (k+1)/\mathsf{u}_\jboundary}
\abs{y}_p^{\Re(s_\jboundary-\mathsf{d}_\jboundary)}\, dy\,dz
\le \frac{p^{-\Re(s_\jboundary-\mathsf{d}_\jboundary+1)(k+1)/\mathsf{u}_\jboundary}}{1-p^{-\Re(s_\jboundary-\mathsf{d}_\jboundary+1)}} \cdot \frac{1}{p}
\ll \frac{p^{-\Re(s_\jboundary-\mathsf{d}_\jboundary+1)k/\mathsf{u}_\jboundary}}{p^3},
\end{equation*}
since $\Re(s_\jboundary-\mathsf{d}_\jboundary+1)\ge 2\mathsf{u}_\jboundary$.
Finally, summing $\mathcal{I}^0_p(x_0)$, $\mathcal{I}^1_p(x_0)$
over $x_0\in \mathscr{D}^\ast_\jboundary(\FF_p)$ using the Lang--Weil estimate
delivers (1), since $D_\jboundary$ is geometrically irreducible.

We now prove claim~(2).
If $x_0\in \mathscr{D}_\jboundary(\FF_p) \setminus \bigcup_{i\in J\setminus \set{\jboundary}} \mathscr{D}_i(\FF_p)$ for some $\jboundary\in J_1$, then
\begin{equation*}
\mathcal{I}_p(x_0) \ll \int_{y,z\in p\ZZ_p:\, v_p(y) \ge k/\mathsf{u}_\jboundary}
\abs{y}_p^{\Re(s_\jboundary-\mathsf{d}_\jboundary)} \, dy\,dz
\ll p^{-\Re(s_\jboundary-\mathsf{d}_\jboundary+1) k/\mathsf{u}_\jboundary} p^{-1}
\le p^{-(2\mathsf{u}_\jboundary-\delta+1) k/\mathsf{u}_\jboundary} p^{-1},
\end{equation*}
which suffices.
Similarly, if $x_0\in (\mathscr{D}_\jboundary\cap \mathscr{D}_i)(\FF_p)$ for some $\jboundary\in J_1$ and $i\in J_3$, then
\begin{equation*}
\mathcal{I}_p(x_0) \ll \int_{y,z\in p\ZZ_p:\, v_p(y) \ge k/\mathsf{u}_\jboundary}
\abs{y}_p^{\Re(s_\jboundary-\mathsf{d}_\jboundary)} \abs{z}_p^{\Re(s_i-\mathsf{d}_i)} \, dy\,dz
\ll p^{-\Re(s_\jboundary-\mathsf{d}_\jboundary+1) k/\mathsf{u}_\jboundary} p^{-\Re(s_i-\mathsf{d}_i+1)},
\end{equation*}
which suffices since $\Re(s_\jboundary-\mathsf{d}_\jboundary+1)\ge 2\mathsf{u}_\jboundary-\delta+1$ and $\Re(s_i-\mathsf{d}_i+1)\ge 1-\delta$.

The remaining cases are messy to do explicitly,
so we now apply Lemma~\ref{LP-axiom}.
Initially, we take
$(\kboundary,\lboundary,\varsigma_\kboundary,\varsigma_\lboundary,d,e)
= (\jboundary,i,1,1,2,3)$,
with half-planes $\hbar_1,\hbar_2,\hbar_3\belongs \RR^2$ corresponding to
the half-regions $\{Y\le p^{-1}\}$, $\{Z\le p^{-1}\}$, and $\{\abs{\alpha}_p\, Y^{\mathsf{u}_\kboundary\varsigma_\kboundary}Z^{\mathsf{u}_\lboundary\varsigma_\lboundary} \le 1\}$ in $\RR_{>0}^2$,
so that $\mathfrak{S}'_R \belongs \mathfrak{S}_R(\alpha)$.
If $x_0\in (\mathscr{D}_\jboundary\cap \mathscr{D}_i)(\FF_p)$ for some distinct $\jboundary,i\in J_1$, then \eqref{lebesgue-sum} gives\footnote{The pesky term $(\abs{\alpha}_p^{-1}+\abs{\alpha}_p)^{-1/\eps} \le p^{-k/\eps}$ in \eqref{lebesgue-sum} is negligible, so we do not mention it further.}
\begin{equation*}
\mathcal{I}_p(x_0) \ll \int_{\substack{y,z\in p\ZZ_p: \\ y^{\mathsf{u}_\jboundary}z^{\mathsf{u}_i}\in p^k\ZZ_p}}
\abs{y}_p^{\Re(s_\jboundary-\mathsf{d}_\jboundary)} \abs{z}_p^{\Re(s_i-\mathsf{d}_i)} \, dy\,dz
\ll_\eps p^{k\eps} \max_{\substack{Y,Z\le p^{-1}: \\ Y^{\mathsf{u}_\jboundary}Z^{\mathsf{u}_i}\le p^{-k}}}
{Y^{\Re(s_\jboundary-\mathsf{d}_\jboundary+1)} Z^{\Re(s_i-\mathsf{d}_i+1)}},
\end{equation*}
which by \eqref{R-vertex-bounding} is $\le p^{k\eps} p^{-\Re(s_\jboundary-\mathsf{d}_\jboundary+1) \max(1, (k-\mathsf{u}_i)/\mathsf{u}_\jboundary)} p^{-\Re(s_i-\mathsf{d}_i+1)}$ (assuming $\argmax({\cdots}) \cap \set{Z=p^{-1}} \ne \emptyset$; otherwise, switch $\jboundary$, $i$),
which is in turn (using $\Re(s_\jboundary-\mathsf{d}_\jboundary+1)\ge (1-\delta) + 2\mathsf{u}_\jboundary$)
\begin{equation}
\label{INEQ:tricky-numerical-point}
\le p^{k\eps} p^{-(1-\delta)\max(1, (k-\mathsf{u}_i)/\mathsf{u}_\jboundary) - (2\mathsf{u}_\jboundary)(k-\mathsf{u}_i)/\mathsf{u}_\jboundary} p^{-(2\mathsf{u}_i-\delta+1)}
= p^{k\eps} p^{-(1-\delta)\max(2, (k+\mathsf{u}_\jboundary-\mathsf{u}_i)/\mathsf{u}_\jboundary)} p^{-2k};
\end{equation}
this suffices, since $\max(2, \frac{k+\mathsf{u}_\jboundary-\mathsf{u}_i}{\mathsf{u}_\jboundary})\ge \bm{1}_{k\le \max(\mathsf{u}_\jboundary,\mathsf{u}_i)} + \frac{k}{\max(\mathsf{u}_\jboundary,\mathsf{u}_i)}$ by Lemma~\ref{LEM:tricky-numerical-point} below.
If $x_0\in (\mathscr{D}_\jboundary\cap \mathscr{D}_i)(\FF_p)$ for some $\jboundary\in J_1$ and $i\in J_2$, then \eqref{lebesgue-sum} gives
\begin{equation*}
\mathcal{I}_p(x_0) \ll \int_{\substack{y,z\in p\ZZ_p: \\ y^{\mathsf{u}_\jboundary}z^{\mathsf{u}_i}\in p^k\ZZ_p}}
\abs{y}_p^{\Re(s_\jboundary-\mathsf{d}_\jboundary)} \abs{z}_p^{\Re(s_i-\mathsf{d}_i)} \, dy\,dz
\ll_\eps p^{k\eps} \max_{\substack{Y,Z\le p^{-1}: \\ Y^{\mathsf{u}_\jboundary}\le p^{-k}Z^{\abs{\mathsf{u}_i}}}}
{Y^{\Re(s_\jboundary-\mathsf{d}_\jboundary+1)} Z^{\Re(s_i-\mathsf{d}_i+1)}};
\end{equation*}
but the conditions $Z\le p^{-1}$ and $Y^{\mathsf{u}_\jboundary}\le p^{-k}Z^{\abs{\mathsf{u}_i}}$ together imply $Y<1$, say,
so if we extend the half-region $\{Y\le p^{-1}\}$ to $\{Y\le 1\}$
before applying \eqref{R-vertex-bounding}, then we find that
\begin{equation*}
\max_{\substack{Y,Z\le p^{-1}: \\ Y^{\mathsf{u}_\jboundary}\le p^{-k}Z^{\abs{\mathsf{u}_i}}}}
({\cdots})
\le \max_{\substack{Y\le 1,\; Z\le p^{-1}: \\ Y^{\mathsf{u}_\jboundary}\le p^{-k}Z^{\abs{\mathsf{u}_i}}}}
({\cdots})
= ({\cdots})\vert_{\substack{Z=p^{-1}, \\ Y^{\mathsf{u}_\jboundary}=p^{-k}Z^{\abs{\mathsf{u}_i}}}}
= p^{-\Re(s_\jboundary-\mathsf{d}_\jboundary+1) (k+\abs{\mathsf{u}_i})/\mathsf{u}_\jboundary} p^{-\Re(s_i-\mathsf{d}_i+1)},
\end{equation*}
which is $\le p^{-(2\mathsf{u}_\jboundary-\delta+1) (k+\abs{\mathsf{u}_i})/\mathsf{u}_\jboundary} p^{-(2\mathsf{u}_i-\delta+1)}
= p^{-2k} p^{(\delta-1) (k+\abs{\mathsf{u}_i})/\mathsf{u}_\jboundary} p^{\delta-1}
\le p^{-2k} p^{(\delta-1) k/\mathsf{u}_\jboundary} p^{\delta-1}$.
\end{proof}

\begin{lemma}
\label{LEM:tricky-numerical-point}
Let $k,l,m\in \RR_{>0}$.
Then $\max(1, \frac{k+m-l}{m}) \ge \min(\frac{k}{l}, \frac{k}{m})$.
\end{lemma}

\begin{proof}
If $l\le m$, then
$\frac{k+m-l}{m} \ge \frac{k}{m}$.
If $k\le l$, then
$1 \ge \frac{k}{l}$.
If $k>l>m$,
then $\frac{k+m-l}{m} > \frac{k}{l}$.
\end{proof}

\begin{lemma}
[Numerator bias]
\label{LEM:numerator-bias}
Suppose $v_p(\alpha) = k > 0$.
Let
\begin{equation*}
\mathfrak{N}=\mathfrak{N}_p(\bm{s},\lambda,\alpha)\defeq H^\vee_p(\bm{s},\lambda,0,\alpha)
- \sum_{\jboundary\in J^\ast_2} p^{-k (s_\jboundary-\mathsf{d}_\jboundary+1)/\abs{\mathsf{u}_\jboundary}} \bm{1}_{\mathsf{u}_\jboundary\mid k}.
\end{equation*}
Then $\mathfrak{N} \ll_\eps 
p^{(2+\eps)k} \sum_{\jboundary\in J_2} (p^{\delta-1} + p^{-1/2}\bm{1}_{\mathsf{u}_\jboundary\mid k} + \bm{1}_{k>\abs{\mathsf{u}_\jboundary}}) p^{(\delta-1)k/\abs{\mathsf{u}_\jboundary}}$.
\end{lemma}

\begin{proof}
Here $\abs{\alpha}_p = p^{-k}$, so by Lemma~\ref{LEM:general-fixed-p-bound} and the bound $1 + p^{-k\Re(s_\jboundary-\mathsf{d}_\jboundary+1)/\abs{\mathsf{u}_\jboundary}} \le 2p^{2k} p^{(\delta-1)k/\abs{\mathsf{u}_\jboundary}}$ (and the fact $J_2\ne \emptyset$), we may assume $p$ is large.
Define $\mathcal{I}_p(x_0)$ and $\mathscr{D}^\ast_\jboundary$ as before.
Then $H^\vee_p = \sum_{x_0\in \mathscr{X}(\FF_p)} \mathcal{I}_p(x_0)$.
As noted in \cite{tanimoto2012height}*{proof of Lemma~5.8}, we have
\begin{equation*}
\sum_{x_0\notin \bigcup_{\jboundary\in J_2} \mathscr{D}_\jboundary(\FF_p)} \abs{\mathcal{I}_p(x_0)}
\le \sum_{x_0\notin \bigcup_{\jboundary\in J_2} \mathscr{D}_\jboundary(\FF_p)}
\int_{\rho^{-1}(x_0)} \abs{H_p(\bm{s},g)}^{-1}\, dg
\ll 1;
\end{equation*}
this is because $\Re(s_\jboundary-\mathsf{d}_\jboundary+1) \ge 1-\delta > 0$ for $\jboundary\in J_1\cup J_3$.
This $O(1)$ bound is satisfactory, since $J_2\ne \emptyset$ and $1\le p^{2k} p^{\delta-1} p^{(\delta-1)k/\abs{\mathsf{u}_\jboundary}}$.
Thus it remains to study $\mathcal{I}_p(x_0)$ for $x_0\in \bigcup_{\jboundary\in J_2} \mathscr{D}_\jboundary(\FF_p)$.
We claim the following, of which the lemma is a direct consequence:
\begin{enumerate}
\item If $\jboundary\in J_2$ and $x_0\in \mathscr{D}^\ast_\jboundary(\FF_p)$, then
\begin{equation*}
p\mathcal{I}_p(x_0)
= p^{-(s_\jboundary-\mathsf{d}_\jboundary+1)k/\abs{\mathsf{u}_\jboundary}} \bm{1}_{\mathsf{u}_\jboundary\mid k} \bm{1}_{\jboundary\in J^\ast_2}
+ p^{(2\abs{\mathsf{u}_\jboundary}+\delta-1)k/\abs{\mathsf{u}_\jboundary}} O(p^{-1/2}\bm{1}_{\mathsf{u}_\jboundary\mid k} + p^{-1}).
\end{equation*}

\item If $x_0\in \bigcup_{\jboundary\in J_2} \mathscr{D}_\jboundary(\FF_p)$, then
\begin{equation*}
\mathcal{I}_p(x_0)
\ll_\eps p^{(2+\eps)k} \sum_{\jboundary\in J_2:\, x_0\in \mathscr{D}_\jboundary(\FF_p)} (p^{\delta-1} + \bm{1}_{k>\abs{\mathsf{u}_\jboundary}}) p^{(\delta-1)k/\abs{\mathsf{u}_\jboundary}}.
\end{equation*}
\end{enumerate}

For (1),
suppose $\jboundary\in J_2$ and $x_0\in \mathscr{D}^\ast_\jboundary(\FF_p)$.
Then by Lemma~\ref{LEM:large-p-local-analytic-coordinates-generically},
the set $\rho^{-1}(x_0)$ has analytic coordinates $y,z\in p\ZZ_p$
with $a = \unitone y^{\mathsf{u}_\jboundary}$
and $b-c(\jboundary) = y^{\mathsf{v}_\jboundary} (\unittwo+z)$,
where $\unitone,\unittwo\in \ZZ_p^\times$.
Since $c(\jboundary)\alpha\in \ZZ_p$, it follows that
\begin{equation*}
\mathcal{I}_p(x_0) = \int_{y,z\in p\ZZ_p:\, v_p(y) \le k/\abs{\mathsf{u}_\jboundary}}
\abs{y}_p^{s_\jboundary-\mathsf{d}_\jboundary} \phase(-y^{\mathsf{v}_\jboundary}(\unittwo+z)\alpha \bmod{\ZZ_p})
\lambda_{\Sbad,p}(\alpha y^{\mathsf{u}_\jboundary}) \, \frac{dy\,dz}{1-p^{-1}}.
\end{equation*}
For any given $y$, the integral over $z$ vanishes unless $y^{\mathsf{v}_\jboundary}p\alpha\in \ZZ_p$, i.e.~$v_p(y)\le (k+1)/\abs{\mathsf{v}_\jboundary}$;
here $\mathsf{v}_\jboundary\le \mathsf{u}_\jboundary<0$ by Proposition~\ref{PROP:apply-lower-bound-on-anticanonical-Weil-divisor}, since $\jboundary\in J_2$.
Integrating over the ranges $v_p(y) = (k+1)/\abs{\mathsf{v}_\jboundary}$
(i.e.~$y^{\mathsf{v}_\jboundary}\alpha\in p^{-1}\ZZ_p^\times$)
and $v_p(y) \le k/\abs{\mathsf{v}_\jboundary}$ (i.e.~$y^{\mathsf{v}_\jboundary}\alpha\in \ZZ_p$) separately, we then get
\begin{align*}
p \mathcal{I}_p(x_0)
&= p \int_{y,z\in p\ZZ_p:\, (k+1)/\abs{\mathsf{v}_\jboundary} = v_p(y) \le k/\abs{\mathsf{u}_\jboundary}} ({\cdots})
+ p \int_{y,z\in p\ZZ_p:\, v_p(y) \le k/\abs{\mathsf{v}_\jboundary}} ({\cdots}) \\
&= p^{-\Re(s_\jboundary-\mathsf{d}_\jboundary+1)(k+1)/\abs{\mathsf{v}_\jboundary}} O(p^{-1/2})
\bm{1}_{\ZZ\ni (k+1)/\abs{\mathsf{v}_\jboundary} \le k/\abs{\mathsf{u}_\jboundary}}
+ \sum_{1\le r\le k/\abs{\mathsf{v}_\jboundary}} p^{-(s_\jboundary-\mathsf{d}_\jboundary+1)r} \lambda_{\Sbad,p}(p^{k-r\mathsf{u}_\jboundary});
\end{align*}
the $O(p^{-1/2})$ factor comes from cancellation over $y$,
occurring since $\mathsf{v}_\jboundary\ne 0$ and $\unittwo\in \ZZ_p^\times$.
Since $-\Re(s_\jboundary-\mathsf{d}_\jboundary+1) \le 2\abs{\mathsf{u}_\jboundary}+\delta-1$, and any integer $<k/\abs{\mathsf{u}_\jboundary}$ is $\le (k-1)/\abs{\mathsf{u}_\jboundary}$, we get
\begin{equation*}
p \mathcal{I}_p(x_0)
= p^{-(s_\jboundary-\mathsf{d}_\jboundary+1)k/\abs{\mathsf{u}_\jboundary}} \bm{1}_{\mathsf{u}_\jboundary\mid k} \bm{1}_{\mathsf{v}_\jboundary=\mathsf{u}_\jboundary}
+ p^{(2\abs{\mathsf{u}_\jboundary}+\delta-1)k/\abs{\mathsf{u}_\jboundary}} O(p^{-1/2}\bm{1}_{\mathsf{u}_\jboundary\mid k} + p^{-(2\abs{\mathsf{u}_\jboundary}+\delta-1)/\abs{\mathsf{u}_\jboundary}}).
\end{equation*}
This implies (1), since $2\abs{\mathsf{u}_\jboundary}+\delta-1 \ge \abs{\mathsf{u}_\jboundary}$.

We now turn to (2).
If $\jboundary\in J_2$ and $x_0\in \mathscr{D}_\jboundary(\FF_p) \setminus \bigcup_{i\in J\setminus \set{\jboundary}} \mathscr{D}_i(\FF_p)$, then
\begin{equation*}
\mathcal{I}_p(x_0) \ll \int_{p\ZZ_p^2:\, v_p(y) \le k/\abs{\mathsf{u}_\jboundary}}
\abs{y}_p^{\Re(s_\jboundary-\mathsf{d}_\jboundary)} \, dy\,dz
\le \int_{p\ZZ_p^2:\, v_p(y) \le k/\abs{\mathsf{u}_\jboundary}}
\abs{y}_p^{2\mathsf{u}_\jboundary-\delta} \, dy\,dz
\ll \frac{p^{-(2\mathsf{u}_\jboundary-\delta+1)k/\abs{\mathsf{u}_\jboundary}}}{p},
\end{equation*}
which suffices.
Similarly, if $x_0\in (\mathscr{D}_\jboundary\cap \mathscr{D}_i)(\FF_p)$ for some $\jboundary\in J_2$ and $i\in J_3$, then
\begin{equation*}
\mathcal{I}_p(x_0)
\ll \int_{p\ZZ_p^2:\, v_p(y) \le k/\abs{\mathsf{u}_\jboundary}}
\abs{y}_p^{2\mathsf{u}_\jboundary-\delta} \abs{z}_p^{2\mathsf{u}_i-\delta} \, dy\,dz
\ll p^{-(2\mathsf{u}_\jboundary-\delta+1)k/\abs{\mathsf{u}_\jboundary}} p^{-(1-\delta)}
\quad(\textnormal{since $\mathsf{u}_i=0$}).
\end{equation*}
We treat the remaining cases by Lemma~\ref{LP-axiom},
with $(d,e)=(2,3)$ initially.
If $x_0\in (\mathscr{D}_\jboundary\cap \mathscr{D}_i)(\FF_p)$ for some distinct $\jboundary,i\in J_2$,
then $\mathcal{I}_p(x_0)=0$ unless $k\ge \abs{\mathsf{u}_i}+\abs{\mathsf{u}_\jboundary}$, in which case
\begin{equation*}
\mathcal{I}_p(x_0) \ll \int_{\substack{y,z\in p\ZZ_p: \\ p^k y^{\mathsf{u}_\jboundary}z^{\mathsf{u}_i}\in \ZZ_p}}
\abs{y}_p^{2\mathsf{u}_\jboundary-\delta} \abs{z}_p^{2\mathsf{u}_i-\delta} \, dy\,dz
\ll_\eps p^{k\eps} \max_{\substack{Y,Z\le p^{-1}: \\ Y^{\abs{\mathsf{u}_\jboundary}}Z^{\abs{\mathsf{u}_i}}\ge p^{-k}}} {Y^{2\mathsf{u}_\jboundary-\delta+1} Z^{2\mathsf{u}_i-\delta+1}}
\end{equation*}
by \eqref{lebesgue-sum},
where $\argmax({\cdots})
\belongs \set{Y^{\abs{\mathsf{u}_\jboundary}}Z^{\abs{\mathsf{u}_i}} = p^{-k}}$
by \eqref{R-vertex-bounding}
(since if $Y^{\abs{\mathsf{u}_\jboundary}}Z^{\abs{\mathsf{u}_i}} > p^{-k}$,
then $Y^{2\mathsf{u}_\jboundary-\delta+1} Z^{2\mathsf{u}_i-\delta+1}$
can be increased by decreasing $Y$ or $Z$)
and thus, after switching $\jboundary$, $i$ if necessary so that $\argmax({\cdots}) \cap \set{Z=p^{-1}} \ne \emptyset$, we have
\begin{equation*}
\max({\cdots})
= p^{(2\mathsf{u}_\jboundary-\delta+1)(\abs{\mathsf{u}_i}-k)/\abs{\mathsf{u}_\jboundary}} p^{-(2\mathsf{u}_i-\delta+1)}
= p^{2k} p^{(\delta-1)(k+\abs{\mathsf{u}_\jboundary}-\abs{\mathsf{u}_i})/\abs{\mathsf{u}_\jboundary}}
\end{equation*}
(cf.~\eqref{INEQ:tricky-numerical-point});
here $\frac{k+\abs{\mathsf{u}_\jboundary}-\abs{\mathsf{u}_i}}{\abs{\mathsf{u}_\jboundary}} = \max(2, \frac{k+\abs{\mathsf{u}_\jboundary}-\abs{\mathsf{u}_i}}{\abs{\mathsf{u}_\jboundary}}) \ge \bm{1}_{k\le \max(\abs{\mathsf{u}_\jboundary},\abs{\mathsf{u}_i})} + \frac{k}{\max(\abs{\mathsf{u}_\jboundary},\abs{\mathsf{u}_i})}$ by Lemma~\ref{LEM:tricky-numerical-point}.
If $x_0\in (\mathscr{D}_\jboundary\cap \mathscr{D}_i)(\FF_p)$ for some $\jboundary\in J_2$ and $i\in J_1$, then
\begin{equation*}
\mathcal{I}_p(x_0) \ll \int_{\substack{y,z\in p\ZZ_p: \\ p^k y^{\mathsf{u}_\jboundary}z^{\mathsf{u}_i}\in \ZZ_p}}
\abs{y}_p^{2\mathsf{u}_\jboundary-\delta} \abs{z}_p^{2\mathsf{u}_i-\delta} \, dy\,dz
\ll_\eps p^{k\eps} \max_{\substack{Y,Z\le p^{-1}: \\ Y^{\abs{\mathsf{u}_\jboundary}}\ge p^{-k}Z^{\mathsf{u}_i}}}
{Y^{2\mathsf{u}_\jboundary-\delta+1} Z^{2\mathsf{u}_i-\delta+1}}
\end{equation*}
by \eqref{lebesgue-sum},
where (after decreasing $Y$ if necessary,
which will only increase $Y^{2\mathsf{u}_\jboundary-\delta+1}$)
\begin{equation*}
\max_{\substack{Y,Z\le p^{-1}: \\ Y^{\abs{\mathsf{u}_\jboundary}}\ge p^{-k}Z^{\mathsf{u}_i}}}
({\cdots})
= \max_{\substack{Y,Z\le p^{-1}: \\ Y^{\abs{\mathsf{u}_\jboundary}} = p^{-k}Z^{\mathsf{u}_i}}}
({\cdots})
\le \max_{\substack{Y\le 1,\; Z\le p^{-1}: \\ Y^{\abs{\mathsf{u}_\jboundary}} = p^{-k}Z^{\mathsf{u}_i}}}
({\cdots})
= ({\cdots})\vert_{\substack{Z = p^{-1}: \\ Y^{\abs{\mathsf{u}_\jboundary}} = p^{-k}Z^{\mathsf{u}_i}}}
\end{equation*}
(by the $(d,e) = (1,2)$ case of \eqref{R-vertex-bounding},\footnote{We take the affine line $V\belongs \RR^2$ corresponding to the curve $\{Y^{\abs{\mathsf{u}_\jboundary}} = p^{-k}Z^{\mathsf{u}_i}\}$ in $\RR_{>0}^2$,
and we take $\hbar_1,\hbar_2\belongs \RR^2$ corresponding to $\{Y\le 1\}$ and $\{Z\le p^{-1}\}$ in $\RR_{>0}^2$.
Note that $\emptyset\ne \mathfrak{S}'_R\belongs \mathfrak{S}_R(\alpha)$.}
since $\set{Z\le p^{-1}}\cap \set{Y^{\abs{\mathsf{u}_\jboundary}} = p^{-k}Z^{\mathsf{u}_i}} \belongs \set{Y<1}$),
which is $$= p^{-(2\mathsf{u}_\jboundary-\delta+1) (k+\mathsf{u}_i)/\abs{\mathsf{u}_\jboundary}} p^{-(2\mathsf{u}_i-\delta+1)}
= p^{2k} p^{(\delta-1)(k+\mathsf{u}_i)/\abs{\mathsf{u}_\jboundary}} p^{\delta-1}
\le p^{2k} p^{(\delta-1)k/\abs{\mathsf{u}_\jboundary}} p^{\delta-1}.$$
This completes the proof of (2).
\end{proof}

We now handle the ``generic'' case of $p$ coprime to the numerator and denominator of $\alpha$.

\begin{lemma}
\label{LEM:generic-local-factor-estimate}
Suppose $v_p(\alpha) = 0$.
Let $$\mathfrak{G}=\mathfrak{G}_p(\bm{s},\lambda,\alpha)\defeq H^\vee_p(\bm{s},\lambda,0,\alpha) - 1.$$
Then $\mathfrak{G} \ll p^{2(\delta-1)}$.
\end{lemma}

\begin{proof}
For small $p$, use Lemma~\ref{LEM:general-fixed-p-bound} (which gives the satisfactory bound $\mathfrak{G} \ll_p 1$).
For large $p$, we roughly follow \cite{tanimoto2012height}*{proof of Lemma~5.10}.
Define $\mathcal{I}_p(x_0)$ and $\mathscr{D}^\ast_\jboundary$ as before.
Since $\mathcal{I}_p(x_0) = p^{-2} / (1-p^{-1}) = 1/\card{\mathscr{G}(\FF_p)}$ for all $x_0\notin \bigcup_{\jboundary\in J} \mathscr{D}_\jboundary(\FF_p)$, we have
\begin{equation*}
H^\vee_p(\bm{s},\lambda,0,\alpha) - 1
= \sum_{x_0\in \bigcup_{\jboundary\in J} \mathscr{D}_\jboundary(\FF_p)} \mathcal{I}_p(x_0)
= \sum_{x_0\in \bigcup_{\jboundary\in J_1\cup J_3} \mathscr{D}_\jboundary(\FF_p)} \mathcal{I}_p(x_0),
\end{equation*}
since $\mathcal{I}_p(x_0) = 0$ for all $x_0\in \mathscr{D}_\jboundary(\FF_p) \setminus \bigcup_{i\in J_1} \mathscr{D}_i(\FF_p)$ if $\jboundary\in J_2$.
We claim the following, which readily imply the lemma (by summing over $\mathscr{D}^\ast_\jboundary$ first, and then $\mathscr{D}_\jboundary\setminus \mathscr{D}^\ast_\jboundary$ separately):
\begin{enumerate}
    \item If $\jboundary\in J_1\cup J_3$ and $x_0\in \mathscr{D}^\ast_\jboundary(\FF_p)$, then $\mathcal{I}_p(x_0) \ll p^{\delta-3}$.

    \item If $\jboundary\in J_1\cup J_3$ and $x_0\in \mathscr{D}_\jboundary(\FF_p)$, then $\mathcal{I}_p(x_0) \ll p^{2\delta-2}$.
\end{enumerate}

For (2),
we integrate absolutely.
If $x_0\in \mathscr{D}_\jboundary(\FF_p) \setminus \bigcup_{i\in J\setminus \set{\jboundary}} \mathscr{D}_i(\FF_p)$ for some $\jboundary\in J_1\cup J_3$, then
\begin{equation}
\label{INEQ:crude-unit-alpha-generic-x_0-p-adic-integral-bound}
\mathcal{I}_p(x_0) \ll \int_{p\ZZ_p^2}
\abs{y}_p^{\Re(s_\jboundary-\mathsf{d}_\jboundary)} \, dy\,dz
\le \int_{p\ZZ_p^2}
\abs{y}_p^{2\mathsf{u}_\jboundary-\delta} \, dy\,dz
\ll p^{-(2\mathsf{u}_\jboundary-\delta+1)} p^{-1}
= p^{\delta-2-2\mathsf{u}_\jboundary},
\end{equation}
which is $\le p^{\delta-2}$ since $\mathsf{u}_\jboundary\ge 0$.
If $x_0\in (\mathscr{D}_\jboundary\cap \mathscr{D}_i)(\FF_p)$ for some distinct $\jboundary,i\in J$, then
\begin{equation*}
\mathcal{I}_p(x_0)
\ll \int_{p\ZZ_p^2:\, y^{\mathsf{u}_\jboundary}z^{\mathsf{u}_i}\in \ZZ_p}
\abs{y}_p^{\Re(s_\jboundary-\mathsf{d}_\jboundary)} \abs{z}_p^{\Re(s_i-\mathsf{d}_i)} \, dy\,dz
\le \int_{p\ZZ_p^2:\, y^{\mathsf{u}_\jboundary}z^{\mathsf{u}_i}\in \ZZ_p}
\abs{y}_p^{2\mathsf{u}_\jboundary-\delta} \abs{z}_p^{2\mathsf{u}_i-\delta} \, dy\,dz,
\end{equation*}
which (since $\abs{y^{\mathsf{u}_\jboundary}z^{\mathsf{u}_i}}_p\le 1$) is
$\le \int_{p\ZZ_p^2} \abs{y}_p^{-\delta} \abs{z}_p^{-\delta} \, dy\, dz
\ll p^{2\delta-2}$.
Thus (2) holds.

Yet (1) requires cancellation in some cases.
Suppose $x_0\in \mathscr{D}^\ast_\jboundary(\FF_p)$.
Claim~(1) is clear by \eqref{INEQ:crude-unit-alpha-generic-x_0-p-adic-integral-bound} if $\jboundary\in J_1$, because then $\mathsf{u}_\jboundary\ge 1$.
On the other hand, if $\jboundary\in J_3$, then
Lemma~\ref{LEM:large-p-local-analytic-coordinates-generically} gives
\begin{equation*}
\begin{split}
\mathcal{I}_p(x_0)
&= \int_{p\ZZ_p^2} \abs{y}_p^{s_\jboundary-\mathsf{d}_\jboundary} \phase(-\unitone y^{\mathsf{v}_\jboundary}\alpha\bmod{\ZZ_p}) \,\frac{dy\,dz}{1-p^{-1}} \\
&= \int_{p\ZZ_p^\times} \abs{y}_p^{s_\jboundary-\mathsf{d}_\jboundary+1} \phase(-\unitone y^{\mathsf{v}_\jboundary}\alpha\bmod{\ZZ_p}) \,\frac{dy/\abs{y}_p}{p-1}
= \frac{p^{-(s_\jboundary-\mathsf{d}_\jboundary+1)}}{p-1} \cdot \frac{-\bm{1}_{\abs{\mathsf{v}_\jboundary} = 1}}{p}
\ll p^{-\Re(s_\jboundary-\mathsf{d}_\jboundary+3)}
\end{split}
\end{equation*}
(since $c(\jboundary)\alpha\in \ZZ_p$ and $\mathsf{v}_\jboundary<0$),
which suffices for (1) since $\Re(s_\jboundary-\mathsf{d}_\jboundary) \ge -\delta$.
\end{proof}

Aside from the \emph{estimates} above, it is essential for us to have a local \emph{constancy} result:

\begin{lemma}
\label{LEM:local-constancy}
There exists $A_9\in \ZZ_{\ge 1}$, independent of $p$,
such that if
$$\alpha'/\alpha \equiv 1 \bmod{p^{\max(1,-v_p(\alpha))}A_9\ZZ_p},$$
then $\mathfrak{F}_p(\bm{s},\lambda,\alpha) = \mathfrak{F}_p(\bm{s},\lambda,\alpha')$,
where $\mathfrak{F} = \mathfrak{D}\cdot \bm{1}_{v_p(\alpha)<0}
+ \mathfrak{N}\cdot \bm{1}_{v_p(\alpha)>0}
+ \mathfrak{G}\cdot \bm{1}_{v_p(\alpha)=0}$.
\end{lemma}

\begin{proof}
Let $A_9\ge 1$ be a highly divisible integer.
By Lemmas~\ref{LEM:large-p-localized-height-symmetry-constancy} and~\ref{LEM:fixed-p-localized-height-symmetry-constancy} and Proposition~\ref{PROP:U-with-nonempty-c(U)-cover-X}, we may partition $X(\QQ_p)$ into compact open sets $\oset_{c,p}$ (indexed by a finite set of $c\in \QQ$) such that
(1) $\oset_{c,p}$ is invariant under an action of $x_1,x_2\in 1+pA_9\ZZ_p$ that is defined on $G(\QQ_p)$ by scaling $a$, $b-c$ by $x_1$, $x_2$, respectively;
and (2) $H_p(\bm{s},g)$ is invariant under this group action.
So by \eqref{EXPR:formula-for-H^vee_p}, we may decompose $H^\vee_p(\bm{s},\lambda,0,\alpha)$ (for any $\alpha\in \QQ_p^\times$) as a sum of the form $\sum_c \phase(c\alpha\bmod{\ZZ_p}) f_{23,c,p}(\bm{s},\lambda,\alpha)$, where $f_{23,c,p}$ is invariant under scaling of $\alpha\in \QQ_p^\times$ by $1+pA_9\ZZ_p$.
Since $\phase(c\alpha\bmod{\ZZ_p})$ depends only on $c\alpha\bmod{\ZZ_p}$, the desired lemma follows.

(Previously, a similar but more explicit non-archimedean symmetry argument has played an essential role in the circle method over function fields; see \cite{glas2022question}*{Lemma~3.6}.)
\end{proof}

\section{Archimedean endgame}
\label{SEC:archimedean-endgame}

In this section, assume $X$ is split and $D$ has strict normal crossings.
For convenience, let
$$\mathbf{u}\defeq \div(a)
=\sum_{\jboundary\in J} \mathsf{u}_\jboundary D_\jboundary\in \RR^J,
\quad
\mathbf{d}\defeq -\div(\omega)
=\sum_{\jboundary\in J} \mathsf{d}_\jboundary D_\jboundary\in \RR^J.$$
Let $C^\infty(X(\RR))$ be the set of smooth functions $X(\RR) \to \CC$
on the (possibly non-orientable) closed manifold $X(\RR)$.
The \emph{left-invariant} differential operators $a\,\frac{\partial}{\partial a}$ and $a\,\frac{\partial}{\partial b}$ play an important role in \cite{tanimoto2012height}*{proof of Lemma~5.9}, but the (non-invariant) operators $(b-c)\,\frac{\partial}{\partial b}$ help us too.

\begin{lemma}
\label{LEM:localized-regular-and-height-derivative-estimates}
Let $k_1,k_2,k_3\in \ZZ_{\ge 0}$ and $f_{24}\in C^\infty(X(\RR))$.
Suppose $f_{24}$ is supported on a compact set $K\belongs U(\RR)$
for some $U\in \mathcal{C}^{(2)}$
(see Proposition~\ref{PROP:U-with-nonempty-c(U)-cover-X}),
and choose $c\in \mathcal{C}^{(1)}(U)$.
\begin{enumerate}
\item The function $((b-c)\,\frac{\partial}{\partial b})^{k_1}
(a\,\frac{\partial}{\partial a})^{k_2}
(a\,\frac{\partial}{\partial b})^{k_3}
f_{24}\vert_{G(\RR)}$
extends to a function in $C^\infty(X(\RR))$.

\item Let $\jboundary\in J$.
The function $$H_{D_\jboundary,\infty}(a,b)
\cdot ((b-c)\,\tfrac{\partial}{\partial b})^{k_1}
(a\,\tfrac{\partial}{\partial a})^{k_2}
(a\,\tfrac{\partial}{\partial b})^{k_3}
[f_{24}(a,b) H_{D_\jboundary,\infty}(a,b)^{-1}]$$ on $G(\RR)$
extends to a function in $C^\infty(X(\RR))$.
\end{enumerate}
\end{lemma}

\begin{proof}
(1):
Let $x\in K$.
By the chain rule, it suffices to prove that the derivatives of
local coordinates $\tloc,\uloc\in \Rat(X)$ near $x$
are regular.
By algebraic Hartogs, it suffices to prove regularity
away from a codimension $2$ locus
(as in \cite{chambert2002distribution}*{proof of Proposition~2.2}).
Via analytic coordinates,
regularity then follows from Lemma~\ref{LEM:convenient-local-analytic-coordinates-generically}
(applicable by Definition~\ref{DEFN:C1U-maximizes-for-any-Dj-intersecting-U}).

(2):
By an induction with (1), we may assume $k_1+k_2+k_3=1$.
By the product rule,
together with (1) and Definition~\ref{DEFN:standard-local-Weil-height-for-boundary-divisors},
it suffices to show that the derivative of $\tloc\in \Rat(X)$,
where $\tloc=0$ is a local equation for $D_\jboundary$,
is divisible by $\tloc$.
It suffices to do this away from a codimension $2$ locus.
Divisibility then follows from Lemma~\ref{LEM:convenient-local-analytic-coordinates-generically}.
\end{proof}

Via Proposition~\ref{PROP:U-with-nonempty-c(U)-cover-X},
cover $X(\RR)$ by finitely many open sets $\Omega\in \set{U(\RR): U\in \mathcal{C}^{(2)}}$.
Take a smooth partition of unity
$1 = \sum_\Omega w_\Omega$ of $X(\RR)$
subordinate to $\set{\Omega}$.
For each $\Omega$,
choose a constant $\cfinal_0=\cfinal_0(\Omega)\in \mathcal{C}^{(1)}(U)$,
for any $U\in \mathcal{C}^{(2)}$ with $U(\RR)=\Omega$.
Let
\begin{equation*}
\mathcal{I}_\Omega = \mathcal{I}_\Omega(\bm{s},\lambda,\alpha)
\defeq \int_{G(\RR)} w_\Omega(g_\infty)
H_\infty(\bm{s},g_\infty)^{-1}
\phase(\alpha b_\infty)
\lambda_\infty(\alpha a_\infty)\, dg_\infty,
\quad \mathcal{J}_\Omega
\defeq \frac{\mathcal{I}_\Omega}{\phase(\cfinal_0\alpha)}.
\end{equation*}
Additionally, we will need to study $\alpha \ll 1$ and $\abs{\alpha} \gg 1$ separately;
so take a smooth partition of unity $1 = w_0 + w_\infty$ of $\RR$
subordinate to the open sets $\openint(-2,2)$ and $\RR \setminus [-1,1]$, and let
\begin{equation*}
\mathcal{I}_{\Omega,0} \defeq w_0(\alpha) \mathcal{I}_\Omega,
\quad \mathcal{I}_{\Omega,\infty} \defeq w_\infty(\alpha) \mathcal{I}_\Omega,
\quad \mathcal{J}_{\Omega,0} \defeq w_0(\alpha) \mathcal{J}_\Omega,
\quad \mathcal{J}_{\Omega,\infty} \defeq w_\infty(\alpha) \mathcal{J}_\Omega.
\end{equation*}
Then we have the decomposition
\begin{equation}
\label{decompose_Hveeinfty}
H^\vee_\infty(\bm{s},\lambda,0,\alpha)
= \sum_\Omega \mathcal{I}_\Omega
= \sum_\Omega \phase(\cfinal_0\alpha)
(\mathcal{J}_{\Omega,0} + \mathcal{J}_{\Omega,\infty}).
\end{equation}

\begin{lemma}
\label{LEM:uniform-derivative-bound-on-localized-modulated-integrals}
For all $k_1,k_2\in \ZZ_{\ge 0}$,
$\alpha\in \RR^\times$,
$t\in \RR$,
and $\bm{s}$ in the region \eqref{INEQ:key-new-region},
we have
\begin{equation*}
(\alpha\,\tfrac{\partial}{\partial\alpha})^{k_2}
{\mathcal{J}_\Omega(\bm{s}-it\mathbf{u},\lambda,\alpha)}
\ll_{k_1,k_2,A_8,\delta}
\frac{1+\norm{\bm{s}}^{2+k_1+2k_2}}{\alpha^2 (1+\abs{t}^{k_1})}.
\end{equation*}
The same estimate holds if we replace $\mathcal{J}_\Omega$
with $\mathcal{J}_{\Omega,0}$ or $\mathcal{J}_{\Omega,\infty}$.
(Here $\norm{\bm{s}}\defeq \max_\jboundary{\abs{s_\jboundary}}$.)
\end{lemma}




\begin{proof}
(The case $k_2=0$ is essentially \cite{tanimoto2012height}*{Lemma~5.9}.
But differentiability and uniform control for $k_2\ge 1$ are not obvious;
cf.~the ``junior arc'' difficulties in \cite{hooley1986HasseWeil}.
Nonetheless,
``nicer than expected'' derivatives have appeared in other settings as well;
see e.g.~the integral derivative estimates in \cite{heath1996new},
in the circle method for \emph{homogeneous} equations.)

Writing $dg\defeq dg_\infty$, etc.,
and applying \eqref{EQN:local-character-height-consistency-relation}
(cf.~\eqref{EQN:change-of-t-identity}),
we have
$$\mathcal{I}_\Omega(\bm{s}-it\mathbf{u},\lambda,\alpha)
= \int_{G(\RR)} w_\Omega(g)
H_\infty(\bm{s},g)^{-1} \phase(\alpha b) \lambda_\infty(\alpha a)
\abs{a}_\infty^{-it}\, dg.$$
Integrating by parts twice in $b$,
then $k_1+k_2$ times in $\log{\abs{a}}$,
gives (by Lemma~\ref{LEM:localized-regular-and-height-derivative-estimates})
\begin{equation*}
\mathcal{I}_\Omega(\bm{s}-it\mathbf{u},\lambda,\alpha)
= \alpha^{-2} (2+it)^{-k_1-k_2} \int_{G(\RR)} w_{\Omega,1}(\bm{s},g)
H_\infty(\bm{s},g)^{-1} \phase(\alpha b) \lambda_\infty(\alpha a)
\abs{a}_\infty^{-2-it} \, dg
\end{equation*}
where $w_{\Omega,1}$ is a degree $2+k_1+k_2$ polynomial function of
$\bm{s}$ and finitely many elements of $C^\infty(X(\RR))$;
cf.~\cite{tanimoto2012height}*{proof of Lemma~5.9}.
Letting $(a', b') = (\alpha a, \alpha (b-\cfinal_0))$, we get
\begin{equation*}
\mathcal{J}_\Omega(\bm{s}-it\mathbf{u},\lambda,\alpha)
= \frac{\abs{\alpha}_\infty^{it}}{(2+it)^{k_1+k_2}} \int_{(a',b')\in G(\RR)}
\frac{w_{\Omega,1}(\bm{s},(a,b))}
{H_\infty(\bm{s},(a,b))}
\frac{\phase(b') \lambda_\infty(a')}{\abs{a'}_\infty^{2+it}}
\, \frac{db'}{\abs{\alpha}_\infty}\, \frac{da'}{\abs{a'}_\infty},
\end{equation*}
where $(a,b) = (a'/\alpha, \cfinal_0+b'/\alpha)$
is viewed as a function of $(a',b',\alpha)$.
Applying $\alpha\,\frac{\partial}{\partial\alpha}$ repeatedly ($k_2$ times)
using Lemma~\ref{LEM:localized-regular-and-height-derivative-estimates}
and the chain rule identity
\begin{equation*}
\alpha\,\frac{\partial\varphi(a,b)}{\partial\alpha}
= \alpha\,\frac{\partial{a}}{\partial\alpha}\frac{\partial\varphi(a,b)}{\partial{a}}
+ \alpha\,\frac{\partial{b}}{\partial\alpha}\frac{\partial\varphi(a,b)}{\partial{b}}
= -a\,\frac{\partial\varphi(a,b)}{\partial{a}}
- (b-\cfinal_0)\,\frac{\partial\varphi(a,b)}{\partial{b}},
\end{equation*}
and then changing variables back from $(a',b')$ to $(a,b)$ (and applying Lemma~\ref{LEM:snc-implies-nice-densities-implies-Tamagawa-extraction}(1)$\Rightarrow$(2)), we get the desired bound.
\end{proof}

We now rewrite \eqref{EQN:difficult-part-of-spectral-expansion}.
Plugging \eqref{EQN:change-of-t-identity} into \eqref{EQN:difficult-part-of-spectral-expansion}
and recalling the definition of $\mathscr{S}$ from \eqref{EQN:define-shifted-integral},
we get
\begin{equation*}
\mathsf{Z}_1(\bm{s},1_G)
= \sum_{\lambda\in \Mbad:\, \lambda(-1)=1}\,
\sum_{\alpha\in \QQ^\times}
\mathscr{S}{\left(\prod_v H^\vee_v(\bm{s},\lambda,0,\alpha)\right)},
\end{equation*}
for large $\Re(\bm{s})$.
Decomposing the factor $H^\vee_\infty$
using \eqref{decompose_Hveeinfty}
gives
\begin{equation}
\label{london}
\mathsf{Z}_1(\bm{s},1_G)
= \sum_{f_{25},\lambda}\,
\sum_{\alpha\in \QQ^\times} \phase(\cfinal_0(\Omega)\alpha)
\mathscr{S}{\left(f_{25}(\bm{s},\lambda,\alpha)
\prod_p H^\vee_p(\bm{s},\lambda,0,\alpha)\right)}.
\end{equation}
where $f_{25}$ ranges over the set
$\bigcup_\Omega \set{\mathcal{J}_{\Omega,0}, \mathcal{J}_{\Omega,\infty}}$,
and where $\lambda\in \Mbad$ with $\lambda(-1)=1$.

For each $p$, we now decompose $H^\vee_p$
in terms of $\mathfrak{N}_p$, $\mathfrak{D}_p$, or $\mathfrak{G}_p$,
according as $v_p(\alpha)>0$, $v_p(\alpha)<0$, or $v_p(\alpha)=0$,
respectively.
Specifically, if $v_p(\alpha)>0$, then
by the definition of $\mathfrak{N}_p$ from Lemma~\ref{LEM:numerator-bias}
we have
\begin{equation}
\label{paris-1}
H^\vee_p = \mathfrak{N}_p
+ \sum_{\jboundary\in J^\ast_2} p^{-v_p(\alpha) (s_\jboundary-\mathsf{d}_\jboundary+1)/\abs{\mathsf{u}_\jboundary}} \bm{1}_{\mathsf{u}_\jboundary\mid v_p(\alpha)};
\end{equation}
if $v_p(\alpha)<0$, then
by the definition of $\mathfrak{D}_p$ from Lemma~\ref{LEM:denominator-bias}
we have
\begin{equation}
\label{paris-2}
H^\vee_p = \mathfrak{D}_p
+ \bm{1}_{p\notin \ol{\Sbad}} \sum_{c\in \QQ}
\sum_{\jboundary\in J^c_1} \phase(-c\alpha \bmod{\ZZ_p}) p^{-\abs{v_p(\alpha)} (s_\jboundary-\mathsf{d}_\jboundary+1)/\mathsf{u}_\jboundary} \bm{1}_{\mathsf{u}_\jboundary\mid v_p(\alpha)};
\end{equation}
and if $v_p(\alpha)=0$, then
by the definition of $\mathfrak{G}_p$ from Lemma~\ref{LEM:generic-local-factor-estimate}
we have $H^\vee_p = 1 + \mathfrak{G}_p$.

For convenience,
let $$\mathfrak{m}\maps \set{1,\dots,\kspecial} \to J^\ast_2,
\quad \mathfrak{n}\maps \set{1,\dots,\lspecial} \to \bigcup_{q\in \QQ} J^q_1$$
be bijective functions, so that
$\kspecial,\lspecial\in \ZZ_{\ge 0}$
are the number of special divisors in $J_2$ and in $J_1$, respectively.
Let $\cfinal_\jspecial\defeq q$ if $\mathfrak{n}(\jspecial)\in J^q_1$,
and for any $\ispecial\in \set{1,\dots,\kspecial}$
or $\jspecial\in \set{1,\dots,\lspecial}$ let
\begin{equation*}
\beta_\ispecial
\defeq s_{\mathfrak{m}(\ispecial)}-\mathsf{d}_{\mathfrak{m}(\ispecial)}+1,
\quad \gamma_\jspecial
\defeq s_{\mathfrak{n}(\jspecial)}-\mathsf{d}_{\mathfrak{n}(\jspecial)}+1,
\quad \mathfrak{u}_\ispecial
\defeq \abs{\mathsf{u}_{\mathfrak{m}(\ispecial)}},
\quad \mathfrak{v}_\jspecial
\defeq \mathsf{u}_{\mathfrak{n}(\jspecial)},
\end{equation*}
and let $\mathfrak{F}_n \defeq \prod_{p\mid n} \mathfrak{F}_p$
for all $\mathfrak{F}\in \set{\mathfrak{N}, \mathfrak{D}, \mathfrak{G}}$.
Write $\abs{\alpha} = \alpha_1/\alpha_2$,
where $\alpha_1,\alpha_2\in \ZZ_{\ge 1}$ and $\gcd(\alpha_1,\alpha_2)=1$.
Multiplying \eqref{paris-1} over $p\mid \alpha_1$, we find that
\begin{equation}
\label{wien-1}
\prod_{p:\, v_p(\alpha)>0} H^\vee_p(\bm{s},\lambda,0,\alpha)
= \sum_{\substack{m_0,\dots,m_\kspecial\ge 1: \\
\textnormal{pairwise coprime}, \\
m_0m_1^{\mathfrak{u}_1}\cdots m_\kspecial^{\mathfrak{u}_\kspecial}
= \alpha_1}}\, 
\frac{\mathfrak{N}_{m_0}(\bm{s},\lambda,\alpha)}
{m_1^{\beta_1}\cdots m_\kspecial^{\beta_\kspecial}},
\end{equation}
where $\mathfrak{N}_{m_0}$
comes from the term $\mathfrak{N}_p$ in \eqref{paris-1} for $p\mid m_0$,
and where $m_\ispecial^{-\beta_\ispecial}$
comes from the term
$$p^{-v_p(m_\ispecial^{\mathfrak{u}_\ispecial})
(s_{\mathfrak{m}(\ispecial)}-\mathsf{d}_{\mathfrak{m}(\ispecial)}+1)/\abs{\mathsf{u}_{\mathfrak{m}(\ispecial)}}}
\bm{1}_{\mathsf{u}_{\mathfrak{m}(\ispecial)}\mid v_p(m_\ispecial^{\mathfrak{u}_\ispecial})}
= p^{-v_p(m_\ispecial) \beta_\ispecial}$$
in \eqref{paris-1} for $p\mid m_\ispecial$.
Similarly, multiplying \eqref{paris-2} over $p\mid \alpha_2$, we find that
\begin{equation}
\label{wien-2}
\prod_{p:\, v_p(\alpha)<0} H^\vee_p(\bm{s},\lambda,0,\alpha)
= \sum_{\substack{n_0,\dots,n_\lspecial\ge 1: \\
\textnormal{pairwise coprime}, \\
p\mid n_1\cdots n_\lspecial \Rightarrow p\notin \ol{\Sbad}, \\
n_0n_1^{\mathfrak{v}_1}\cdots n_\lspecial^{\mathfrak{v}_\lspecial}
= \alpha_2}}\, 
\mathfrak{D}_{n_0}(\bm{s},\lambda,\alpha)
\prod_{1\le \jspecial\le \lspecial}
\frac{\phase(-\cfinal_\jspecial\alpha\bmod{\ZZ_{n_\jspecial}})}
{n_\jspecial^{\gamma_\jspecial}}.
\end{equation}
Finally,
\begin{equation}
\label{wien-3}
\prod_{p:\, v_p(\alpha)=0} H^\vee_p(\bm{s},\lambda,0,\alpha)
= \sum_{\substack{r\ge 1: \\
\gcd(r,\alpha_1\alpha_2)=1, \\
\textnormal{$r$ square-free}}}\, 
\mathfrak{G}_r(\bm{s},\lambda,\alpha).
\end{equation}

Multiplying \eqref{wien-1}, \eqref{wien-2}, and \eqref{wien-3}
together, we find that
\begin{equation*}
\prod_p H^\vee_p(\bm{s},\lambda,0,\alpha)
= \sum_{\substack{r,m_0,\dots,m_\kspecial,n_0,\dots,n_\lspecial\ge 1: \\
\textnormal{pairwise coprime}, \\
\textnormal{$r$ square-free}, \\
p\mid n_1\cdots n_\lspecial \Rightarrow p\notin \ol{\Sbad}, \\
(m_0m_1^{\mathfrak{u}_1}
\cdots m_\kspecial^{\mathfrak{u}_\kspecial})/(n_0n_1^{\mathfrak{v}_1}
\cdots n_\lspecial^{\mathfrak{v}_\lspecial})
= \abs{\alpha}}}\, 
\frac{(\mathfrak{N}_{m_0} \mathfrak{D}_{n_0} \mathfrak{G}_r)(\bm{s},\lambda,\alpha)}
{m_1^{\beta_1}\cdots m_\kspecial^{\beta_\kspecial}}
\prod_{1\le \jspecial\le \lspecial}
\frac{\phase(-\cfinal_\jspecial\alpha\bmod{\ZZ_{n_\jspecial}})}
{n_\jspecial^{\gamma_\jspecial}}.
\end{equation*}
Plugging this into \eqref{london},
we conclude that for $\Re(\bm{s})$ large,
the contribution to \eqref{EQN:difficult-part-of-spectral-expansion} from $\alpha>0$
(the case $\alpha<0$ being completely analogous)
equals
\begin{equation}
\label{EXPR:final-expanded-adelic-data-sum}
\sum_{f_{25}, \lambda}
\, \sum_{\substack{r,m_0,\dots,m_\kspecial,n_0,\dots,n_\lspecial\ge 1: \\
\textnormal{pairwise coprime}, \\
\textnormal{$r$ square-free}, \\
p\mid n_1\cdots n_\lspecial \Rightarrow p\notin \ol{\Sbad}}}
\, \mathscr{S}{\left(
\frac{(\mathfrak{N}_{m_0} \mathfrak{D}_{n_0} \mathfrak{G}_r f_{25})
(\bm{s},\lambda,\alpha)
\phase(\cfinal_0\alpha)}{m_1^{\beta_1}\cdots m_\kspecial^{\beta_\kspecial}}
\prod_{1\le \jspecial\le \lspecial}
\frac{\phase(-\cfinal_\jspecial\alpha\bmod{\ZZ_{n_\jspecial}})}
{n_\jspecial^{\gamma_\jspecial}} \right)},
\end{equation}
where for any given $m_0,\dots,m_\kspecial,n_0,\dots,n_\lspecial\ge 1$ we let
\begin{equation}
\label{alpha-factorize}
\alpha=\alpha(m_0,\dots,m_\kspecial,n_0,\dots,n_\lspecial)
\defeq \frac{m_0m_1^{\mathfrak{u}_1}\cdots m_\kspecial^{\mathfrak{u}_\kspecial}}
{n_0n_1^{\mathfrak{v}_1}\cdots n_\lspecial^{\mathfrak{v}_\lspecial}}.
\end{equation}
(Cf.~\eqref{EXPR:desired-key-general-multiple-Dirichlet-series}.)
Let $\mathsf{P}(\alpha) \defeq m_1\cdots m_\kspecial n_1\cdots n_\lspecial$,
for convenience.



Before proceeding, we decompose \eqref{EXPR:final-expanded-adelic-data-sum} into three pieces.
By reordering if necessary, assume
\begin{equation*}
\set{\cfinal_\jspecial: 1\le \jspecial\le \lspecial}
= \set{\cfinal_1,\dots,\cfinal_{\lspecial'}},
\end{equation*}
where $\cfinal_1,\dots,\cfinal_{\lspecial'}$ are pairwise distinct
and $\lspecial'\le \lspecial$.
For each $\jspecial'\in \set{1,\dots,\lspecial'}$, let
\begin{equation*}
n[\jspecial']
\defeq \prod_{1\le \jspecial\le \lspecial:\, \cfinal_\jspecial=\cfinal_{\jspecial'}}
n_\jspecial^{\mathfrak{v}_\jspecial}.
\end{equation*}
Let $\xi>0$ be small.
Let $\Xi \defeq \xi^{1/2}$.
We are now prepared to define some important functions
\begin{equation*}
\mathcal{P}_i = \mathcal{P}_i(r,m_0,\dots,m_\kspecial,n_0,\dots,n_\lspecial),
\end{equation*}
where $i\in \{1,2,3\}$.
Let $\mathcal{P}_3 \defeq 1 - \mathcal{P}_1 - \mathcal{P}_2$, where
\begin{align*}
\mathcal{P}_1
&\defeq w_0{\left(\frac{rm_0n_0\alpha}{\mathsf{P}(\alpha)^\xi}\right)}
w_0{\left(\frac{rm_0n_0\alpha^{-1}}{\mathsf{P}(\alpha)^\xi}\right)}
\sum_{\mathfrak{V}\belongs \set{1,\dots,\lspecial'}:\, \card{\mathfrak{V}}\le 1}\,
\prod_{\jspecial'\in \mathfrak{V}}
w_\infty{\left(\frac{n[\jspecial']}{\mathsf{P}(\alpha)^\Xi}\right)}
\prod_{\jspecial'\notin \mathfrak{V}}
w_0{\left(\frac{n[\jspecial']}{\mathsf{P}(\alpha)^\Xi}\right)}, \\
\mathcal{P}_2
&\defeq w_0{\left(\frac{rm_0n_0\alpha}{\mathsf{P}(\alpha)^\xi}\right)}
w_0{\left(\frac{rm_0n_0\alpha^{-1}}{\mathsf{P}(\alpha)^\xi}\right)}
\sum_{\mathfrak{V}\belongs \set{1,\dots,\lspecial'}:\, \card{\mathfrak{V}}\ge 2}\,
\prod_{\jspecial'\in \mathfrak{V}}
w_\infty{\left(\frac{n[\jspecial']}{\mathsf{P}(\alpha)^\Xi}\right)}
\prod_{\jspecial'\notin \mathfrak{V}}
w_0{\left(\frac{n[\jspecial']}{\mathsf{P}(\alpha)^\Xi}\right)}.
\end{align*}
Using smooth, as opposed to sharp, cutoffs, helps in $\mathcal{P}_1$ (see the proof of Lemma~\ref{LEM:range-1}).

\begin{definition}
\label{piece-defined-by-P_i}
The \emph{piece of \eqref{EXPR:final-expanded-adelic-data-sum} defined by $\mathcal{P}_i$} is
the quantity
\begin{equation*}
\sum_{f_{25}, \lambda}
\, \sum_{\substack{r,m_0,\dots,m_\kspecial,n_0,\dots,n_\lspecial\ge 1: \\
\textnormal{pairwise coprime}, \\
\textnormal{$r$ square-free}, \\
p\mid n_1\cdots n_\lspecial \Rightarrow p\notin \ol{\Sbad}}}
\, \mathcal{P}_i\cdot \mathscr{S}{\left(
\frac{(\mathfrak{N}_{m_0} \mathfrak{D}_{n_0} \mathfrak{G}_r f_{25})
(\bm{s},\lambda,\alpha)
\phase(\cfinal_0\alpha)}{m_1^{\beta_1}\cdots m_\kspecial^{\beta_\kspecial}}
\prod_{1\le \jspecial\le \lspecial}
\frac{\phase(-\cfinal_\jspecial\alpha\bmod{\ZZ_{n_\jspecial}})}
{n_\jspecial^{\gamma_\jspecial}} \right)},
\end{equation*}
where $\mathcal{P}_i = \mathcal{P}_i(r,m_0,\dots,m_\kspecial,n_0,\dots,n_\lspecial)$.
\end{definition}

In other words, Definition~\ref{piece-defined-by-P_i} introduces a weight of $\mathcal{P}_i$ into the original sum \eqref{EXPR:final-expanded-adelic-data-sum}.
Roughly speaking,
$\mathcal{P}_3$ is supported on very degenerate ranges of variables,
$\mathcal{P}_2$ is supported on ranges where two or more of the quantities $n[\jspecial']$ are large,
and $\mathcal{P}_1$ is supported on ranges where only one of the quantities $n[\jspecial']$ is large.
We formalize these support conditions in the next paragraph.
For now, we mention that the weight $\mathcal{P}_3$ can be handled quite crudely, the weight $\mathcal{P}_2$ allows for significant cancellation via a global reciprocity argument, and the weight $\mathcal{P}_1$ allows for cancellation via local averaging modulo all but the largest quantity $n[\jspecial']$.

Since $\Supp{w_0} \belongs \openint(-2,2)$
and $\Supp{w_\infty} \belongs \RR \setminus [-1,1]$,
we see that if $n\in \ZZ_{\ge 1}$,
then $$w_0(n/\mathsf{P}(\alpha)^\Xi) \ne 0
\Rightarrow n<2\mathsf{P}(\alpha)^\Xi,
\quad w_\infty(n/\mathsf{P}(\alpha)^\Xi) \ne 0
\Rightarrow n>\mathsf{P}(\alpha)^\Xi.$$
Therefore, the function $\mathcal{P}_1$ is supported on
\begin{equation}
\label{COND:define-dominant-Z_1-range}
rm_0n_0 \max(\alpha, \alpha^{-1}) < 2\mathsf{P}(\alpha)^\xi,
\quad
\#\set{\jspecial'\in \set{1,\dots,\lspecial'}:
n[\jspecial'] < 2\mathsf{P}(\alpha)^\Xi}
\ge \lspecial'-1
\end{equation}
(i.e.~every point in the support of $\mathcal{P}_1$ satisfies \eqref{COND:define-dominant-Z_1-range}),
and $\mathcal{P}_2$ is supported on
\begin{equation}
\label{COND:define-challenging-Kloosterman-oscillation-Z_1-range}
rm_0n_0 \max(\alpha, \alpha^{-1}) < 2\mathsf{P}(\alpha)^\xi,
\quad
\#\set{\jspecial'\in \set{1,\dots,\lspecial'}:
n[\jspecial'] > \mathsf{P}(\alpha)^\Xi}
\ge 2.
\end{equation}
Also, since $w_0\vert_{[-1,1]} = 1$ and $(w_0+w_\infty)(\frac{n[\jspecial']}{\mathsf{P}(\alpha)^\Xi}) = 1$, we see that if $rm_0n_0 \max(\alpha, \alpha^{-1}) \le \mathsf{P}(\alpha)^\xi$, then $\mathcal{P}_1 + \mathcal{P}_2 = 1$.
Therefore, $\mathcal{P}_3$ is supported on
\begin{equation}
\label{COND:define-absolute-decay-range-in-Z_1}
rm_0n_0 \max(\alpha, \alpha^{-1}) > \mathsf{P}(\alpha)^\xi.
\end{equation}

Recall the definition of $\mathcal{H}_J(-\delta,\infty)$
from \S\ref{SUBSEC:polar-combinatorics}.
Let
\begin{equation}
\label{define-translate-space-H-star}
\mathcal{H}_\star(\delta)
\defeq \set{h(\bm{s}):
h(\bm{z}+\mathbf{d}+2\mathbf{u})\in \mathcal{H}_J(-\delta,\infty)},
\end{equation}
where $h(\bm{z}+\mathbf{d}+2\mathbf{u})$ is viewed as a function of $\bm{z}$.
The point of this definition is that
if $\bm{s}$ lies in the region \eqref{INEQ:key-new-region},
then $\bm{z}\defeq \bm{s}-\mathbf{d}-2\mathbf{u}$
satisfies $\Re(z_\jboundary)\ge -\delta$ for all $\jboundary\in J$.

\begin{lemma}
\label{LEM:range-3}
Let $\xi>0$ be small.
Then the piece of \eqref{EXPR:final-expanded-adelic-data-sum} defined by $\mathcal{P}_3$
extends to an element of $\mathcal{H}_\star(\delta)$, provided $\delta$ is sufficiently small (in terms of $\xi$).
\end{lemma}


\begin{proof}
Assume $\delta\le \xi$.
We would like to apply Lemma~\ref{LEM:shifted-integral-polar-and-residue-structure}(4), but will need to treat small and large $\alpha$ separately, using two different $\RR\div(a)$-translates of the variable $\bm{s}$.
Let
\begin{equation*}
F_{26,-1}(\bm{s}) = \sum_{f_{25}, \lambda} \sum_{r,m_0,\dots}
\bm{1}_{\alpha<1} \cdot \mathcal{P}_3 \cdot
\frac{(\mathfrak{N}_{m_0} \mathfrak{D}_{n_0} \mathfrak{G}_r f_{25})(\bm{s},\lambda,\alpha)
\phase(\cfinal_0\alpha)}{m_1^{\beta_1}\cdots m_\kspecial^{\beta_\kspecial}}
\prod_{1\le \jspecial\le \lspecial}
\frac{\phase(-\cfinal_\jspecial\alpha\bmod{\ZZ_{n_\jspecial}})}
{n_\jspecial^{\gamma_\jspecial}},
\end{equation*}
and let $F_{26,1}(\bm{s})$ be the corresponding sum
with $\bm{1}_{\alpha\ge 1}$ in place of $\bm{1}_{\alpha<1}$.
In $F_{26,\varsigma}$, where $\varsigma\in \set{\pm 1}$,
the weight $\mathcal{P}_3(r,m_0,n_0,\dots)$ implies,
via \eqref{COND:define-absolute-decay-range-in-Z_1},
that
$rm_0n_0 \alpha^\varsigma
> \mathsf{P}(\alpha)^\xi$.
Therefore, upon taking absolute values in $F_{26,\varsigma}$, and plugging in Lemma~\ref{LEM:uniform-derivative-bound-on-localized-modulated-integrals} and the inequality $1 \le (rm_0n_0 \alpha^\varsigma)^\xi / \mathsf{P}(\alpha)^{\xi^2}$,
we get that $F_{26,\varsigma}(\bm{s}-\varsigma\xi\mathbf{u} - it\mathbf{u})$ is $\ll_\xi (1+\norm{\bm{s}}^4)/(1+t^2)$ times
\begin{equation*}
\sum_{r,m_0,\dots} \frac{\abs{\mathfrak{G}_r\mathfrak{N}_{m_0}\mathfrak{D}_{n_0}(\bm{s}-\varsigma\xi\mathbf{u} - it\mathbf{u},\lambda,\alpha)}
\, (n_0/m_0)^2(\alpha n_0/m_0)^{-\varsigma\xi}}{(m_1\cdots m_\kspecial n_1\cdots n_\lspecial)^{1-\delta}}
\cdot \frac{(rm_0n_0 \alpha^\varsigma)^\xi}{(m_1\cdots m_\kspecial n_1\cdots n_\lspecial)^{\xi^2}},
\end{equation*}
for all $\bm{s}$ in \eqref{INEQ:key-new-region} and $t\in \RR$,
in view of the inequality
\begin{equation*}
\begin{split}
\frac{\alpha^{-2}}{\prod_{1\le \ispecial\le \kspecial} \abs{m_\ispecial^{\beta_\ispecial-\varsigma\xi\mathsf{u}_{\mathfrak{m}(\ispecial)}}}}
\prod_{1\le \jspecial\le \lspecial}
\frac{1}{\abs{n_\jspecial^{\gamma_\jspecial-\varsigma\xi\mathsf{u}_{\mathfrak{n}(\jspecial)}}}}
&\le \frac{\alpha^{-2}}{\prod_{1\le \ispecial\le \kspecial} m_\ispecial^{(2-\varsigma\xi)\mathsf{u}_{\mathfrak{m}(\ispecial)}+1-\delta}}
\prod_{1\le \jspecial\le \lspecial}
\frac{1}{n_\jspecial^{(2-\varsigma\xi)\mathsf{u}_{\mathfrak{n}(\jspecial)}+1-\delta}} \\
&= \frac{\alpha^{-2} (\alpha n_0/m_0)^{2-\varsigma\xi}}{(m_1\cdots m_\kspecial n_1\cdots n_\lspecial)^{1-\delta}}
= \frac{(n_0/m_0)^2 (\alpha n_0/m_0)^{-\varsigma\xi}}{(m_1\cdots m_\kspecial n_1\cdots n_\lspecial)^{1-\delta}},
\end{split}
\end{equation*}
where $\alpha^{-2}$ comes from Lemma~\ref{LEM:uniform-derivative-bound-on-localized-modulated-integrals}.
Observe that the powers of $\alpha^\varsigma$ above cancel out.

When all summation variables but $n_0$ are fixed,
Lemma~\ref{LEM:denominator-bias}
implies
\begin{equation*}
\sum_{n_0\ge 1}
n_0^{2-\varsigma\xi} n_0^\xi
\abs{\mathfrak{D}_{n_0}}
\le \prod_p \left(1 + O_\xi(p^{O(\xi)}) \sum_{\jboundary\in J_1} \left(p^{-(1 + 1/\mathsf{u}_\jboundary)}
+ p^{-1/2} p^{-1} + p^{-(\mathsf{u}_\jboundary+1)/\mathsf{u}_\jboundary}\right)\right)
\ll_\xi 1,
\end{equation*}
provided $\bm{s}$ lies in \eqref{INEQ:key-new-region} and $\xi$ is sufficiently small.
If we then similarly use Lemma~\ref{LEM:numerator-bias} to sum over $m_0$, and Lemma~\ref{LEM:generic-local-factor-estimate} to sum over $r$, we get
(suppressing coprimality restrictions)
\begin{equation}
\label{INEQ:model-r,m_0,n_0-total-bound}
\sum_{r,m_0,n_0\ge 1}
(n_0/m_0)^{2-\varsigma\xi} (rm_0n_0)^\xi
\abs{\mathfrak{G}_r\mathfrak{N}_{m_0}\mathfrak{D}_{n_0}(\bm{s}-\varsigma\xi\mathbf{u} - it\mathbf{u},\lambda,\alpha)}
\ll_\xi 1.
\end{equation}
So if $\delta < \xi^2$, the sum $F_{26,\varsigma}(\bm{s}-\varsigma\xi\mathbf{u} - it\mathbf{u})$ (and its analog with absolute values everywhere) is
\begin{equation*}
\ll_\xi \frac{1+\norm{\bm{s}}^{O(1)}}{1+t^2}
\sum_{m_1,\dots,m_\kspecial,n_1,\dots,n_\lspecial\ge 1}
\frac{(m_1\cdots m_\kspecial n_1\cdots n_\lspecial)^{\delta-1}}{(m_1\cdots m_\kspecial n_1\cdots n_\lspecial)^{\xi^2}}
\ll_{\delta,\xi} \frac{1+\norm{\bm{s}}^{O(1)}}{1+t^2}.
\end{equation*}
Applying Lemma~\ref{LEM:shifted-integral-polar-and-residue-structure}(4) to $\mathscr{S}(F_{26,\varsigma})$ now gives Lemma~\ref{LEM:range-3},
since the piece of \eqref{EXPR:final-expanded-adelic-data-sum} defined by $\mathcal{P}_3$
is precisely $\sum_{\varsigma\in \set{\pm 1}} \mathscr{S}(F_{26,\varsigma})(\bm{s}) = \sum_{\varsigma\in \set{\pm 1}} \mathscr{S}(F_{26,\varsigma})(\bm{s}-\varsigma\xi\mathbf{u})$ (when $\Re(\bm{s})$ is large).
\end{proof}

For $\mathcal{P}_2$, we use a Weyl-type inequality for monomials in several variables.
There are many ways to establish such inequalities (see e.g.~\cites{parsell2012hua,parsell2013near,bourgain2022multi} and references within).
However, handling technical issues such as lopsidedness and coprimality requires care.

\begin{proposition}
\label{PROP:Weyl-inequality-for-monomials-in-several-variables}
Let $M_1,\dots,M_\kspecial,y,q\in \ZZ_{\ge 1}$ with $\gcd(y,q)=1$.
Let $P_1,\dots,P_\kspecial$ be arithmetic progressions contained in $[M_1,2M_1),\dots,[M_\kspecial,2M_\kspecial)$, respectively, all with modulus $\le R$.
Let $K(I) = \sum_{i\in I} (\mathfrak{u}_i-1)$.
Then for any nonempty set $I\belongs \set{1,\dots,\kspecial}$, we have
\begin{equation*}
\sum_{(m_1,\dots,m_\kspecial) \in P_1 \times \dots \times P_\kspecial}
\phase(ym_1^{\mathfrak{u}_1}\cdots m_\kspecial^{\mathfrak{u}_\kspecial}/q)
\ll_\eps \frac{R^{O(1)} (M_1\cdots M_\kspecial)^{1+\eps}}
{\min(q \prod_{i\notin I} M_i^{-\mathfrak{u}_i},
\min_{i\in I}(M_i), q^{-1} \prod_{i\in I} M_i^{\mathfrak{u}_i})^{1/2^{K(I)}}}.
\end{equation*}
\end{proposition}

\begin{proof}
We first prove a simpler case, and then the general case.

\emph{Case~1: $M_i=1$ and $P_i=\set{1}$ for all $i\notin I$.}
For each $i\in I$, let $P_i-P_i \defeq \{x-x': x,x'\in P_i\}$.
We use Weyl differencing $K=K(I)$ times (to ``linearize'' the monomial $\prod_{i\in I} m_i^{\mathfrak{u}_i}$) to get
\begin{equation*}
{\textstyle \abs{\EE_{\bm{m}\in \prod_{i\in I} P_i}\, \phase(y(\prod_{i\in I} m_i^{\mathfrak{u}_i})/q)}^{2^K}}
\ll \EE_{\bm{h}\in \prod_{i\in I} (P_i-P_i)^{\mathfrak{u}_i-1}}\,
\abs{\EE_{\bm{m}\in \prod_{i\in I} P_i}\, \phase(yM(\bm{h},\bm{m})/q) \bm{1}_{m_i\in Q_i(\bm{h})}},
\end{equation*}
where $\EE_{x\in A}$ denotes an average over $x\in A$,
where $Q_i(\bm{h})$ is a real interval defined in terms of $P_i$ and $\bm{h}$,
and where $M(\bm{h},\bm{m})\defeq \prod_{i\in I} (\mathfrak{u}_i!\, m_i \prod_{1\le j\le \mathfrak{u}_i-1} h_{i,j})$.
Fix an $i_\star\in I$, sum over $m_{i_\star}\in P_{i_\star} \cap Q_{i_\star}(\bm{h})$, and use the divisor bound as in \cite{parsell2012hua}*{proof of Lemma~2.2}, to get
\begin{equation*}
{\textstyle \abs{\EE_{\bm{m}\in \prod_{i\in I} P_i}\, \phase(y(\prod_{i\in I} m_i^{\mathfrak{u}_i})/q)}^{2^K}}
\ll_\eps \sum_{i\in I} \frac{1}{\card{P_i}}
+ \frac{(M')^\eps}{\prod_{i\in I} \card{P_i}^{\mathfrak{u}_i}} \sum_{1\le h\le M'} \min(\card{P_{i_\star}}, \norm{yhr/q}_{\RR/\ZZ}^{-1})
\end{equation*}
for some $r\in \ZZ\cap [1, R \prod_{i\in I} \mathfrak{u}_i!]$,
where $M' \defeq M_{i_\star}^{-1}\prod_{i\in I}(2M_i)^{\mathfrak{u}_i}$
and $\norm{\vartheta}_{\RR/\ZZ}\defeq
\min_{n\in \ZZ}{\abs{\vartheta-n}}$.
By \cite{vaughan1997hardy}*{Lemma~2.2}, the sum over $h$ is
$$\ll_\eps M' \card{P_{i_\star}}
\left(\frac{r}{q}
+ \frac{1}{\card{P_{i_\star}}}
+ \frac{q}{M' \card{P_{i_\star}}}\right)
(M' q)^\eps.$$
Multiplying by $\prod_{i\in I} \card{P_i}^{2^K}$, we now find that $\abs{\sum_{\bm{m} \in P_1 \times \dots \times P_\kspecial}
\phase(ym_1^{\mathfrak{u}_1}\cdots m_\kspecial^{\mathfrak{u}_\kspecial}/q)}^{2^K}$ is
\begin{equation*}
\ll_\eps \sum_{i\in I} \frac{(M_1\cdots M_\kspecial)^{2^K}}{M_i}
+ (M' q)^\eps \left(\frac{M'M_{i_\star}R}{q} + M' + q\right)
\prod_{i\in I} M_i^{2^K - \mathfrak{u}_i},
\end{equation*}
since $2^K\ge \mathfrak{u}_i$.
This implies the result up to $q^\eps$,
which suffices,
since we may assume $\prod_{i\in I} M_i^{\mathfrak{u}_i} \ge q$.
(The result is trivial if $\prod_{i\in I} M_i^{\mathfrak{u}_i} \le q$.)

\emph{Case~2: The general case.}
Suppose we freeze $(m_j)_{j\notin I}$
and sum over the remaining variables.
By Case~1,
with $y'/q' = y(\prod_{j\notin I} m_j^{\mathfrak{u}_j})/q$ in place of $y/q$,
we find that
\begin{equation*}
\sum_{(m_i)_{i\in I} \in \prod_{i\in I} P_i}
\phase(ym_1^{\mathfrak{u}_1}\cdots m_\kspecial^{\mathfrak{u}_\kspecial}/q)
\ll_\eps \frac{R^{O(1)} \prod_{i\in I} M_i^{1+\eps}}
{\min(q',
\min_{i\in I}(M_i), (q')^{-1} \prod_{i\in I} M_i^{\mathfrak{u}_i})^{1/2^{K(I)}}}
\end{equation*}
However, since $q\ge q'\gg q \prod_{j\notin I} M_j^{-\mathfrak{u}_j}$, we have
$${\textstyle \min(q',
\min_{i\in I}(M_i), (q')^{-1} \prod_{i\in I} M_i^{\mathfrak{u}_i})
\gg \min(q \prod_{j\notin I} M_j^{-\mathfrak{u}_j},
\min_{i\in I}(M_i), q^{-1} \prod_{i\in I} M_i^{\mathfrak{u}_i})}.$$
Combining the last two displays
and summing over $(m_j)_{j\notin I}$, we conclude that
\begin{equation*}
\sum_{(m_1,\dots,m_\kspecial) \in P_1 \times \dots \times P_\kspecial}
\phase(ym_1^{\mathfrak{u}_1}\cdots m_\kspecial^{\mathfrak{u}_\kspecial}/q)
\ll_\eps \frac{\card{\prod_{j\notin I} P_j}\,
R^{O(1)} \prod_{i\in I} M_i^{1+\eps}}
{\min(q \prod_{j\notin I} M_j^{-\mathfrak{u}_j},
\min_{i\in I}(M_i), q^{-1} \prod_{i\in I} M_i^{\mathfrak{u}_i})}.
\end{equation*}
This suffices, since $\card{P_j}\ll M_j$ for all $j\notin I$.
\end{proof}

\begin{corollary}
\label{COR:clean-Weyl-inequality}
In the setting of Proposition~\ref{PROP:Weyl-inequality-for-monomials-in-several-variables}, we have (for some constant $\eta>0$)
\begin{equation*}
\sum_{(m_1,\dots,m_\kspecial) \in P_1 \times \dots \times P_\kspecial}
\phase(ym_1^{\mathfrak{u}_1}\cdots m_\kspecial^{\mathfrak{u}_\kspecial}/q)
\ll_\eps \frac{R^{O(1)} (M_1\cdots M_\kspecial)^{1+\eps}}
{\min(q, M_1^{\mathfrak{u}_1} \cdots M_\kspecial^{\mathfrak{u}_\kspecial} / q)^\eta}.
\end{equation*}
\end{corollary}

\begin{proof}
If $\prod_{1\le i\le \kspecial} M_i^{\mathfrak{u}_i} \le q$, the desired bound is trivial.
So suppose $\prod_{1\le i\le \kspecial} M_i^{\mathfrak{u}_i} \ge q$.
Let $\mathfrak{Y} = \min(q, M_1^{\mathfrak{u}_1} \cdots M_\kspecial^{\mathfrak{u}_\kspecial} / q) \ge 1$;
then $M_1^{\mathfrak{u}_1} \cdots M_\kspecial^{\mathfrak{u}_\kspecial} \ge \mathfrak{Y}^2$.
Let $$I = \set{1\le i\le \kspecial: M_i\ge \mathfrak{Y}^\varrho}$$
for some $\varrho>0$
small enough in terms of $\mathfrak{u}_1,\dots,\mathfrak{u}_\kspecial$.
Then $I \ne \emptyset$,
and $\min_{i\in I}(M_i) \ge \mathfrak{Y}^\varrho$.
Also, $${\textstyle \min(q \prod_{i\notin I} M_i^{-\mathfrak{u}_i},
q^{-1} \prod_{i\in I} M_i^{\mathfrak{u}_i})
\ge \mathfrak{Y} / \prod_{i\notin I} M_i^{\mathfrak{u}_i}}
\ge \mathfrak{Y}^\varrho,$$
say.
Consequently, by Proposition~\ref{PROP:Weyl-inequality-for-monomials-in-several-variables}, the exponent $\eta = \varrho/2^{K(\{1,\dots,\kspecial\})}$ is admissible.
\end{proof}

\begin{corollary}
\label{COR:clean-Weyl-inequality-with-gcd-conditions}
In the setting of Proposition~\ref{PROP:Weyl-inequality-for-monomials-in-several-variables}, we have (for some constant $\eta>0$)
\begin{equation*}
\sum_{\substack{(m_1,\dots,m_\kspecial) \in P_1 \times \dots \times P_\kspecial: \\ \gcd(m_i,m_j) = \gcd(m_j,Q) = 1}}
\phase(ym_1^{\mathfrak{u}_1}\cdots m_\kspecial^{\mathfrak{u}_\kspecial}/q)
\ll_\eps \frac{Q^\eps R^{O(1)} (M_1\cdots M_\kspecial)^{1+\eps}}
{\min(q, M_1^{\mathfrak{u}_1} \cdots M_\kspecial^{\mathfrak{u}_\kspecial} / q)^\eta}
\quad \textnormal{for all $Q\in \ZZ_{\ge 1}$}.
\end{equation*}
\end{corollary}

\begin{proof}
We sieve, i.e.~use the M\"{o}bius inversion formula
$\bm{1}_{\gcd(m,n)=1} = \sum_{h\mid m,n} \mu(h)$,
to rewrite the $\gcd$ conditions.
We separately analyze small and large values of the divisor $h$.
Let $B\in \RR_{>0}$.
If $h_{ij},g_j\in \ZZ\cap [1,B]$, then Corollary~\ref{COR:clean-Weyl-inequality} implies
\begin{equation*}
\sum_{(m_1,\dots,m_\kspecial) \in P_1 \times \dots \times P_\kspecial}
\bm{1}_{h_{ij}\mid m_i,m_j} \bm{1}_{g_j\mid m_j,Q}\,
\phase(ym_1^{\mathfrak{u}_1}\cdots m_\kspecial^{\mathfrak{u}_\kspecial}/q)
\ll_\eps \frac{(BR)^{O(1)} (M_1\cdots M_\kspecial)^{1+\eps}}
{\min(q, M_1^{\mathfrak{u}_1} \cdots M_\kspecial^{\mathfrak{u}_\kspecial} / q)^\eta}.
\end{equation*}
Multiplying by $\prod_{1\le i<j\le \kspecial} \mu(h_{ij})$ and $\prod_{1\le j\le \kspecial} \mu(g_j)$, and summing over $h_{ij},g_j\ge 1$, we get
\begin{equation*}
\sum_{\substack{(m_1,\dots,m_\kspecial) \in P_1 \times \dots \times P_\kspecial: \\ \gcd(m_i,m_j) = \gcd(m_j,Q) = 1}}
\phase(ym_1^{\mathfrak{u}_1}\cdots m_\kspecial^{\mathfrak{u}_\kspecial}/q)
\ll_\eps \frac{(BR)^{O(1)} (M_1\cdots M_\kspecial)^{1+\eps}}
{\min(q, M_1^{\mathfrak{u}_1} \cdots M_\kspecial^{\mathfrak{u}_\kspecial} / q)^\eta}
+ \sum_{\substack{m_1,\dots,m_\kspecial\ge 1: \\ M_j\le m_j<2M_j}} T_B(\bm{m}),
\end{equation*}
where $T_B(\bm{m}) \defeq (\sum_{1\le i<j\le \kspecial} \sum_{h>B} \bm{1}_{h\mid m_i,m_j} + \sum_{1\le j\le \kspecial} \sum_{g>B:\, g\mid Q} \bm{1}_{g\mid m_j}) \cdot (m_1\cdots m_\kspecial)^\eps$.
Here $$\sum_{\substack{m_1,\dots,m_\kspecial\ge 1: \\ M_j\le m_j<2M_j}} T_B(\bm{m})
\ll_\eps (M_1\cdots M_\kspecial)^{1+\eps} (1+Q^\eps) / B,$$
since $Q$ has $O_\eps(Q^\eps)$ divisors.
Now take $B$ to be a small power of $\min(q, M_1^{\mathfrak{u}_1} \cdots M_\kspecial^{\mathfrak{u}_\kspecial} / q)$.
\end{proof}

We can now handle $\mathcal{P}_2$,
using reciprocity ($\psi(\QQ)=1$)
and partial summation.

\begin{lemma}
\label{LEM:range-2}
Let $\xi>0$ be small.
Then the piece of \eqref{EXPR:final-expanded-adelic-data-sum} defined by $\mathcal{P}_2$
extends to an element of $\mathcal{H}_\star(\delta)$, provided $\delta$ is sufficiently small (in terms of $\xi$).
\end{lemma}



\begin{proof}
If $\kspecial\lspecial=0$,
then $m_0n_0 \max(\alpha,\alpha^{-1}) \ge \mathsf{P}(\alpha)$,
so the piece of \eqref{EXPR:final-expanded-adelic-data-sum}
satisfying \eqref{COND:define-challenging-Kloosterman-oscillation-Z_1-range}
is a finite sum (if $\xi$ is sufficiently small),
and thus lies in $\mathcal{H}_\star(\delta)$.
So we may assume $\kspecial,\lspecial\ge 1$.

If $\lspecial'\le 1$ then \eqref{COND:define-challenging-Kloosterman-oscillation-Z_1-range} is impossible, so assume $\lspecial'\ge 2$.
Fix $\Omega$, $f_{25}$, $\lambda$.
By a partition of unity on
$[n[1]:\dots:n[\lspecial']]\in \RR_{>0}^{\lspecial'}/\RR_{>0}$,
we may assume $n[1]\gg n[2]\gg \dots\gg n[\lspecial']$.
More precisely, the right-hand side of the equality
\begin{equation}
\label{ordering-partition-of-unity}
1 = \prod_{1\le i<j\le \lspecial'} (w_0(n[i]/n[j]) + w_\infty(n[i]/n[j]))
\end{equation}
expands into $2^{\binom{\lspecial'}{2}}$ terms,
each of which approximately imposes an ordering of $n[1],\dots,n[\lspecial']$.
Let $\Cfinal_0 = \cfinal_0-\cfinal_2$
and $\Cfinal_\jspecial = \cfinal_\jspecial-\cfinal_2$,
where $\cfinal_0=\cfinal_0(\Omega)$ is defined in the paragraph containing \eqref{decompose_Hveeinfty}.
For any $r,m_0,\dots$ satisfying \eqref{COND:define-challenging-Kloosterman-oscillation-Z_1-range}, we have $n[1] \gg n[2] \gg \mathsf{P}(\alpha)^\Xi$.

On the range $n[1]\gg n[2]\gg \dots\gg n[\lspecial']$,
the piece of \eqref{EXPR:final-expanded-adelic-data-sum} we are interested in is (for some weight $\nu=\nu(n[1],\dots,n[\lspecial'])\in C^\infty(\RR_{>0}^{\lspecial'}/\RR_{>0})$ supported on $n[1]\gg n[2]\gg \dots\gg n[\lspecial']$)
\begin{equation*}
\mathcal{Q}_3\defeq
\sum_{r,m_0,\dots} \nu \cdot \mathcal{P}_2 \cdot \mathscr{S}{\left(
\frac{(\mathfrak{N}_{m_0} \mathfrak{D}_{n_0} \mathfrak{G}_r f_{25})(\bm{s},\lambda,\alpha)
\phase(\Cfinal_0\alpha) \phase(\cfinal_2\alpha\bmod{\ZZ_{c'_2n_0}})}{m_1^{\beta_1}\cdots m_\kspecial^{\beta_\kspecial}}
\prod_{1\le \jspecial\le \lspecial} \frac{\phase(-\Cfinal_\jspecial\alpha\bmod{\ZZ_{n_\jspecial}})}{n_\jspecial^{\gamma_\jspecial}} \right)};
\end{equation*}
here we have rewritten the exponentials in \eqref{EXPR:final-expanded-adelic-data-sum} using $\psi(\cfinal_2\alpha) = 1$ (reciprocity) in the form $\phase(\cfinal_2\alpha) = \phase(\cfinal_2\alpha\bmod{\ZZ_{c'_2n_0}}) \prod_{1\le \jspecial\le \lspecial} \phase(\cfinal_2\alpha\bmod{\ZZ_{n_\jspecial}})$,
where $c'_2\ge 1$ is the denominator of $\cfinal_2$.
(Note that $c'_2n_0,n_1,\dots,n_\lspecial$ are pairwise coprime, by Definition~\ref{DEFN:bad-set-S}.)

Let $\bm{s}$ lie in \eqref{INEQ:key-new-region}.
We seek to obtain \emph{nontrivial cancellation} in $\mathcal{Q}_3$ over the variables
$(m_\ispecial)_{1\le \ispecial\le \kspecial}$,
$q_1 \defeq n[1]n[3]\cdots n[\lspecial']$,
and $q_2 \defeq n[2]$.
We have $\alpha = (m_0m_1^{\mathfrak{u}_1}\cdots m_\kspecial^{\mathfrak{u}_\kspecial})/(n_0q_1q_2)$
by \eqref{alpha-factorize}.
Since $\Cfinal_2=0$ and $\Cfinal_1\Cfinal_3\cdots \Cfinal_{\lspecial'} \ne 0$,
there exist coprime integers $\mathfrak{a},q\ge 1$ such that
\begin{equation*}
\phase{\left(\frac{\cfinal_2m_0}{n_0q_1q_2}\bmod{\ZZ_{c'_2n_0}}\right)}\,
\prod_{1\le \jspecial\le \lspecial}
\phase{\left(-\frac{\Cfinal_\jspecial m_0}{n_0q_1q_2}\bmod{\ZZ_{n_\jspecial}}\right)}
= \phase(\mathfrak{a}/q),
\end{equation*}
where $q\mid n_0q_1$
and $q\gg q_1$.
We now want to apply Corollary~\ref{COR:clean-Weyl-inequality-with-gcd-conditions}
with this value of $q$,
and with $Q=rm_0n_0q_1q_2$,
but some explanation is needed.
First, we split $(m_\ispecial)_{1\le \ispecial\le \kspecial}$
into residue classes modulo $R\ll m_0n_0r$
on which $\mathfrak{N}_{m_0} \mathfrak{D}_{n_0} \mathfrak{G}_r$ is constant;
this is possible by Lemma~\ref{LEM:local-constancy}.
There are $R^\kspecial$ residue classes in total.
Second, we break $(m_\ispecial)_{1\le \ispecial\le \kspecial}$
into dyadic ranges $M_\ispecial\le m_\ispecial<2M_\ispecial$,
where $M_\ispecial\in \set{2^e: e\in \ZZ_{\ge 0}}$.
Third, we use partial summation
in each variable $m_\ispecial\in [M_\ispecial,2M_\ispecial)$,
in order to remove the archimedean weight
$f_{25}(\bm{s},\lambda,\alpha)/(m_1^{\beta_1}\cdots m_\kspecial^{\beta_\kspecial})$
when applying Corollary~\ref{COR:clean-Weyl-inequality-with-gcd-conditions};
this is possible by the $k_2\in \{0,\dots,\kspecial\}$ case
of Lemma~\ref{LEM:uniform-derivative-bound-on-localized-modulated-integrals},
which controls the derivatives of
$f_{25}\in \bigcup_\Omega \set{\mathcal{J}_{\Omega,0}, \mathcal{J}_{\Omega,\infty}}$
with respect to $\alpha$.
In the end, we find that
the quantity $\mathcal{Q}_3$
is\footnote{and, more precisely, may be rewritten as a sum of holomorphic functions whose sum of absolute values is}
at most $1+\norm{\bm{s}}^{O(\kspecial)}$ times
\begin{equation}
\label{EXPR:main-factor-of-near-final-reciprocity-plus-Weyl-cancellation-bound}
\mathcal{Q}_4\defeq
\sum_{r,m_0,\ldots:
\, \eqref{COND:define-challenging-Kloosterman-oscillation-Z_1-range}}\,
R^\kspecial
\frac{\abs{\mathfrak{N}_{m_0}\mathfrak{D}_{n_0}\mathfrak{G}_r}
\, (n_0/m_0)^2 (1+\abs{\alpha})^\kspecial}
{(m_1\cdots m_\kspecial n_1\cdots n_\lspecial)^{1-\delta}}
\frac{O_\eps(Q^\eps R^{O(1)}(m_1\cdots m_\kspecial)^\eps)}
{\min(q,
m_1^{\mathfrak{u}_1}\cdots m_\kspecial^{\mathfrak{u}_\kspecial}/q)^\eta},
\end{equation}
where $q_1\ll q\le n_0q_1$, $Q=rm_0n_0q_1q_2$, and $R\ll m_0n_0r$.
Since
\begin{equation*}
rm_0n_0 \max(\alpha, \alpha^{-1}) \ll \mathsf{P}(\alpha)^\xi,
\quad
\min(q_1,q_2) \gg \mathsf{P}(\alpha)^\Xi,
\quad \frac{m_1^{\mathfrak{u}_1}\cdots m_\kspecial^{\mathfrak{u}_\kspecial}}{n_0q_1}
= \frac{\alpha q_2}{m_0},
\end{equation*}
a crude application of Lemmas~\ref{LEM:denominator-bias}, \ref{LEM:numerator-bias}, and~\ref{LEM:generic-local-factor-estimate}
(e.g.~$\mathfrak{N}_{m_0}\mathfrak{D}_{n_0}\mathfrak{G}_r \ll (m_0n_0r)^{O(1)}$)
gives
\begin{equation*}
\mathcal{Q}_4
\ll_{\xi,\delta} \sum_{m,n\ge 1}
\frac{(mn)^{O(\xi) - \Xi\eta}}{(mn)^{1-\delta}}
\ll \sum_{m,n\ge 1} \frac{(mn)^{-\Xi\eta/2}}{(mn)^{1-\delta}}
\ll_{\xi,\delta} 1
\end{equation*}
(if we take $\eps=\xi$ and choose $\xi \ll 1$ and $\delta \ll_\xi 1$ appropriately),
since $\xi = \Xi^2 = o_{\xi\to\infty}(\Xi)$.
\end{proof}

Finally, we rewrite the $\mathcal{P}_1$ piece in a form convenient for \S\ref{SEC:final-reductions},
via an equidistribution estimate:

\begin{proposition}
\label{PROP:Poisson-sum-based-equidistribution-estimate}
Let $y,q,Q,M,A\in \ZZ_{\ge 1}$ and $Y\in \RR_{>0}$.
Let $w\in C^\infty(\RR)$ be supported on $[1,10]$,
and suppose $\sup{\abs{w^{(B)}}} \ll_B A^B Y$ for all $B\ge 0$,
where $w^{(B)}$ is the $B$th derivative of $w$.
Let $\EE_{n\in \ZZ_q:\, \gcd(n,Q) = 1}[\ast]$
be the average of $\ast$ over the set $\set{n\in \ZZ_q: \gcd(n,Q) = 1}$.
Then
\begin{equation}
\label{q-goal}
\sum_{m\in \ZZ:\, \gcd(m,Q)=1} w(\tfrac{m}{M})\, (\bm{1}_{m\equiv y\bmod{q}}
- \EE_{n\in \ZZ_q:\, \gcd(n,Q) = 1}[\bm{1}_{n\equiv y\bmod{q}}])
\ll_\eps (QM)^\eps AYq.
\end{equation}
\end{proposition}

\begin{proof}
By replacing $w$ with $w/Y$, assume $Y=1$.
Let $f_{27}(x')\defeq w(x'/M)$
and let $$\hat{f}_{27}(x)
\defeq \int_\RR f_{27}(x')\phase(-xx')\,dx'
\ll_k M (1+\abs{xM/A})^{-k},$$
where the bound on the right-hand side follows from
repeated integration by parts over $x'$ if $\abs{xM/A}\ge 1$,
or from the triangle inequality if $\abs{xM/A}\le 1$.
Given $h\in \ZZ_{\ge 1}$,
let $$f_{28}(m)\defeq \bm{1}_{h\mid m} \bm{1}_{m\equiv y\bmod{q}}
- \EE_{n\in \ZZ_{hq}}[\bm{1}_{h\mid n} \bm{1}_{n\equiv y\bmod{q}}]$$
and let $\hat{f}_{28}(x)
\defeq \EE_{m\in \ZZ_{hq}}[f_{28}(m) \phase(xm\bmod{\ZZ_{hq}})]$,
where $\EE_{m\in \ZZ_{hq}}[\ast]$ is the average of $\ast$ over $\ZZ_{hq}$.
By the Poisson summation formula in $\RR \times (\ZZ/hq\ZZ)$,\footnote{By this,
we mean the adelic Poisson formula
$\sum_{m\in \QQ} \phi(m) = \sum_{x\in \QQ} \hat{\phi}(x)$,
with a test function $\phi\maps \adele_\QQ \to \CC$
supported on $\adele_\ZZ = \RR \times \prod_p \ZZ_p$,
such that $\phi\vert_{\adele_\ZZ}$ factors through the reduction map
$\adele_\ZZ \to \RR \times (\ZZ/hq\ZZ)$.}
or by splitting $\sum_{m\in \ZZ}$ into residue classes modulo $hq$
and applying \cite{iwaniec2004analytic}*{(4.24), with $v=hq$} in each class,
we have
\begin{equation*}
\sum_{m\in \ZZ} f_{27}(m) f_{28}(m)
= \sum_{m'\in \ZZ} \hat{f}_{27}(\tfrac{m'}{hq}) \hat{f}_{28}(\tfrac{m'}{hq}).
\end{equation*}
Clearly $\hat{f}_{28}(0) = 0$
and $\sup\abs{\hat{f}_{28}} \le \sup\abs{f_{28}} \le 1$.
If $h\le \frac{M^{1-\eps}}{Aq}$,
then $\hat{f}_{27}(\frac{m'}{hq}) \ll_k \frac{M}{\abs{m'M^\eps}^k}$, so
\begin{equation}
\label{rapid-decay}
\sum_{m\in \ZZ} w(\tfrac{m}{M})\, (\bm{1}_{h\mid m} \bm{1}_{m\equiv y\bmod{q}}
- \EE_{n\in \ZZ_{hq}}[\bm{1}_{h\mid n} \bm{1}_{n\equiv y\bmod{q}}])
= \sum_{m\in \ZZ} f_{27}(m) f_{28}(m)
\ll_{\eps,k} M^{-k}.
\end{equation}

We have $\bm{1}_{\gcd(m,Q)=1} = \sum_{h\mid m,Q} \mu(h)
= \sum_{h\mid Q} \mu(h) \bm{1}_{h\mid m}$;
cf.~the sieving in Corollary~\ref{COR:clean-Weyl-inequality-with-gcd-conditions}.
Multiplying \eqref{rapid-decay} by $\mu(h)$,
summing over $h\mid Q$,
and writing $$\EE_{n\in \ZZ_{hq}}[\bm{1}_{h\mid n} \bm{1}_{n\equiv y\bmod{q}}]
= \EE_{n\in \ZZ_{Qq}}[\bm{1}_{h\mid n} \bm{1}_{n\equiv y\bmod{q}}]$$
for each $h\mid Q$,
we find that
\begin{equation*}
\sum_{m\in \ZZ} w(\tfrac{m}{M})\, (\bm{1}_{\gcd(m,Q)=1} \bm{1}_{m\equiv y\bmod{q}}
- \EE_{n\in \ZZ_{Qq}}[\bm{1}_{\gcd(n,Q)=1} \bm{1}_{n\equiv y\bmod{q}}])
\ll_{\eps,k} \mathcal{Q}_5
+ \sum_{h\le \frac{M^{1-\eps}}{Aq}} M^{-k},
\end{equation*}
where
\begin{equation*}
\begin{split}
\mathcal{Q}_5 &\defeq
\sum_{m\in \ZZ} w(\tfrac{m}{M})
\sum_{h>\frac{M^{1-\eps}}{Aq}:\, h\mid Q}
\abs{\bm{1}_{h\mid m} \bm{1}_{m\equiv y\bmod{q}}
- \EE_{n\in \ZZ_{Qq}}[\bm{1}_{h\mid n} \bm{1}_{n\equiv y\bmod{q}}]} \\
&\ll \sum_{M\le m\le 10M}
\sum_{h>\frac{M^{1-\eps}}{Aq}:\, h\mid Q}
(\bm{1}_{h\mid m} + \tfrac{1}{h})
\ll \sum_{h>\frac{M^{1-\eps}}{Aq}:\, h\mid Q} \tfrac{M}{h}
\ll_\eps Q^\eps \tfrac{M}{M^{1-\eps}/(Aq)}.
\end{split}
\end{equation*}
Thus
\begin{equation}
\label{q-blah}
\sum_{m\in \ZZ} w(\tfrac{m}{M})\, (\bm{1}_{\gcd(m,Q)=1} \bm{1}_{m\equiv y\bmod{q}}
- \EE_{n\in \ZZ_{Qq}}[\bm{1}_{\gcd(n,Q)=1} \bm{1}_{n\equiv y\bmod{q}}])
\ll_\eps (QM)^\eps Aq.
\end{equation}

On the other hand, the $q=1$ case of \eqref{q-blah} gives
\begin{equation}
\label{q=1}
\sum_{m\in \ZZ} w(\tfrac{m}{M})\, (\bm{1}_{\gcd(m,Q)=1}
- \EE_{n\in \ZZ_Q}[\bm{1}_{\gcd(n,Q)=1}])
\ll_\eps (QM)^\eps A.
\end{equation}
Multiplying \eqref{q=1}
by $\EE_{n\in \ZZ_q:\, \gcd(n,Q) = 1}[\bm{1}_{n\equiv y\bmod{q}}]$,
then subtracting \eqref{q-blah},
gives \eqref{q-goal}, since
$\EE_{n\in \ZZ_{Qq}}[\bm{1}_{\gcd(n,Q)=1} \bm{1}_{n\equiv y\bmod{q}}]
= \EE_{n\in \ZZ_Q}[\bm{1}_{\gcd(n,Q)=1}]
\cdot \EE_{n\in \ZZ_q:\, \gcd(n,Q) = 1}[\bm{1}_{n\equiv y\bmod{q}}]$.
\end{proof}

\begin{lemma}
\label{LEM:range-1}
Let $\xi>0$ be small.
For each $\jspecial'\in \{1,\dots,\lspecial'\}$, let
\begin{equation}
\label{define-restricted-alpha-product}
\alpha_{\jspecial'}
\defeq \frac{m_1^{\mathfrak{u}_1}\cdots m_\kspecial^{\mathfrak{u}_\kspecial}}
{\prod_{\jspecial\ge 1:\, \cfinal_\jspecial=\cfinal_{\jspecial'}}
n_\jspecial^{\mathfrak{v}_\jspecial}}.
\end{equation}
If $\delta$ is small enough in terms of $\xi$,
then the piece of \eqref{EXPR:final-expanded-adelic-data-sum} defined by $\mathcal{P}_1$
equals (for $\Re(\bm{s})$ large)
\begin{equation*}
h_0(\bm{s})
+ \sum_{1\le \jspecial'\le \lspecial'}
\, \sum_{m_1,\dots,m_\kspecial\ge 1}\,
\sum_{n_\jspecial\ge 1:\, \jspecial\ge 1,\; \cfinal_\jspecial=\cfinal_{\jspecial'}}\,
\frac{h_{\jspecial'}(\bm{s},m_1,\dots,m_\kspecial,(n_\jspecial)_{\jspecial\ge 1:\, \cfinal_\jspecial=\cfinal_{\jspecial'}})}
{\alpha_{\jspecial'}^2 (\alpha_{\jspecial'}^\xi+\alpha_{\jspecial'}^{-\xi}) \cdot
m_1^{\beta_1}\cdots m_\kspecial^{\beta_\kspecial} \prod_{\jspecial\ge 1:\, \cfinal_\jspecial=\cfinal_{\jspecial'}} n_\jspecial^{\gamma_\jspecial}}
\end{equation*}
for some functions $h_0,\dots,h_{\lspecial'}$ such that \emph{uniformly} over $m_1,\dots,m_\kspecial,n_1,\dots,n_\lspecial\ge 1$, we have $h_0,h_{\jspecial'}\in \mathcal{H}_\star(\delta)$ and $h_0,h_{\jspecial'} \ll_K (1+\norm{\bm{s}})^{O(1)}$ in vertical strips $\Re(\bm{s})\in K$ (for compact $K$).
\end{lemma}


\begin{proof}
As in the proof of Lemma~\ref{LEM:range-2}, we may assume $\kspecial,\lspecial\ge 1$.
Fix $\Omega$, $f_{25}$, $\lambda$.
Assume $n[1]\ll n[2]\ll \dots\ll n[\lspecial']$
and $m_1\ll m_2\ll \dots\ll m_\kspecial$,
by introducing a suitable factor
$\nu=\nu(n[1],\dots,n[\lspecial'],m_1,\dots,m_\kspecial)
\in C^\infty(\RR_{>0}^{\lspecial'+\kspecial}/\RR_{>0})$,
belonging to a smooth partition of unity on $\RR_{>0}^{\lspecial'+\kspecial}/\RR_{>0}$
(cf.~\eqref{ordering-partition-of-unity} from before).
Let $\Cfinal_0 = \cfinal_0-\cfinal_{\lspecial'}$
and $\Cfinal_\jspecial = \cfinal_\jspecial-\cfinal_{\lspecial'}$.
For any $r,m_0,\dots$ satisfying \eqref{COND:define-dominant-Z_1-range},
we have $n[\jspecial'] \ll \mathsf{P}(\alpha)^\Xi$ for all $1\le \jspecial'<\lspecial'$.

Let $c'_{\lspecial'} \in \ZZ_{\ge 1}$ be the denominator of $\cfinal_{\lspecial'}$.
The piece of \eqref{EXPR:final-expanded-adelic-data-sum} of interest is
$\mathscr{S}(F_{29})(\bm{s})$, where
\begin{equation*}
F_{29}(\bm{s}) = \sum_{r,m_0,\dots} \nu \cdot \mathcal{P}_1 \cdot
\frac{(\mathfrak{N}_{m_0} \mathfrak{D}_{n_0} \mathfrak{G}_r f_{25})(\bm{s},\lambda,\alpha)
\phase(\Cfinal_0\alpha) \phase(\cfinal_{\lspecial'}\alpha\bmod{\ZZ_{c'_{\lspecial'}n_0}})}{m_1^{\beta_1}\cdots m_\kspecial^{\beta_\kspecial}}
\prod_{1\le \jspecial\le \lspecial} \frac{\phase(-\Cfinal_\jspecial\alpha\bmod{\ZZ_{n_\jspecial}})}{n_\jspecial^{\gamma_\jspecial}}.
\end{equation*}
Write $\bm{m}' = (m_0,\dots,m_{\kspecial-1})$ and $\bm{n} = (n_0,\dots,n_\lspecial)$.
Let
\begin{equation*}
\begin{split}
f_{30}(\bm{s},\lambda,r,\bm{m}',\bm{n})
&\defeq \EE[\phase(\cfinal_{\lspecial'}\alpha\bmod{\ZZ_{c'_{\lspecial'}n_0}})
\mathfrak{N}_{m_0}\mathfrak{D}_{n_0}\mathfrak{G}_r], \\
f_{31,\jspecial'}(r,\bm{m}',\bm{n})
&\defeq \EE[\phase(-\Cfinal_{\jspecial'}\alpha\bmod{\ZZ_{n[\jspecial']}})],
\end{split}
\end{equation*}
where for a locally constant function $\ast$ of $\alpha$ on $\QQ_q$
we let $\EE[\ast]$ denote the average of $\ast$
over the set $\set{m_\kspecial\in \ZZ_q:
\gcd(m_\kspecial, rm_0\cdots m_{\kspecial-1}n_0\cdots n_\lspecial) = 1}$.
Let
\begin{equation*}
F_{32}(\bm{s}) = \sum_{r,m_0,\dots} \nu \cdot \mathcal{P}_1 \cdot
\frac{f_{25}(\bm{s},\lambda,\alpha)
\phase(\Cfinal_0\alpha) f_{30}(\bm{s},\lambda,r,\bm{m}',\bm{n})}{m_1^{\beta_1}\cdots m_\kspecial^{\beta_\kspecial}}
\frac{\prod_{1\le \jspecial'<\lspecial'} f_{31,\jspecial'}(r,\bm{m}',\bm{n})}
{n_1^{\gamma_1}\cdots n_\lspecial^{\gamma_\lspecial}}.
\end{equation*}
Let $\bm{z}\defeq \bm{s}-\mathbf{d}-2\mathbf{u}$, and assume $\bm{s}$ lies in \eqref{INEQ:key-new-region} with $\delta = \xi/2$.

For $t\in \RR$,
we now aim to bound the difference $(F_{29}-F_{32})(\bm{s}-it\mathbf{u})$
by summing over $m_\kspecial$,
using Proposition~\ref{PROP:Poisson-sum-based-equidistribution-estimate}
with suitable parameters $M$, $A$, $Y$, $q$, $Q$.
First we show that $m_\kspecial$ is large.
The weight $\nu$ in $F_{29}(\bm{s})$ forces
$m_\kspecial \gg m_1,m_2,\dots,m_\kspecial$,
and the weight $\mathcal{P}_1$ forces $m_0,n_0,\alpha^{-1}\ll \mathsf{P}(\alpha)^\xi$
by \eqref{COND:define-dominant-Z_1-range}.
Yet by \eqref{alpha-factorize}
and the inequality $\mathfrak{v}_1,\dots,\mathfrak{v}_\lspecial\ge 1$, we have
$$\alpha n_0/m_0
= \frac{m_1^{\mathfrak{u}_1}\cdots m_\kspecial^{\mathfrak{u}_\kspecial}}
{n_1^{\mathfrak{v}_1}\cdots n_\lspecial^{\mathfrak{v}_\lspecial}}
\le \frac{m_1^{\mathfrak{u}_1}\cdots m_\kspecial^{\mathfrak{u}_\kspecial}}
{n_1\cdots n_\lspecial}
= \frac{m_1^{1+\mathfrak{u}_1}\cdots m_\kspecial^{1+\mathfrak{u}_\kspecial}}
{\mathsf{P}(\alpha)}.$$
Letting $A_{10}\defeq (1+\mathfrak{u}_1)+\dots+(1+\mathfrak{u}_\kspecial)
\ge 2\kspecial > 0$, it follows that
\begin{equation}
\label{m-large-bound}
m_\kspecial
\gg (m_1^{1+\mathfrak{u}_1}\cdots m_\kspecial^{1+\mathfrak{u}_\kspecial})^{1/A_{10}}
= (\mathsf{P}(\alpha) \alpha n_0/m_0)^{1/A_{10}}
\gg \mathsf{P}(\alpha)^{(1-2\xi)/A_{10}},
\end{equation}
since $n_0\ge 1$ and $\alpha,m_0^{-1}\gg \mathsf{P}(\alpha)^{-\xi}$.
Second, we show that $\sum_{m_\kspecial}$ is smoothly weighted.
By the chain rule,
the $k_1=2$ case of Lemma~\ref{LEM:uniform-derivative-bound-on-localized-modulated-integrals},
the corollary $m_\kspecial\,\tfrac{\partial\alpha}{\partial{m_\kspecial}}
= \mathfrak{u}_\kspecial\alpha$ to \eqref{alpha-factorize},
and the corollary $(\alpha\,\frac{\partial}{\partial\alpha})^{k_2} \phase(\Cfinal_0\alpha) \ll_{k_2} (1+\alpha)^{k_2} \ll_{k_2} \mathsf{P}(\alpha)^{\xi k_2}$ to \eqref{COND:define-dominant-Z_1-range},
we have
\begin{equation}
\label{m-derivative-bound}
(m_\kspecial\,\tfrac{\partial}{\partial{m_\kspecial}})^{k_2}\,
\frac{\nu \cdot \mathcal{P}_1
\cdot f_{25}(\bm{s}-it\mathbf{u},\lambda,\alpha) \phase(\Cfinal_0\alpha)}
{m_\kspecial^{\beta_\kspecial+it\mathfrak{u}_\kspecial}}
\ll_{k_2} \frac{(1+\norm{\bm{s}})^{O(1+k_2)} \mathsf{P}(\alpha)^{\xi k_2}}
{\alpha^2 (1+t^2)}.
\end{equation}
Third, Lemma~\ref{LEM:local-constancy} lets us break $\sum_{m_k}$
into $q$ residue classes $\mathcal{R}$ modulo $q$
on which the quantities
$\phase(\cfinal_{\lspecial'}\alpha\bmod{\ZZ_{c'_{\lspecial'}n_0}})$,
$\mathfrak{N}_{m_0}$, $\mathfrak{D}_{n_0}$, $\mathfrak{G}_r$,
$\phase(-\Cfinal_{\jspecial'}\alpha\bmod{\ZZ_{n[\jspecial']}})$
are constant, where
\begin{equation*}
q\ll rm_0c'_{\lspecial'}n_0\prod_{1\le \jspecial'<\lspecial'} n[\jspecial']
\ll \mathsf{P}(\alpha)^{3\xi+\lspecial'\Xi}
\end{equation*}
by \eqref{COND:define-dominant-Z_1-range}.
Fourth, we have a coprimality condition $\gcd(m_k,Q) = 1$, where
\begin{equation*}
Q = rm_0\cdots m_{\kspecial-1}n_0\cdots n_\lspecial
\le rm_0n_0 \mathsf{P}(\alpha)
\ll \mathsf{P}(\alpha)^{1+3\xi}
\end{equation*}
by \eqref{COND:define-dominant-Z_1-range}.
Finally, if in $F_{29}(\bm{s})$,
we make a smooth dyadic partition of unity on the variable $m_\kspecial$,
then Proposition~\ref{PROP:Poisson-sum-based-equidistribution-estimate}
with $M\gg \mathsf{P}(\alpha)^{(1-2\xi)/A_{10}}$ (in view of \eqref{m-large-bound}),
with $A \asymp (1+\norm{\bm{s}})^{O(1)} \mathsf{P}(\alpha)^{\xi}$
and $Y \asymp (1+\norm{\bm{s}})^{O(1)} \alpha^{-2}(1+t^2)^{-1}$
(in view of \eqref{m-derivative-bound}),
and with $q$ and $Q$ as displayed above,
reveals that
\begin{equation*}
(F_{29}-F_{32})(\bm{s}-it\mathbf{u})
\ll_\eps \frac{(1+\norm{\bm{s}})^{O(1)}}{1+t^2}
\sum_{\substack{r,m_0,\ldots:
\, \eqref{COND:define-dominant-Z_1-range}, \\
m_\kspecial\gg \mathsf{P}(\alpha)^{(1-2\xi)/A_{10}}}}\,
\frac{\abs{\mathfrak{N}_{m_0}\mathfrak{D}_{n_0}\mathfrak{G}_r}
\, (n_0/m_0)^2}
{(m_1\cdots m_\kspecial n_1\cdots n_\lspecial)^{1-\delta}}
\frac{q^2 (Qm_\kspecial)^\eps A}{m_\kspecial}.
\end{equation*}
Key here is the rightmost factor of $m_\kspecial$ in the denominator,
coming from cancellation in dyadic ranges $m_\kspecial\asymp M$.
At this point,
writing $m_\kspecial\gg \mathsf{P}(\alpha)^{(1-2\xi)/A_{10}}
= (m_1\cdots m_\kspecial n_1\cdots n_\lspecial)^{(1-2\xi)/A_{10}}$ in the denominator,
and using the crude bound $\mathfrak{N}_{m_0}\mathfrak{D}_{n_0}\mathfrak{G}_r
\ll (m_0n_0r)^{O(1)}
\ll \mathsf{P}(\alpha)^{O(\xi)}$
as we did for Lemma~\ref{LEM:range-2} after \eqref{EXPR:main-factor-of-near-final-reciprocity-plus-Weyl-cancellation-bound},
we find that if $\xi$ is small enough in terms of $A_{10}$, then
\begin{equation*}
(F_{29}-F_{32})(\bm{s}-it\mathbf{u})
\ll \frac{(1+\norm{\bm{s}})^{O(1)}}{1+t^2}
\sum_{m,n\ge 1} \frac{1}{(mn)^{1+1/(2A_{10})}}
\ll \frac{(1+\norm{\bm{s}})^{O(1)}}{1+t^2}.
\end{equation*}
Thus, by definition \eqref{INEQ:growth-decay-condition-for-f-in-H_dagger,J},
\begin{equation}
\label{EQN:F39-F42-convergence-bound}
(F_{29}-F_{32})(\bm{s})
= (F_{29}-F_{32})(\bm{z}+\mathbf{d}+2\mathbf{u})
\in \mathcal{H}_{\dagger,J}(-\delta,\infty).
\end{equation}
Therefore, $\mathscr{S}(F_{29}-F_{32})(\bm{s})\in \mathcal{H}_\star(\delta)$
by Lemma~\ref{LEM:shifted-integral-polar-and-residue-structure}(4)
and the definition \eqref{define-translate-space-H-star} of $\mathcal{H}_\star(\delta)$.

It remains to analyze $\mathscr{S}(F_{32})$.
Thanks to the factorization \eqref{alpha-factorize} of $\alpha$,
the average over $m_\kspecial$ in $f_{31,\jspecial'}$
exhibits cancellation
for each $1\le \jspecial'<\lspecial'$,
since $\Cfinal_{\jspecial'}\ne 0$.
By the $$(u,N,C)
= \left(\mathfrak{u}_\kspecial,n[\jspecial'],
\frac{-\Cfinal_{\jspecial'}m_0\prod_{1\le \ispecial<\kspecial}m_\ispecial^{\mathfrak{u}_\ispecial}}
{n_0\prod_{1\le \jspecial\le \lspecial':\, \jspecial\ne \jspecial'} n[\jspecial]}
\right)$$
case of Lemma~\ref{gauss-bound} below, we have
\begin{equation}
\label{local-m-cancellation}
f_{31,\jspecial'}
= \EE[\phase(m_\kspecial^{\mathfrak{u}_\kspecial}CN^{-1}\bmod{\ZZ_N})]
\ll_\eps n[\jspecial']^{\eps-1/2},
\end{equation}
whenever $1\le \jspecial'<\lspecial'$.
On the other hand, we will satisfactorily bound $f_{30}$ by
simply plugging in Lemmas~\ref{LEM:denominator-bias}, \ref{LEM:numerator-bias}, and~\ref{LEM:generic-local-factor-estimate} pointwise
(just as in the proof of \eqref{INEQ:model-r,m_0,n_0-total-bound}
for Lemma~\ref{LEM:range-3}),
without attempting to obtain cancellation over $m_\kspecial$.

Let $\varsigma = -1$ if $f_{25} = \mathcal{J}_{\Omega, 0}$,
and $\varsigma = 1$ if $f_{25} = \mathcal{J}_{\Omega, \infty}$.
For $\Re(\bm{s})$ large, we have $\mathscr{S}(F_{32})(\bm{s}) = \mathscr{S}(F_{32})(\bm{s}-\varsigma\xi\mathbf{u})$ (cf.~the proof of Lemma~\ref{LEM:range-3}).
But for $\bm{s}$ in \eqref{INEQ:key-new-region},
the Dirichlet coefficient $$\mathfrak{C}=\mathfrak{C}(\bm{s},t,m_1,\dots,m_\kspecial,(n_\jspecial)_{\jspecial\ge 1:\, \cfinal_\jspecial=\cfinal_{\jspecial'}})$$
of $m_1^{\beta_1}\cdots m_\kspecial^{\beta_\kspecial}
\prod_{\jspecial\ge 1:\, \cfinal_\jspecial=\cfinal_{\lspecial'}}
n_\jspecial^{\gamma_\jspecial}$
in the series $F_{32}(\bm{s}-\varsigma\xi\mathbf{u}-it\mathbf{u})$
is of the form
\begin{equation*}
\sum_{r,m_0,n_0\ge 1}
\sum_{n_\jspecial\ge 1:\, \jspecial\ge 1,\; \cfinal_\jspecial\ne \cfinal_{\lspecial'}}
\frac{O(1+\norm{\bm{s}}^{O(1)}) \bm{1}_{\alpha^\varsigma \gg 1}}{\alpha^2 (1+t^2)}
\frac{f_{30}(\bm{s}-\varsigma\xi\mathbf{u}-it\mathbf{u},\lambda,r,\bm{m}',\bm{n})}
{(\alpha n_0/m_0)^{\varsigma\xi}}
\frac{O(\prod_{1\le \jspecial'<\lspecial'} n[\jspecial']^{\xi-1/2})}
{\prod_{\jspecial\ge 1:\, \cfinal_\jspecial\ne \cfinal_{\lspecial'}} n_\jspecial^{\gamma_\jspecial}},
\end{equation*}
by Lemma~\ref{LEM:uniform-derivative-bound-on-localized-modulated-integrals}
and \eqref{local-m-cancellation}.
We proceed to bound this via the strategy of Lemma~\ref{LEM:range-3}.

By \eqref{alpha-factorize} and \eqref{define-restricted-alpha-product},
$\alpha/\alpha_{\lspecial'}
= m_0/(n_0
\prod_{\jspecial\ge 1:\, \cfinal_\jspecial\ne \cfinal_{\lspecial'}}
n_\jspecial^{\mathfrak{v}_\jspecial})$.
Thus $$rm_0n_0
\prod_{\jspecial\ge 1:\, \cfinal_\jspecial\ne \cfinal_{\lspecial'}}
n_\jspecial^{\mathfrak{v}_\jspecial}
\ge \max(1, (\alpha/\alpha_{\lspecial'})^\varsigma)
\gg 1 + \alpha_{\lspecial'}^{-\varsigma},$$
since $\alpha^\varsigma \gg 1$.
Plugging $1\ll (rm_0n_0 \prod_{\jspecial\ge 1:\, \cfinal_\jspecial\ne \cfinal_{\lspecial'}} n_\jspecial^{\mathfrak{v}_\jspecial})^{2\xi} / (1 + \alpha_{\lspecial'}^{-\varsigma})^{2\xi}$
and the identities $\alpha^2 = (m_0/n_0)^2 (\alpha n_0/m_0)^2$
and $\alpha n_0/m_0 = \alpha_{\lspecial'}/\prod_{1\le \jspecial'<\lspecial'} n[\jspecial']$ into our bound for $\mathfrak{C}$ above,
and summing over $r$, $m_0$, $n_0$ as in the proof of \eqref{INEQ:model-r,m_0,n_0-total-bound}, we get
\begin{equation*}
\mathfrak{C}\ll_\xi
\sum_{n_\jspecial\ge 1:\, \jspecial\ge 1,\; \cfinal_\jspecial\ne \cfinal_{\lspecial'}}
\frac{O(1+\norm{\bm{s}}^{O(1)})}{\alpha_{\lspecial'}^{2+\varsigma\xi} (1+t^2)}
\frac{O(\prod_{1\le \jspecial'<\lspecial'} n[\jspecial']^{2+\varsigma\xi+\xi-1/2+2\xi})}
{\abs{\prod_{\jspecial\ge 1:\, \cfinal_\jspecial\ne \cfinal_{\lspecial'}} n_\jspecial^{\gamma_\jspecial}}}
\frac{1}{(1 + \alpha_{\lspecial'}^{-\varsigma})^{2\xi}}.
\end{equation*}
But $\abs{\prod_{\jspecial\ge 1:\, \cfinal_\jspecial\ne \cfinal_{\lspecial'}} n_\jspecial^{\gamma_\jspecial}}
\ge \prod_{\jspecial\ge 1:\, \cfinal_\jspecial\ne \cfinal_{\lspecial'}} n_\jspecial^{2\mathsf{u}_{\mathfrak{n}(\jspecial)}+1-\delta}
= (\prod_{1\le \jspecial'<\lspecial'} n[\jspecial'])^2
\prod_{\jspecial\ge 1:\, \cfinal_\jspecial\ne \cfinal_{\lspecial'}} n_\jspecial^{1-\delta}$, so
\begin{equation*}
\mathfrak{C}\ll_\xi
\sum_{n_\jspecial\ge 1:\, \jspecial\ge 1,\; \cfinal_\jspecial\ne \cfinal_{\lspecial'}}
\frac{O(1+\norm{\bm{s}}^{O(1)})}{\alpha_{\lspecial'}^{2+\varsigma\xi} (1+t^2)}
\frac{O(\prod_{1\le \jspecial'<\lspecial'} n[\jspecial']^{\varsigma\xi+\xi-1/2+2\xi})}
{\prod_{\jspecial\ge 1:\, \cfinal_\jspecial\ne \cfinal_{\lspecial'}} n_\jspecial^{1-\delta}}
\frac{1}{(1 + \alpha_{\lspecial'}^{-\varsigma})^{2\xi}}.
\end{equation*}
Since $\prod_{1\le \jspecial'<\lspecial'} n[\jspecial']
\ge \prod_{\jspecial\ge 1:\, \cfinal_\jspecial\ne \cfinal_{\lspecial'}} n_\jspecial$,
we find that if $\xi$ is small enough that $4\xi+\delta < 1/2$, then
\begin{equation*}
\mathfrak{C}\ll_\xi
\frac{1+\norm{\bm{s}}^{O(1)}}{\alpha_{\lspecial'}^{2+\varsigma\xi} (1+t^2)}
\frac{1}{(1 + \alpha_{\lspecial'}^{-\varsigma})^{2\xi}}
\asymp \frac{1+\norm{\bm{s}}^{O(1)}}
{\alpha_{\lspecial'}^2 (1+t^2) \alpha_{\lspecial'}^{\varsigma\xi}}
\cdot \frac{1}{1 + \alpha_{\lspecial'}^{-2\varsigma\xi}}
= \frac{1+\norm{\bm{s}}^{O(1)}}{\alpha_{\lspecial'}^2 (1+t^2)}
\cdot \frac{1}{\alpha_{\lspecial'}^\xi + \alpha_{\lspecial'}^{-\xi}}.
\end{equation*}
Integrating $\mathfrak{C}$ over $t\in \RR$
(and summing over $f_{25}$, $\lambda$, $\nu$)
gives the lemma,
with $$h_{\lspecial'}
\defeq \alpha_{\lspecial'}^2 (\alpha_{\lspecial'}^\xi + \alpha_{\lspecial'}^{-\xi})
\sum_{f_{25},\lambda,\nu} \int_{t\in \RR} \mathfrak{C}(\bm{s},t,m_1,\dots,m_\kspecial,(n_\jspecial)_{\jspecial\ge 1:\, \cfinal_\jspecial=\cfinal_{\jspecial'}})\, dt
\in \mathcal{H}_\star(\delta),$$
where we restrict $\sum_\nu$ to weights $\nu$ associated to the index $\jspecial'=\lspecial'$.
(In the partition of unity at the beginning of the proof, each weight $\nu$ is associated to some index $1\le \jspecial'\le \lspecial'$, corresponding approximately to which of the quantities $n[1],\dots,n[\lspecial']$ is largest.)
\end{proof}

The following standard lemma on Gauss sums
is similar to \cite{tanimoto2012height}*{Lemma~5.7}:
\begin{lemma}
\label{gauss-bound}
Let $u,N\in \ZZ_{\ge 1}$ and $C\in \ZZ[1/N]^\times$.
Then $$\frac{1}{\card{(\ZZ/N\ZZ)^\times}}\sum_{x\in (\ZZ/N\ZZ)^\times}
\phase(x^uCN^{-1}\bmod{\ZZ_N})
\ll_{u,\eps} N^{\eps-1/2},$$
where $\phase(x^uCN^{-1}\bmod{\ZZ_N})\defeq \phase(\tilde{x}^uCN^{-1}\bmod{\ZZ_N})$
for any lift $\tilde{x}\in \ZZ$ of $x$ modulo $N$.
\end{lemma}

\begin{proof}
By the Chinese remainder theorem and the divisor bound,
it suffices to treat the case where $N$ is a prime power.
Suppose $N=p^k$ with $k\ge 1$.
By the structure theory of finite abelian groups,
there exists a finite group $\mathfrak{H}$
of characters $\chi\maps (\ZZ/N\ZZ)^\times\to \CC^\times$
such that
\begin{equation}
\label{power-detector}
\#\{x\in (\ZZ/N\ZZ)^\times: \bm{1}_{x^u = y}\} = \sum_{\chi\in \mathfrak{H}} \chi(y)
\end{equation}
for all $y\in (\ZZ/N\ZZ)^\times$,
with $\card{\mathfrak{H}}
= \#\{x\in (\ZZ/N\ZZ)^\times: \bm{1}_{x^u = 1}\}
\ll_u 1$.
By \eqref{power-detector},
\begin{equation*}
\mathcal{Q}_6\defeq \sum_{x\in (\ZZ/N\ZZ)^\times}
\phase(x^uCN^{-1}\bmod{\ZZ_N})
= \sum_{y\in (\ZZ/N\ZZ)^\times}
\phase(yCN^{-1}\bmod{\ZZ_N})
\sum_{\chi\in \mathfrak{H}} \chi(y).
\end{equation*}
But by \cite{iwaniec2004analytic}*{Lemma~3.1},
we have $\abs{\sum_{y\in (\ZZ/N\ZZ)^\times}
\chi(y)\phase(yCN^{-1}\bmod{\ZZ_N})} \le \sqrt{N}$
(with equality if and only if $\chi$ is primitive
in the sense of \cite{iwaniec2004analytic}*{\S3.3}).
Therefore,
$\abs{\mathcal{Q}_6}
\le \card{\mathfrak{H}} \sqrt{N}
\ll_u \sqrt{N}$.
This suffices, since $\card{(\ZZ/N\ZZ)^\times} = (1-p^{-1}) N \ge N/2$.
\end{proof}






\section{Final reductions}
\label{SEC:final-reductions}

In this section, assume $X$ is strictly split but not that $D$ has strict normal crossings.
We proceed using the blowup strategy of \cite{chambert2002distribution}*{\S6} (see also \cite{tanimoto2012height}*{paragraph after Theorem~5.1}).
Let $\pi\maps \wt{X}\to X$ be an equivariant morphism of the form specified in Definition~\ref{DEFN:strictly-split-X}.
Then in particular, $\wt{X}$ is split and its boundary $\wt{D}\defeq \wt{X}\setminus G$ has strict normal crossings.


Index the irreducible components of $\wt{D}$ by $\wt{J}\contains J$
so that $\wt{D}_\jboundary$ is the strict transform of $D_\jboundary$ for $\jboundary\in J$,
and $\wt{D}_\jboundary$ is an exceptional divisor ($\cong \PP^1$, since $\wt{X}$ is split) for $\jboundary\in \wt{J}\setminus J$.
Then
\begin{equation*}
\Pic(\wt{X}) = \ZZ^{\wt{J}\setminus J} \oplus \pi^\ast{\Pic(X)}, \quad
\Pic^G(\wt{X}) = \ZZ^{\wt{J}} = \ZZ^{\wt{J}\setminus J} \oplus \pi^\ast{\Pic^G(X)},
\quad \div(a\vert_{\wt{X}}) = \pi^\ast{\div(a)}.
\end{equation*}
Thus we may choose $\wt{H}$ (satisfying Proposition~\ref{PROP:homomorphic-adelic-local-height-decomposition} for $\wt{X}$) so that $\wt{H}(\pi^\ast\bm{s},g) = H(\bm{s},g)$.

For numerical comparison of $X$ and $\wt{X}$,
the following standard lemma is essential:
\begin{lemma}
In $\Pic^G(\wt{X})$, we have $-\pi^\ast{\div(\omega)}
\in -\div(\pi^\ast{\omega})
+ \sum_{\jboundary\in \wt{J}\setminus J} \ZZ_{\ge 1} \wt{D}_\jboundary$.
\end{lemma}



We can now study $\mathsf{Z}(sK_X^{-1},1_G) = \wt{\mathsf{Z}}(s\pi^\ast{K_X^{-1}},1_G)$.
Let $w\in C^\infty_c(\RR)$.
Then the Mellin transform $w^\vee(s)\defeq \int_0^\infty w(x) x^{s-1}\, dx$ is holomorphic on $\Re(s)>0$, with rapid decay (i.e.~$w^\vee(s) \ll_L (1+\abs{s})^{-L}$ for all $L\ge 0$) in vertical strips.
Let $\sigma,B\ge 2$ be large.
By Mellin inversion,
\begin{equation}
\label{EQN:Mellin-inversion-for-Z}
\sum_{x\in G(\QQ)} w{\left(\frac{H(K_X^{-1},x)}{B}\right)}
= \frac{1}{2\pi} \int_{\Re(s) = \sigma} w^\vee(s) B^s \wt{\mathsf{Z}}(s\pi^\ast{K_X^{-1}},1_G)\, dt,
\quad\textnormal{where $t=\Im(s)$}.
\end{equation}

Contour shifting, via Proposition~\ref{PROP:Z_0-partial-canonical-main-terms} (for $\wt{X}$, with $\bm{\kappa} = \pi^\ast{\mathbf{d}}=-\pi^\ast{\div(\omega)}$), gives
\begin{equation}
\label{EQN:main-final-Z_0-result}
\frac{1}{2\pi} \int_{\Re(s) = \sigma} w^\vee(s) B^s \wt{\mathsf{Z}}_0(s\pi^\ast{\mathbf{d}},1_G)\, dt
= \left(\tilde{c}_{X,H} w^\vee(1) + O{\left(\frac{1}{\log{B}}\right)}\right)
\frac{B (\log{B})^{\card{J}-2}}{(\card{J}-2)!},
\end{equation}
where $\tilde{c}_{X,H} = \mathcal{X}_{\Lambda_J(\wt{X})}(\mathbf{d}) \lim_{\bm{s}\to \mathbf{d}}{H^\ast(\bm{s},1) \prod_{\jboundary\in J}(s_\jboundary-\mathsf{d}_\jboundary)}$.
By \eqref{EQN:define-X-Lambda_I-quotient-cone-function},
$\mathcal{X}_{\Lambda_J(\wt{X})} = \mathcal{X}_{\Lambda_J(X)}$.
Also, Lemma~\ref{LEM:snc-implies-nice-densities-implies-Tamagawa-extraction}(1)$\Rightarrow$(2) for $\wt{H}^\ast_v(\pi^\ast{\bm{s}},1)$ on $\wt{X}$, followed by Lemma~\ref{LEM:snc-implies-nice-densities-implies-Tamagawa-extraction}(2)$\Rightarrow$(3) for $X$, gives $\lim_{\bm{s}\to \mathbf{d}}{H^\ast(\bm{s},1) \prod_{\jboundary\in J}(s_\jboundary-\mathsf{d}_\jboundary)} = \tau(X,\mathsf{H})$.
Meanwhile, $\card{J}-2 = \rank(\Pic(X))-1$ by Proposition~\ref{PROP:compute-Picard-group-in-terms-of-boundary}.
Therefore, $\tilde{c}_{X,H}/(\card{J}-2)!$
equals Peyre's constant $\mathcal{A}_{X,\mathsf{H}}$ \eqref{EQN:Peyre-constant}.

Still, $\wt{\mathsf{Z}}_1$ remains.
Let $\bm{s} \defeq s(\mathbf{d}+\mathbf{u}) + \mathbf{u} = \mathbf{d}+2\mathbf{u} + (s-1)(\mathbf{d}+\mathbf{u})$;
note that $\mathbf{d}+\mathbf{u} \in \ZZ_{\ge 1}^J$ by Proposition~\ref{PROP:poles-of-left-invariant-top-form}.
Define $\mathfrak{m}$, $\mathfrak{n}$, $\kspecial$, $\lspecial$, $\cfinal_\jspecial$, $\lspecial'$, $\alpha_{\jspecial'}$ in terms of $X$ as in \S\ref{SEC:archimedean-endgame}, even if $D$ does not have strict normal crossings.
Let $\upsilon_\ispecial
\defeq \mathsf{d}_{\mathfrak{m}(\ispecial)}+\mathsf{u}_{\mathfrak{m}(\ispecial)}\ge 1$
for $\ispecial\in \set{1,\dots,\kspecial}$,
and $\nu_\jspecial\defeq \mathsf{d}_{\mathfrak{n}(\jspecial)}+\mathsf{u}_{\mathfrak{n}(\jspecial)}\ge 1$
for $\jspecial\in \set{1,\dots,\lspecial}$.
Let $\Psi_{\jspecial'}\defeq m_1^{\upsilon_1}\cdots m_\kspecial^{\upsilon_\kspecial}
\prod_{\jspecial\ge 1:\, \cfinal_\jspecial=\cfinal_{\jspecial'}} n_\jspecial^{\nu_j}$
and $\varrho_{\jspecial'} \defeq B / \Psi_{\jspecial'}$
for $\jspecial'\in \set{1,\dots,\lspecial'}$.
By Lemmas~\ref{LEM:range-3}, \ref{LEM:range-2}, and~\ref{LEM:range-1} for $\wt{X}$, we have (since $1 = \mathcal{P}_1 + \mathcal{P}_2 + \mathcal{P}_3$)
\begin{equation*}
\int_{\Re(s) = \sigma} w^\vee(s) B^{s-1} \wt{\mathsf{Z}}_1(\pi^\ast{\bm{s}},1_G)\, dt
\ll B^{-\delta}
+ \sum_{1\le \jspecial'\le \lspecial'}\, \sum_{m_1,\dots,m_\kspecial\ge 1}\,
\sum_{n_\jspecial\ge 1:\, \jspecial\ge 1,\; \cfinal_\jspecial=\cfinal_{\jspecial'}}\,
\frac{(\alpha_{\jspecial'}^\delta+\alpha_{\jspecial'}^{-\delta})^{-1} (\varrho_{\jspecial'}^{\delta}+\varrho_{\jspecial'}^{-\delta})^{-1}}
{m_1\cdots m_\kspecial
\prod_{\jspecial\ge 1:\, \cfinal_\jspecial=\cfinal_{\jspecial'}} n_\jspecial}
\end{equation*}
for some $\delta=\delta_X>0$,
where the factor of $(\varrho_{\jspecial'}^{\delta}+\varrho_{\jspecial'}^{-\delta})^{-1}$ comes from the bound
\begin{equation*}
\int_{\Re(s) = \sigma} w^\vee(s) \varrho_{\jspecial'}^{s-1}
h_{\jspecial'}(\bm{s},m_1,\dots,m_\kspecial,(n_\jspecial)_{\jspecial\ge 1:\, \cfinal_\jspecial=\cfinal_{\jspecial'}})\, dt
\ll (\varrho_{\jspecial'}^{\delta}+\varrho_{\jspecial'}^{-\delta})^{-1}
\end{equation*}
proven by shifting $\Re(s) = \sigma$ to $\Re(s) = 1 \pm \delta$.

Suppose $\jspecial'\in \set{1,\dots,\lspecial'}$.
Let $\cfinal_{\jspecial'}=q$;
then $\card{\set{\jspecial\ge 1: \cfinal_\jspecial=\cfinal_{\jspecial'}}}
= \card{J^q_1}$.
Let $A,P\in \RR_{>0}$.
The exponents $\upsilon_i$, $\nu_j$ in $\Psi_{\jspecial'}$ are positive,
and the exponent vectors of $\alpha_{\jspecial'}$, $\Psi_{\jspecial'}$ are linearly independent over $\RR$,
so the number of tuples of $M_\ispecial,N_\jspecial\in \set{2^e: e\in \ZZ_{\ge 0}}$
admitting $m_\ispecial\in [M_\ispecial, 2M_\ispecial)$,
$n_\jspecial\in [N_\jspecial, 2N_\jspecial)$
with $\alpha_{\jspecial'}\in [A, 2A)$,
$\varrho_{\jspecial'}\in [P, 2P)$
is $\ll (\log(2 + B/P))^{\kspecial+\card{J^q_1}-2}$.
Thus
\begin{equation*}
\sum_{m_1,\dots,m_\kspecial\ge 1}\,
\sum_{n_\jspecial\ge 1:\, \jspecial\ge 1,\; \cfinal_\jspecial=\cfinal_{\jspecial'}}\,
\frac{(\alpha_{\jspecial'}^\delta+\alpha_{\jspecial'}^{-\delta})^{-1}
(\varrho_{\jspecial'}^{\delta}+\varrho_{\jspecial'}^{-\delta})^{-1}}
{m_1\cdots m_\kspecial
\prod_{\jspecial\ge 1:\, \cfinal_\jspecial=\cfinal_{\jspecial'}} n_\jspecial}
\ll \sum_{A,P\in \set{2^e: e\in \ZZ}}
\frac{(\log(2 + B/P))^{\kspecial+\card{J^q_1}-2}}
{(A^\delta+A^{-\delta}) (P^\delta+P^{-\delta})},
\end{equation*}
which is $\ll (\log{B})^{\kspecial+\card{J^q_1}-2}$.
But $\kspecial = \card{J^\ast_2}$,
so $\kspecial+\card{J^q_1} \le \rank(\Pic(X))$
by Proposition~\ref{PROP:critical-index-upper-bound}.
Summing over $1\le \jspecial'\le \lspecial'$,
and using the $\CC\pi^\ast{\mathbf{u}}$-invariance of $\wt{\mathsf{Z}}_1(\bm{z},1_G)$ for $\Re(\bm{z})$ large
(which follows from \eqref{EQN:difficult-part-of-spectral-expansion}, \eqref{EQN:change-of-t-identity} for $\wt{X}$),
we get \eqref{EQN:log-saving-Manin-Peyre-conjecture-for-standard-anticanonical-Weil-heights}, in view of \eqref{EQN:Mellin-inversion-for-Z}, \eqref{EQN:main-final-Z_0-result}.

The statement \eqref{sharp-cutoff} follows from
\eqref{EQN:log-saving-Manin-Peyre-conjecture-for-standard-anticanonical-Weil-heights}
and the fact that any continuous function on a compact interval can be approximated from above and below by functions in $C^\infty_c(\RR)$.

It would be interesting to know whether ergodic methods
(cf.~\cites{gorodnik2008manin,
ELMV2009,Zamojski2010,ELMV2011,GorodnikOh2011,ShapiraZheng2024})
could prove Theorem~\ref{THM:main-Manin-application},
either quantitatively or up to $o(1)$.
Another natural question is whether universal torsors
(cf.~\cite{salberger1998tamagawa})
could be used in this setting.

\section{Examples}
\label{SEC:examples}

Classification of $X$ is open.
Some examples can be found in \cite{tanimoto2012height}*{\S1};
see also \cite{derenthal2015equivariant}, which is however written in terms of left actions rather than right actions, which are equivalent via the map $g\mapsto g^{-1}$.
One can obtain further examples
by blowing up along closed $G$-invariant subschemes,
or by taking orbit closures
in projective $G$-representations
(cf.~\cite{shalika2015height}*{first two paragraphs of \S1})
or in tensor products thereof.
Such operations might not keep $X$ split.
However, we can verify the hypotheses of Theorem~\ref{THM:main-Manin-application}
in each example below.

We focus on constructions where interesting behavior with special divisors may occur.

\begin{example}
\label{EX:split-quadric-Hirzebruch}
Embed $G$ in $X \defeq \PP^1 \times \PP^1$ by $(a,b)\mapsto ([1:a],[a:b])$.
The left $G$-action on $G$ does not extend to $X$.
The right $G$-action does:
$$([x_0:x_1],[t_0:t_1])(\ucord,\vcord)
\defeq ([x_0:x_1\ucord],[t_0\ucord:t_0\vcord+t_1]).$$
Here $a = \frac{x_1}{x_0}$, $b = \frac{x_1t_1}{x_0t_0}$,
and $D = \{x_1=0\} \cup \{x_0=0\} \cup \{t_0=0\}$.
Thus $X$ is split,
and $D$ has strict normal crossings.
Moreover, $\card{J}=3$ and
$\lspecial=\card{J^0_1}
=\kspecial=\card{J^\ast_2}
=\card{J_3}=1$.
\end{example}

\begin{example}
\label{EX:A3A1quartic}
Via universal torsors,
Manin's conjecture has been resolved in essentially full generality
for singular quartic del Pezzo surfaces of singularity type
$\mathbf{A}_3+\mathbf{A}_1$;
papers on the topic include \cites{derenthal2009counting,
Bourqui2013,DerenthalFrei2014,FreiPieropan2016,derenthal2020split}.
Let $S\belongs \PP^4_\QQ$
be the split form $$x_0^2+x_0x_3+x_2x_4 = x_1x_3-x_2^2 = 0.$$
Then $S$ admits a unique structure as an equivariant compactification of $G$
\cite{derenthal2015equivariant}*{Lemma~15},
but is neither toric nor a vector group compactification
\cite{DerenthalLoughran2010}*{\S5},
nor a bi-equivariant compactification of $G$
\cite{tanimoto2012height}*{Proposition~1.6}.
To realize this structure,
embed $G$ in $S$ by $(a,b)\mapsto [1-b:a^2:a:1:-(1-b)(2-b)/a]$,
with right $G$-action
\begin{equation*}
[x_0:x_1:x_2:x_3:x_4]
(\ucord,\vcord)
\defeq [x_0-x_2\vcord:
x_1\ucord^2:x_2\ucord:x_3:
(x_4+(2x_0+x_3)\vcord-x_2\vcord^2)/\ucord].
\end{equation*}
Here $a=\frac{x_1}{x_2}=\frac{x_2}{x_3}$
and $b=\frac{x_3-x_0}{x_3}$.
Moreover, $1-b = \frac{x_0}{x_3}$
and $b-2=\frac{-x_0-x_3}{x_3}=\frac{x_2x_4}{x_0x_3}=\frac{ax_4}{x_0}$.


The singular points of $S$ are
$\mathbf{A}_3 = [0:0:0:0:1]$
and $\mathbf{A}_1 = [0:1:0:0:0]$.
By \cite{derenthal2009counting}*{\S8},
the lines on $S$ are $\ell_1 = \{x_0=x_1=x_2=0\}$,
$\ell_2 = \{x_0+x_3=x_1=x_2=0\}$,
and $\ell_3 = \{x_0=x_2=x_3=0\}$.
Here $\mathbf{A}_3\in \ell_1\cap \ell_2\cap \ell_3$
and $\mathbf{A}_1\in \ell_3$.
We note that set-theoretically, $\ell_1\cup \ell_2\cup \ell_3 = \{x_2=0\}\cap S$
and $\{x_2\ne 0\}\cap S = \{x_1x_2x_3\ne 0\}\cap S = G$.

The surface $S$ is normal.
Since $x_2$ cuts out $\ell_1$ on the open set $\{(x_0+x_3)x_3\ne 0\}$,
and $x_0$ cuts out $\ell_1$ on the open set $\{x_3x_4\ne 0\}$,
we have $\ord_{\ell_1}(a) = \ord_{\ell_1}(b-1) = 1$.
Similarly, $x_2$ cuts out $\ell_2$ on $\{x_0x_3\ne 0\}$,
and $x_0+x_3$ cuts out $\ell_2$ on $\{x_3x_4\ne 0\}$,
so $\ord_{\ell_2}(a) = \ord_{\ell_2}(b-2) = 1$.
Therefore, $J^1_1\ne \emptyset$
and $J^2_1\ne \emptyset$.
It turns out that $J^\ast_2\ne \emptyset$ as well, as we will see below.

Let $\pi_1\maps S'\to S$ be the blowup of $S$ along $\mathbf{A}_3\cup \mathbf{A}_1$.
Then $\pi_1^{-1}(\mathbf{A}_1)=\ell_0\cong \PP^1_\QQ$, say.
%
%
%
%
Also, $\pi_1^{-1}(\mathbf{A}_3)=\ell_4\cup \ell_5
\cong \PP^1_\QQ\cup_{[1:0:0]} \PP^1_\QQ\belongs \PP^2_\QQ$, say.
%
%
%
%
Let $\pi_2\maps S''\to S'$ be the blowup of $S'$
along the nodal singularity $\ell_4\cap \ell_5\in \pi_1^{-1}(\mathbf{A}_3)$.
Then $\pi_2^{-1}(\ell_4\cap \ell_5) = \ell_6\cong \PP^1_\QQ$, say.
Moreover, $S''$ is smooth.
This can all be checked explicitly, as we now describe.

Let $\ell'_1,\ell'_2,\ell'_3$ be the strict transforms of $\ell_1,\ell_2,\ell_3$ under $\pi_1$.
Since $\ell_1,\ell_2,\ell_3\in \PP^4_\QQ$ are distinct lines,
$\ell'_1,\ell'_2,\ell'_3$ are pairwise disjoint.
Also, $\ell'_1,\ell'_2$ are disjoint from $\ell_0$,
since their images under $\pi_1$ are disjoint.
On the other hand, $\ell'_3\belongs \{x_0=x_2=0\}$ intersects $\ell_0$
in the affine chart
\begin{equation}
\label{A1-blowup-chart}
\{((y_0x_4,y_2x_4,x_4),[y_0:y_2:1])\in \Aff^3\times \PP^2:
y_0^2+y_0y_2^2x_4+y_2=0\}\belongs S'
\end{equation}
(given by $x_1=y_4=1$ and $x_3=x_2^2$)
at the scheme $\{y_0=y_2=0\} \cap \{x_4=y_0^2+y_2=0\}$,
which is a reduced point.
Similarly, $\ell'_1,\ell'_2,\ell'_3$ intersect $\ell_4\cup \ell_5$
in the affine chart $$\{((y_0z,y_1z,(1-y_1)z),[y_0:y_1:1-y_1])\in \Aff^3\times \PP^2:
y_1(1-y_1)-(y_0+1-y_1)^2y_0^2z^2=0\}\belongs S'$$
(given by $x_4=y_1+y_3=1$, $x_2=-x_0^2-x_0x_3$, and $x_1+x_3=z$)
at the schemes
\begin{equation*}
\begin{split}
\{y_0=y_1=0\} \cap \{z=y_1(1-y_1)=0\}
&= \{y_0=y_1=z=0\} \\
\{y_0=-1,\; y_1=0\} \cap \{z=y_1(1-y_1)=0\}
&= \{y_0=-1,\; y_1=z=0\}\\
\{y_0=0,\; y_1=1\} \cap \{z=y_1(1-y_1)=0\}
&= \{y_0=z=0,\; y_1=1\},
\end{split}
\end{equation*}
respectively.
By symmetry, we may identify $\ell_4$ and $\ell_5$ with the closures in $S'$
of $\{z=y_1=0\}$ and $\{z=1-y_1=0\}$, respectively.
Then $\ell'_1,\ell'_2$ intersect $\ell_4\setminus \ell_5$
at the reduced points $(\bm{0},[0:0:1]),(\bm{0},[-1:0:1])$, respectively,
and $\ell'_3$ intersects $\ell_5\setminus \ell_4$
at the reduced point $(\bm{0},[0:1:0])$.
We can also check that $\ell_4,\ell_5$ intersect
in the affine chart $$\{((x_0,y_1x_0,y_3x_0),[1:y_1:y_3])\in \Aff^3\times \PP^2:
y_1y_3-(1+y_3)^2x_0^2=0\}\belongs S'$$
(given by $x_4=y_0=1$ and $x_2=-x_0^2-x_0x_3$)
at the reduced point $(\bm{0},[1:0:0])$.

Let $\ell'_0,\ell''_1,\ell''_2,\ell''_3,\ell'_4,\ell'_5$ be the strict transforms of $\ell'_0,\ell'_1,\ell'_2,\ell'_3,\ell_4,\ell_5$ under $\pi_2$.
Then $\ell'_0,\ell''_1,\ell''_2,\ell''_3$ are disjoint from $\ell_6$.
Also, $\ell'_4\cap \ell_6$ and $\ell'_5\cap \ell_6$ are distinct, reduced $\QQ$-points.
We conclude that $X\defeq S''$ is split,
and that $D = \ell'_0\cup \ell''_1\cup \ell''_2\cup \ell''_3\cup \ell'_4\cup \ell'_5\cup \ell_6$ has strict normal crossings.
So Theorem~\ref{THM:main-Manin-application} applies to $X$.
But in the chart \eqref{A1-blowup-chart},
we have $y_0^2 + y_2(1+y_0y_2x_4) = 0$,
so $$a = \frac{x_1}{x_2} = \frac{1}{y_2x_4} = -\frac{1+y_0y_2x_4}{y_0^2x_4},
\quad b = \frac{x_3-x_0}{x_3} = 1-\frac{y_0x_4}{(y_2x_4)^2}
= 1-\frac{(1+y_0y_2x_4)^2}{y_0^3x_4},$$
where $y_0$ cuts out $\ell'_3$ and $x_4$ cuts out $\ell_0$.
In particular, $\ord_{\ell'_0}(a)=\ord_{\ell'_0}(b)=-1$,
so $J^\ast_2\ne \emptyset$.
Thus \eqref{INEQ:mysterious-numerical-condition} fails,
and one cannot apply \cite{tanimoto2012height}*{Theorem~5.1} to $X$.
\end{example}

\begin{remark}

Theorem~\ref{THM:main-Manin-application}
does not give any new examples of Manin's conjecture for
\emph{generalized del Pezzo surfaces}
in the sense of \cite{derenthal2015equivariant}*{\S4}.
%
Nonetheless, Theorem~\ref{THM:main-Manin-application}
is in general new for other rational surfaces.
Our $X$ generally satisfies
$\kspecial\lspecial>0$,
and is generally
not del Pezzo,
not bi-equivariant,
not toric,
and not a vector group compactification.
To make $\rank(\Pic(X))$ arbitrarily large, one could repeatedly apply blowups, centered at $G$-invariant $\QQ$-points, to $X$ in Example~\ref{EX:A3A1quartic}.
This keeps $X$ split, because the tangent line to a smooth $\QQ$-curve at a $\QQ$-point is defined over $\QQ$.
The resulting cover would suffice by \cite{tanimoto2012height}*{Proposition~1.3}.

\end{remark}


We end by discussing an explicit projective representation
related to \cite{chambert2002distribution}.

\begin{example}
\label{EX:orbit-closure}



Let $G$ act on $\PP(\set{q\in \QQ[x]: \deg{q}\le n}) \cong \PP^n$
by $q(x) \mapsto q(ax+b)$.
Suppose $n\ge 3$,
and choose a set $R\in \binom{\QQ}{n}$
on which $G(\QQ)$ acts faithfully.
The closure $C$ of the $G$-orbit of $[\prod_{\varrho\in R} (x-\varrho)]\in \PP^n$ is a singular equivariant compactification of $G$.
Its boundary includes
the twisted curve $[(Ax+B)^n]$ via $(a,b) = (\lambda A, \lambda B)$ with $\lambda\to \infty$,
and the line $[Ax+B]$ via $(a,b) = (\eps A, \varrho+\eps B)$ with $\varrho\in R$, $\eps\to 0$.
These are the only boundary components, since $C$ is the image of the map
$X\defeq \map{Bl}_{\set{[0:\rho:1]:\, \rho\in R}}(\PP^2) \to \PP^n$
resolving the rational map
\begin{equation*}
\PP^2\ni [a:b:c] \mapsto [{\textstyle\prod_{\rho\in R} (ax+b-c\rho)}] \in \PP^n.
\end{equation*}
The $G$-action $[a:b:c](\ucord,\vcord)\defeq [a\ucord:a\vcord+b:c]$,
being trivial on the line $\{a=0\}$,
makes $X\to C$ an equivariant resolution of singularities.
Here $X$ is split,
$D = \{c=0\}\cup \{a=0\}$ has strict normal crossings,
$\card{J}=n+2$,
$\kspecial=\card{J^\ast_2}=1$,
$\card{J^\varrho_1}=1$ (for $\varrho\in R$),
and $\lspecial=\card{R}=n$.
Incidentally, one can also realize $X$ as a vector group compactification,
so that Manin's conjecture for $X$ is given by \cite{chambert2002distribution}
(or \cite{chambert2000pointsI}).
Theorem~\ref{THM:main-Manin-application} gives a new proof.



This $X$ is reminiscent of
\cite{duke1993density}*{Example~1.8} on $\SL_2(\ZZ)$-orbits of binary forms,
whose geometry is explained in \cites{hassett2003integral,tanimoto2015distribution}.
But our groups, height functions, and points differ.
It would be interesting to find a common generalization of \cite{tanimoto2015distribution} and our work.
\end{example}





\section*{Acknowledgements}

I thank Yuri Tschinkel for introducing me to the beautiful paper \cite{tanimoto2012height} and associated open questions,
and thank him as well as Ramin Takloo-Bighash and Sho Tanimoto for their encouragement and comments.
Also, I thank Tim Browning and Dan Loughran for comments and suggestions concerning
Manin--Peyre,
homogeneous spaces, and splitness.
Thanks also to Anshul Adve, Peter Sarnak, Philip Tosteson, Katy Woo, and Nina Zubrilina for some interesting discussions.
I thank the Browning Group and Andy O'Desky for many conversations.
This project has received funding from the European Union's Horizon~2020 research and innovation program under the Marie Sk\l{}odowska-Curie Grant Agreement No.~101034413.
Finally, I thank the editors and referees for their detailed input,
which substantially improved the paper.

\bibliographystyle{amsxport}
\bibliography{final.bib}



\end{document}